\documentclass[12pt]{article}
\pdfoutput=1

\usepackage{amsmath, amsthm, amssymb}
\usepackage{fullpage}
\usepackage{enumerate}
\usepackage{ifpdf}
\usepackage{mathrsfs}
\usepackage[normalem]{ulem}
\ifpdf
\usepackage[pdftex]{graphicx}
\else
\usepackage[dvips]{graphicx}
\fi
\usepackage{tikz}
\usepackage[all]{xy}
\usetikzlibrary{arrows,chains,matrix,positioning,scopes}

\makeatletter
\tikzset{join/.code=\tikzset{after node path={%
\ifx\tikzchainprevious\pgfutil@empty\else(\tikzchainprevious)%
edge[every join]#1(\tikzchaincurrent)\fi}}}
\makeatother
\tikzset{>=stealth',every on chain/.append style={join},
         every join/.style={->}}
\tikzstyle{labeled}=[execute at begin node=$\scriptstyle,
   execute at end node=$]

\input xy
\xyoption{all}

\usepackage[pdftex,plainpages=false,hypertexnames=false,pdfpagelabels]{hyperref}
\newcommand{\arXiv}[1]{\href{http://arxiv.org/abs/#1}{\tt arXiv:\nolinkurl{#1}}}
\newcommand{\arxiv}[1]{\href{http://arxiv.org/abs/#1}{\tt arXiv:\nolinkurl{#1}}}

\newcommand{\doi}[1]{\href{http://dx.doi.org/#1}{{\tt DOI:#1}}}

\newcommand{\googlebooks}[1]{(preview at \href{http://books.google.com/books?id=#1}{google books})}

\newcommand{\hookdownarrow}{\mathrel{\rotatebox[origin=c]{-90}{$\hookrightarrow$}}}

\usepackage{xcolor}
\definecolor{dark-red}{rgb}{0.7,0.25,0.25}
\definecolor{dark-blue}{rgb}{0.15,0.15,0.55}
\definecolor{medium-blue}{rgb}{0,0,.8}
\definecolor{DarkGreen}{RGB}{0,150,0}
\hypersetup{
   colorlinks, linkcolor={purple},
   citecolor={medium-blue}, urlcolor={medium-blue}
}
\definecolor{lightred}{RGB}{255,99,99}
\definecolor{lightblue}{RGB}{99,99,255}

\usepackage{relsize}
	\newcommand{\bigast}{\mathop{\mathlarger{\mathlarger{\Asterisk}}}}

\theoremstyle{plain}
\newtheorem{thm}{Theorem}[section]
\newtheorem*{thm*}{Theorem}
\newtheorem{alphathm}{Theorem}

\newtheorem{cor}[thm]{Corollary}
\newtheorem*{cor*}{Corollary}
\newtheorem{lem}[thm]{Lemma}

\newtheorem{prop}[thm]{Proposition}

\newtheorem*{quest*}{Question}

\theoremstyle{definition}
\newtheorem{defn}[thm]{Definition}

\newtheorem{rem}[thm]{Remark}

\DeclareMathOperator{\id}{id}

\DeclareMathOperator{\op}{op}

\DeclareMathOperator{\Tr}{Tr}
\DeclareMathOperator{\tr}{tr}

\newcommand{\comment}[1]{}

\newcommand{\be}{\begin{enumerate}[(1)]}
\newcommand{\ee}{\end{enumerate}}

\newcommand{\N}{\mathbb{N}}
\newcommand{\Z}{\mathbb{Z}}
\newcommand{\Q}{\mathbb{Q}}
\newcommand{\F}{\mathbb{F}}
\newcommand{\R}{\mathbb{R}}

\newcommand{\C}{\mathbb{C}}

\newcommand{\noshow}[1]{}
\newcommand{\MR}[1]{}

\newcommand{\Asterisk}{\mathop{\scalebox{1.5}{\raisebox{-0.2ex}{$*$}}}}%

\newcommand{\vphi}{\varphi}

\newcommand{\<}{\left\langle}
\renewcommand{\>}{\right\rangle}

\def\semicolon{;}
\def\applytolist#1{
    \expandafter\def\csname multi#1\endcsname##1{
        \def\multiack{##1}\ifx\multiack\semicolon
            \def\next{\relax}
        \else
            \csname #1\endcsname{##1}
            \def\next{\csname multi#1\endcsname}
        \fi
        \next}
    \csname multi#1\endcsname}

\def\calc#1{\expandafter\def\csname c#1\endcsname{{\mathcal #1}}}
\applytolist{calc}QWERTYUIOPLKJHGFDSAZXCVBNM;
\def\bbc#1{\expandafter\def\csname bb#1\endcsname{{\mathbb #1}}}
\applytolist{bbc}QWERTYUIOPLKJHGFDSAZXCVBNM;
\def\bfc#1{\expandafter\def\csname bf#1\endcsname{{\mathbf #1}}}
\applytolist{bfc}QWERTYUIOPLKJHGFDSAZXCVBNM;
\def\sfc#1{\expandafter\def\csname s#1\endcsname{{\sf #1}}}
\applytolist{sfc}QWERTYUIOPLKJHGFDSAZXCVBNM;
\def\ffc#1{\expandafter\def\csname f#1\endcsname{{\mathfrak #1}}}
\applytolist{ffc}QWERTYUIOPLKJHGFDSAZXCVBNM;

\usepackage{yfonts}

\tikzstyle{shaded}=[fill=red!10!blue!20!gray!30!white]
\tikzstyle{unshaded}=[fill=white]
\tikzstyle{empty box}=[circle, draw, thick, fill=white, opaque, inner sep=2mm]
\tikzstyle{annular}=[scale=.7, inner sep=1mm, baseline]
\tikzstyle{rectangular}=[scale=.75, inner sep=1mm, baseline=-.1cm]
\usetikzlibrary{calc}

\begin{document}

\title{Free products of finite-dimensional and other von Neumann algebras in terms of free Araki--Woods factors}
\author{
	Michael Hartglass\thanks{Department of Mathematics and Computer Science, Santa Clara University} \\ \tt{\footnotesize mhartglass@scu.edu}
	\and
	Brent Nelson\thanks{Department of Mathematics, Michigan State University} \\ \tt{\footnotesize brent@math.msu.edu}
}\date{}
\maketitle

\begin{abstract}
\noindent
We show that any free product of finite-dimensional von Neumann algebras equipped with non-tracial states is isomorphic to a free Araki--Woods factor with its free quasi-free state possibly direct sum a finite-dimensional von Neumann algebra. This gives a complete answer to questions posed by Dykema in \cite{Dyk97} and Shlyakhtenko in \cite{Shl97}, which had been partially answered by Houdayer in \cite{Hou07} and Ueda in \cite{MR3483468}. We also extend this to suitable infinite-dimensional von Neumann algebras with almost periodic states.
\end{abstract}


\section*{Introduction}

Since the advent of free probability by Voiculescu, there has been significant interest in studying free products of von Neumann algebras (see \cite{MR1201693, MR1363079, MR1372534,Dyk97, Shl97, MR1738186, MR1813523, Hou07, MR2838053, MR2875863, MR3164718, MR3092256, MR3483468} among others).  A landmark result of Dykema expressed free products of finite-dimensional von Neumann algebras with tracial states in terms of interpolated free group factors.  Specifically, he proved:

\begin{thm*}[{\cite[Theorem 4.6]{MR1201693}}]
Let $A$ and $B$ be hyperfinite von Neumann algebras equipped with faithful normal tracial states $\phi$ and $\psi$, respectively.  Assume that $\dim(A) \geq 2$, $\dim(B) \geq 2$, and $\dim(A) + \dim(B) \geq 5$.  Then
	\[
		(A, \phi) * (B, \psi) \cong (L(\F_{t}),\tau) \oplus C
	\]
where $C$ is finite-dimensional (possibly zero), and $t$ can be computed directly from $(A, \phi)$ and $(B, \psi)$ using ``free dimension."
\end{thm*}

\noindent
This theorem therefore established that free group factors were the ``minimal" II$_{1}$ factors that could appear in a free product of two von Neumann algebras with tracial states.

Later, R\u{a}dulescu \cite{MR1372534} studied the free product $(L(\Z),\tau) * (M_{2}(\C), \psi)$ for $\psi$ a nontracial state, and showed that it is a type $\mathrm{III}_{\lambda}$ factor with centralizer $L(\F_{\infty})$ and discrete core $L(\bbF_\infty)\otimes\cB(\cH)$.  Dykema \cite{Dyk97} and Ueda \cite{MR3483468} (using different techniques)  extended R\u{a}dulescu's result with the following theorem.

\begin{thm*}[{\cite[Theorem 3]{Dyk97} and \cite[Theorem 3]{MR3483468}}]\label{thm:Dykema}
Let $A$ and $B$ be two separable von Neumann algebras equipped with faithful normal states $\phi$ and $\psi$ respectively, at least one of which is not a trace.  Assume that $A$ and $B$ are not one-dimensional, and are countable direct sums of the following
\begin{enumerate}

\item[\normalfont{(1)}] Type I factors

\item[\normalfont{(2)}] Diffuse hyperfinite von Neumann algebras on which $\phi$ or $\psi$ is a trace

\end{enumerate}
Let $(\cM, \vphi) = (A, \phi) * (B, \psi)$.  Then $\cM = \cM_{0} \oplus C$ where $C$ is finite-dimensional (possibly zero) and, $\cM_{0}$ is a type $\mathrm{III}$ factor.  Defining $\vphi_{0} = \vphi|_{\cM_{0}}$, we have $\cM_{0}^{\vphi_{0}} \cong L(\F_{\infty})$, $\vphi_{0}$ is almost periodic, and the point spectrum of $\Delta_{\vphi}$ (the modular operator of $\vphi$) is the group generated by the point spectra of $\Delta_{\phi}$ and $\Delta_{\psi}$.

\end{thm*}

Notably, the above theorem does not address how to determine when two different $\cM_{0}$'s are isomorphic.  Shlyakhtenko created a natural candidate for $\cM_{0}$ when he constructed the (almost periodic) \emph{free Araki--Woods factors} $(T_{H}, \vphi_{H})$, indexed by countable, non-trivial subgroups $H<\bbR^+$, and equipped with faithful normal states called \emph{free quasi-free states} \cite{Shl97}.  Shlyakhtenko showed that for countable, non-trivial $H<\bbR^+$, $(T_{H}, \vphi_{H})$ is a factor of type III$_{\lambda}$ for $\lambda \in (0, 1]$, and $\lambda \neq 1$ if and only if $H = \langle \lambda \rangle$.  In addition, it was shown that the point spectrum of $\Delta_{\vphi_{H}}$ is exactly $H$, the family $\{(T_{H}, \vphi_{H})\colon H<\R^+\}$  is closed under taking free products, $H$ uniquely determines $(T_{H}, \vphi_{H})$, $(T_{H})^{\vphi_{H}} \cong L(\F_{\infty})$, and $T_H$ has discrete core isomorphic to $L(\bbF_\infty)\otimes \cB(\cH)$.  Furthermore, Shlyakhtenko showed that the factor of R\u{a}dulescu \cite{MR1372534} is isomorphic to $(T_{\lambda}, \vphi_{\lambda})$ for some $\lambda \in (0, 1)$.
 
These von Neumann algebras arise from a Fock space construction; namely, a modification of Voiculescu's free Gaussian functor \cite{MR1217253}. In particular, the free quasi-free state is given by the vacuum state and for $H=\<1\>$ the construction yields a free group factor. It is thus natural to assert that $(T_{H}, \vphi_{H})$ are the non-tracial analogues of the free group factors. This led to the following question posed by Shlyakhtenko:

\begin{quest*}[\cite{Shl97}]
Suppose that $(A, \phi)$, $(B, \psi)$, and $(\cM_0,\vphi_0)$ are as in the above theorem. Must $(\cM_{0}, \varphi_{0}) \cong (T_{H}, \vphi_{H})$ where $H$ is the subgroup of $\R^{+}$ generated by the point spectra of $\Delta_{\phi}$ and $\Delta_{\psi}$?

\end{quest*}

A partial answer to the above question was obtained by Houdayer \cite{Hou07} who identified $(M_{2}(\C), \phi) *  (M_{2}(\C), \psi)$ with an almost periodic free Araki--Woods factor provided that at least one of $\phi$ or $\psi$ is not a trace. An essential step in his proof was to identify $(M_{2}(\C),\phi)* [\C\oplus \C]$ with an almost periodic free Araki--Woods factor (see Lemma \ref{lem:Hou} below). However, his methods required one to assume that the mass in $\C\oplus \C$ was not overly concentrated in a single summand, and this hypothesis prevents one from applying his methods to free products of higher dimensional matrix algebras.

%
%
%
%
%

More evidence of a positive answer to Shlyakhtenko's question was obtained by Ueda, who showed in \cite{MR3483468} that the discrete core of $\cM_{0}$ as above is isomorphic to $L(\F_{\infty}) \otimes \cB(\cH)$. 

Despite these breakthroughs, it was still unknown how to identify the following free products with almost periodic free Araki--Woods factors (see Section~\ref{subsec:status_quo} for an explanation of the notation):
\begin{itemize}

\item $\underset{\alpha, 1 - \alpha}{M_{2}(\C)} * \left[\underset{\beta}{\C} \oplus \underset{1-\beta}{\C}\right]$ where $\frac12 < \alpha < \beta  < 1$;

\item $(M_{n}(\C), \phi) * \left[\underset{\beta}{\C} \oplus \underset{1-\beta}{\C}\right]$ for $\phi$ non-tracial and $n \geq 3$;

\item $(M_{n}(\C), \phi) * (M_{m}(\C), \psi)$ for arbitrary $n$ and $m$ and at least one of $\phi$ or $\psi$ non-tracial.

\end{itemize}

In this paper we answer Shlyakhtenko's question in the affirmative, namely we prove the following theorem (see Theorem \ref{thm:finitedim}).

\begin{alphathm}\label{thm:finitedim2}
Let $(A, \phi)$ and $(B, \psi)$ be two finite-dimensional von Neumann algebras with faithful states $\phi$ and $\psi$ respectively, both of which are at least two-dimensional.  Assume that at least one of $\phi$ or $\psi$ is not a trace, and that up to unitary conjugation,
$$
(A, \phi) = \bigoplus_{i = 1}^{n} \underset{\alpha_{i,1},\cdots, \alpha_{i, k_{i}}}{\overset{p_{i,1}, \cdots, p_{i, k_{i}}}{M_{k_{i}}(\C)}} \qquad\text{and}\qquad(B, \psi) = \bigoplus_{j = 1}^{m} \underset{\beta_{j,1},\cdots, \beta_{j, \ell_{j}}}{\overset{q_{j,1}, \cdots, q_{j, \ell_{j}}}{M_{\ell_{j}}(\C)}}.
$$
Let $H$ be the group generated by the point spectra of $\Delta_{\phi}$ and $\Delta_{\psi}$.  Then 
$$
(A, \phi) * (B, \psi) \cong (T_{H}, \vphi_{H}) \oplus C
$$
where $C$ is finite-dimensional (possibly zero) and can be determined explicitly from $(A, \phi)$ and $(B, \psi)$.  See Theorem \ref{thm:finitedim} below. 

\end{alphathm}

This theorem is extended to the hyperfinite case in the following Theorem (see Theorem \ref{thm:hyperfinite}).

\begin{alphathm}\label{thm:hyperfinite2}
Let $(A, \phi)$ and $(B, \psi)$ be von Neumann algebras with normal faithful states $\phi$ and $\psi$ with at least one of $\phi$ or $\psi$ not a trace.  Assume further that $\dim(A), \dim(B) \geq 2$ and $A$ and $B$  are countable direct sums of algebras of the following types:
\begin{itemize}

\item separable type I factors with faithful normal states;

\item diffuse von Neumann algebras of the form $\displaystyle \overline{\bigotimes_{n=1}^{\infty} (F_{i}, \phi_{i}) }$ where each $F_{i}$ is finite-dimensional and the state is the tensor product of the $\phi_{i}$;  

\item $(M, \gamma) \otimes (L(\F_{t}),\tau)$ with $M$ a separable type I factor, finite or infinite-dimensional;

\item $(N, \gamma') \otimes (T_{G}, \vphi_{G})$ with $N$ a separable type I factor, finite or infinite-dimensional, and $G$ a countable, non-trivial subgroup of $\R_{+}$.

\end{itemize}
\noindent
Let $(\cM, \vphi) = (A, \phi) * (B, \psi)$.  Then $(\cM, \vphi) \cong (T_{H}, \vphi_{H}) \oplus C$ where $H$ is the group generated by the point spectra of $\Delta_{\phi}$ and $\Delta_{\psi}$, and $C$ is finite-dimensional and is determined exactly as in Theorem \ref{thm:finitedim} below.

\end{alphathm}

The key to attacking this problem is a non-tracial free graph von Neumann algebra $\cM(\Gamma, \mu)$ constructed by the authors in \cite{HNFreeGraph}. Here $\Gamma=(V,E)$ is a finite, directed, connected graph and $\mu\colon E\to \bbR$ is a weighting on the edges. In \cite{HNFreeGraph} the authors identified $\cM(\Gamma, \mu)$ with an almost periodic free Araki--Woods factor (up to direct sum copies of $\bbC$).  The key feature of this von Neumann algebra is that it is naturally expressed as amalgamated free product (see Subsection \ref{subsec:graph} below).  By converting a free product of the form
	\[
		(M_{n}(\C), \phi) * \left[\underset{\beta}{\C} \oplus \underset{1-\beta}{\C}\right]
	\]
to an amalgamated free product over the diagonal of $M_{n}(\C)$ and using induction, we can realize a corner of this free product as corner of some $\cM(\Gamma, \mu)$. We then identify this corner with $\cM(\Gamma', \mu')$ for some other graph $\Gamma'$ and edge weighting $\mu'$.  Using standard free product techniques of Dykema, we are able extend the computation of $M_n(\bbC) * [\bbC\oplus \bbC]$ to free products of arbitrary finite-dimensional algebras. Along the lines of \cite{MR1201693} and \cite{MR1256179} we then develop standard embeddings of almost periodic free Araki--Woods factors, and use these to prove Theorem \ref{thm:hyperfinite2}.
 
The outline of the paper is as follows:  In Section~\ref{sec:prelim}, we establish notation  and review some preliminaries and relevant results about free products. We also recall the construction of the von Neumann algebra $\cM(\Gamma, \mu)$.  In Section~\ref{sec:fd}, we work through the computation of $M_{n}(\C) * [\C \oplus \C]$ (including $n=2$) and use this to prove Theorem~\ref{thm:finitedim2}. We also extend this to a computation of $\cB(\cH) *[\bbC\oplus \bbC]$, where $\cB(\cH)$ is equipped with an arbitrary faithful normal state, and use this to prove a version of Theorem~\ref{thm:finitedim2} that allows infinite-dimensional type $\mathrm{I}$ factors in the direct summands.  In Section~\ref{sec:std_embeddings}, we develop a hyperfinite matricial model for $(T_{H}, \vphi_{H})$ which fuses Shlyakhtenko's matricial model for $(T_{H}, \vphi_{H})$ \cite{Shl97} and Dykema's model for $(L(\F_{t}),\tau)$ \cite{MR1256179}. We use this to develop the notion of a standard embedding of free Araki--Woods factors and utilize this to prove Theorem \ref{thm:hyperfinite2}.

\subsection*{Acknowledgements} 

We would like to thank Corey Jones and David Penneys for initially suggesting the study of the von Neumann algebras $\cM(\Gamma, \mu),$ therefore enabling the writing of this paper. We would also like to thank Dimitri Shlyakhtenko for many helpful conversations about free Araki--Woods factors.  This work was initiated at the Mathematical Sciences Research Institute (MSRI) Summer School on \emph{Subfactors: planar algebras, quantum symmetries, and random matrices}, and continued while Brent Nelson was visiting the Institute for Pure and Applied Mathematics (IPAM) during the Long Program on \emph{Quantitative Linear Algebra}, both of which are supported by the National Science Foundation. Brent Nelson's work was also supported by NSF grant DMS-1502822.


\section{Preliminaries}\label{sec:prelim}


\subsection{Some Notation}\label{subsec:status_quo}

Given the non-tracial nature of our analysis, it is important that we specify the positive linear functionals involved in any free product. Toward that end, we establish some common notation for positive linear functionals that will be frequently used:

	\begin{itemize}
	\item After \cite{MR1201693,Dyk97} we use the following notation:
		\begin{itemize}
		\item[$\circ$] For $t>0$ and a projection $p$ 
			\[
				\overset{p}{\underset{t}{\bbC}}:=( \bbC p, \phi),
			\]
		where $\phi$ is determined by $\phi(p)=t$. We may suppress either `$t$' or `$p$' if they are clear from context. In the context of a direct sum, if $t\leq 0$ then we mean that the summand should be omitted.
		
		\item[$\circ$] For $\alpha_{1}, \cdots, \alpha_{n}>0$ and $p_{1}, \cdots, p_{n}$ orthogonal minimal projections of $M_{n}(\C)$
			\[
				\overset{p_{1},\cdots p_{n}}{\underset{\alpha_{1},\cdots, \alpha_{n}}{M_n(\bbC)}}:= ( M_{n}(\C), \phi),
			\]
		where $\phi$ is determined by $\phi(x)=\Tr(xA)$ where $A = \sum_{i=1}^{n} \alpha_{i}p_{i}$. We may suppress any of `$\alpha_{i}$' or `$p_{i}$' if they are clear from context. 
		
		\item[$\circ$] For $t>0$ and a von Neumann algebra $A$ with identity element $p$ and a state $\phi$
			\[
				\overset{p}{\underset{t}{(A,\phi)}} := (A,t \phi).
			\]
		We may suppress any of `$t$', `$p$', or `$\phi$' if they are clear from context (e.g. a $\mathrm{II}_1$ factor and its canonical trace).
		\end{itemize}
	The above notations allow us to concisely express direct sums with explicit (and sometimes implicit) weightings. E.g.:
		\[
			\overset{p_1}{\underset{t_1}{\bbC}}\oplus \overset{p_2,q_2}{\underset{s_2,t_2}{M_2(\bbC)}}\oplus \overset{p_3}{\underset{t_3}{(A,\vphi)}}.
		\]
	If $t_1+s_2+t_2+t_3=1$ then the associated positive linear functional on this direct sum is a state. However, it will often be notationally convenient to \textbf{not} demand such normalization. If such an unnormalized direct sum appears in a free product, we will ensure that each factor in the free product has the same total mass.
	
	\item Let $\cH$ be a separable infinite-dimensional Hilbert space and $\{e_{i,j}\}_{i,j\in \bbN_0}$ a system of matrix units for $\cB(\cH)$. For $\lambda\in(0,1)$, after \cite{Shl97} we define a normal state $\psi_\lambda\colon \cB(\cH)\to\bbC$ by
		\[
			\psi_\lambda(e_{i,j}):=\begin{cases} \lambda^i(1-\lambda) & \text{if }i=j\\ 0 & \text{otherwise}\end{cases}.
		\]
	If $\cH$ is finite-dimensional so that $\cB(\cH)\cong M_n(\bbC)$ is generated by matrix units $\{e_{i,j}\}_{i,j=0}^{n-1}$, for some $n\in \bbN$, we define a state $\psi_\lambda\colon M_n(\bbC)\to\bbC$ by
		\[
			\psi_\lambda(e_{i,j}):=\begin{cases} \lambda^i \frac{(1-\lambda)}{(1-\lambda^n)} & \text{if }i=j\\ 0 & \text{otherwise}\end{cases}.
		\]

	\item For a von Neumann algebra $A$ with a positive linear functional $\phi$ and a non-zero projection $p\in A$, denote
		\[
			\phi^p(\ \cdot\ ) := \frac{1}{\phi(p)} \phi(p\ \cdot\ p).
		\]
	\end{itemize}


\subsection{Free Araki-Woods factors}

We recall the main features of Shlyakhtenko's almost periodic free Araki--Woods factors \cite{Shl97} that we will use in this paper.  See \cite{Shl97} for the general construction, and \cite{HNFreeGraph} for an overview of the construction.

If $\lambda \in (0, 1)$, then the (unique) type III$_{\lambda}$ almost periodic free Araki--Woods factor $(T_{\lambda}, \varphi_{\lambda})$ arises on the full Fock space of $\C^{2}$, $\cF(\C^{2})$.  If $\{u, v\}$ is an orthonormal basis of $\C^{2}$, then $(T_{\lambda}, \varphi_{\lambda}) \subset B(\cF(\C^{2}))$ is the von Neumann algebra generated by
\[
	y_{\lambda} = \ell(u) + \sqrt{\lambda}\ell(v)^{*}
\] 
with $\ell(\xi)$ the creation operator for $\xi \in \C^{2}$, and $\vphi_{\lambda}$ is the vacuum state.  We will call $y_{\lambda}$ a \emph{generalized circular element of parameter $\lambda$}. Using the polar decomposition of $y_{\lambda}$, it was shown that 
	\begin{align}\label{eqn:fAWf_def}
		(T_{\lambda}, \varphi_{\lambda}) \cong  (L(\Z),\tau) * (\cB(\cH), \psi_{\lambda}).
	\end{align}

For any countable, non-trivial $H < \R_{+}$ with generating set $(\lambda_{i})_{i \in I},$ $\lambda_{i} \in (0, 1)$, Shlyakhtenko showed that the free product $(\cM, \vphi) := \bigast_{i \in I}(T_{\lambda_{i}},\varphi_{\lambda_{i}})$ is a factor that is independent of the generating set of $H$ and the multiplicity of the generators in the free product, and satisfies $\cM^{\vphi} \cong L(\F_{\infty})$ (Recall that $\cM^{\vphi} = \{x \in \cM : \vphi(xy) = \vphi(yx) \text{ for all } y \in \cM\}$ is the \emph{centralizer} of $\vphi$.). Furthermore the state $\varphi$ is almost periodic and $\Delta_{\varphi}$ has point spectrum $H$.  Shlyakhtenko also showed that $(\cM, \vphi)$ is uniquely determined by $H$.

We will denote $(\cM, \vphi)$ as $(T_{H}, \vphi_{H})$, and we will call $\vphi_{H}$ the free quasi-free state on $T_{H}$.  Note that $(T_{H}, \vphi_{H})$ is of type III$_{1}$ if and only if $H$ is not cyclic.  Shlyakhtenko proved the following additional structural results of the factors $(T_{H}, \vphi_{H})$.
\begin{thm*}[\cite{Shl97}]

\begin{itemize} 

\item[]

\item $(T_{H}, \vphi_{H}) * (L(\F_{\infty}),\tau) \cong (T_{H}, \vphi_{H})$.  

\item $(T_{\lambda}, \vphi_{\lambda}) \cong (L(\Z),\tau) * (M_{n}(\C), \psi_{\lambda})$ for any $n \geq 2$. 

\item $(T_{H}, \vphi_{H}) * (T_{H'}, \vphi_{H'}) \cong (T_{G}, \vphi_{G})$ where $G = \langle H \cup H' \rangle$. 

\end{itemize}

\end{thm*}

 The property in the first bullet point is known as \emph{free absorption}, and will be referred to as such throughout the paper. A consequence of free absorption is that $(T_{H}, \vphi_{H}) * (A, \phi) \cong (T_{H}, \vphi_{H})$ whenever $A$ is a countable direct sum of finite-dimensional von Neumann algebras, diffuse hyperfinite von Neumann algebras, and interpolated free group factors where $\phi$ is a trace. 

Moreover, one can choose $(A,\phi)$ to be $\cB(\cH)$ for $\cH$ a separable Hilbert space and $\phi$ \textbf{any} faithful normal state (\emph{cf.} Equation (\ref{eqn:fAWf_def})). Indeed, note that if $\phi$ is a faithful normal state on $\cB(\cH)$, then $\phi(x) = \Tr(xy)$ where $y$ is a positive trace-class operator with trace 1. We can therefore assume that after conjugating by a unitary, there is set of matrix units $\{e_{i,j}\}$ and $\alpha_{i}>0$ satisfying   $\phi(e_{i,j}) = \delta_{i, j}\alpha_{i}$. 

 \begin{prop}\label{prop:L(Z)B(H)}
Assume that $\cH$ is separable and that  $\cB(\cH)$ is equipped with a faithful normal non-tracial state $\phi$ satisfying $\phi(e_{i,j}) = \delta_{i, j}\alpha_{i}$ for a system of matrix units $\{e_{i,j}\}_{i,j\geq 0}$ and $\alpha_i > 0$.  Then
	\[
		(L(\Z),\tau) * (\cB(\cH), \phi) \cong (T_{H}, \vphi_{H}),
	\]
where $H<\bbR^+$ is the subgroup generated $\{\frac{\alpha_i}{\alpha_j}\colon i,j \geq 0\}$.
\end{prop}
\begin{proof}
Let $\cK$ be a separable and infinite-dimensional Hilbert space containing an infinite orthonormal set $(\xi_{ij})_{i, j \in \N_{\geq 0}}$, and $\cF(\cK)$ the full Fock space on $\cK$.  Let $\omega$ be the state on $\cB(\cF(\cK))$ defined by  $\omega(x) = \langle x\Omega, \Omega \rangle$ with $\Omega$ the vaccum vector in $\cF(\cK)$.   Within $(\cB(\cF(\cK)) \otimes \cB(\cH), \omega \otimes \phi)$, let $\displaystyle L = \sum_{i, j \geq 0} \sqrt{\alpha_{i}}\ell(\xi_{i, j}) \otimes e_{i, j}$ with $\ell(\xi_{ij})$ the creation operator.  Then $(\cB(\cH), \phi) * L(\Z) $ is modeled by the sub von Neumann algebra $\cM \subset (B(\cF(\cK)) \otimes \cB(\cH), \omega \otimes \phi)$ generated by $L+L^{*}$ and $1 \otimes \cB(\cH),$ and $L+L^{*}$ is $*$-free from $1 \otimes \cB(\cH)$ \cite{Shl97}.  From this, we see that $e_{0,0}\cM e_{0,0} = W^{*}((\sqrt{\alpha_{i}}\ell(\xi_{ij}) + \sqrt{\alpha_{j}}\ell(\xi_{ji})^{*}) \otimes e_{0,0}\colon i,j\geq 0)$.  From \cite{Shl97}, this means that $(e_{0,0}\cM e_{0,0}, (\omega \otimes \phi)^{e_{0,0}} ) \cong (T_{H}, \vphi_{H})$.

We also note that if $D$ is the diagonal of $\cB(\cH)$, then $(L(\Z),\tau) * (D, \phi)$ is a factor by \cite{MR1201693}.  Therefore, $1\otimes e_{0,0}$ has full central support in $\cM^{\omega\otimes \phi}$.  Applying Lemma \ref{lem:antman} below finishes the proof.
\end{proof}

This proposition spawns the following useful corollaries, which we will use extensively

\begin{cor}\label{cor:ArakiB(H)}
Assume that $\cH$ is separable and that  $\cB(\cH)$ is equipped with a faithful state $\phi$ satisfying $\phi(e_{i,j}) = \delta_{i, j}\alpha_{i}$ for a system of matrix units $\{e_{i,j}\}_{i,j\geq 0}$ and $\alpha_i > 0$, and $H$ is a countable subgroup of $\R_{+}$.  Then
	\[
		(T_{H}, \vphi_{H}) * (\cB(\cH), \phi) \cong (T_{G}, \vphi_{G})
	\]
where $G$ is the group generated by $H$ and $H'$ where $H' = \langle\frac{\alpha_{i}}{\alpha_{j}} : i, j \geq 0\rangle$.
\end{cor}

\begin{proof}

Using Proposition~\ref{prop:L(Z)B(H)} and free absorption, we have
	\begin{align*}
		(\cB(\cH), \phi) * (T_{H}, \vphi_{H}) &\cong (\cB(\cH), \phi) * (L(\Z),\tau) * (T_{H}, \vphi_{H})\\
				&\cong (T_{H'}, \vphi_{H'}) * (T_{H}, \vphi_{H}) \cong (T_{G}, \vphi_{G})\qedhere
	\end{align*}
\end{proof}

\begin{cor}\label{cor:tensor}
Let $H < \R^{+}$ be countable and non-trivial. Let $\alpha_{0} \geq \cdots \geq \alpha_{n} \geq \cdots > 0$  have the property that $\frac{\alpha_{i}}{\alpha_{j}} \in H$ for any pair $i, j\geq 0$ (we are allowing for infinite or finite sequences). Assume that $\cH$ is separable and that  $\cB(\cH)$ is equipped with a faithful normal non-tracial state $\phi$ satisfying $\phi(e_{i,j}) = \delta_{i, j}\alpha_{i}$ for a system of matrix units $\{e_{i,j}\}_{i,j\geq 0}$. Then
	\[
		(T_{H}, \vphi_{H}) \cong (T_{H}, \vphi_{H}) \otimes (\cB(\cH), \phi).
	\]

\end{cor}

\begin{proof}

Note from Corollary~\ref{cor:ArakiB(H)} that
	\[
		(\cM , \vphi) := (T_{H}, \vphi_{H}) * (\cB(\cH), \phi) \cong (T_{H}, \vphi_{H}).
	\]
By Lemma \ref{lem:antman} (see below), this also establishes that $(e_{0,0}\cM e_{0,0}, \varphi^{e_{0,0}}) \cong (T_{H}, \vphi_{H})$.  This completes the proof since 
	\[
		(\cM, \varphi) =(e_{0,0}\cM e_{0,0}, \varphi^{e_{0,0}}) \otimes  (\cB(\cH), \phi).\qedhere
	\]
\end{proof}

\subsection{References to existing results}\label{subsec:existing_results}

For the convenience of the reader, we state here some existing results that will be cited frequently in the present paper. Where appropriate, we have adapted the notation. In particular, for $M$ a von Neumann algebra and $p\in M$ a projection, we denote the central support of $p$ in $M$ by $z(p\colon M)$.

The first lemma concerns free products with respect to general states and follows from the same proof as \cite[Theorem 1.2]{MR1201693} (see also \cite[Proposition 5.1]{Dyk97} and \cite[Proposition 3.10]{Hou07}). In particular, we will frequently use the cases when either $\cB(\cH)=\bbC$ or $B=0$.

\begin{lem}\label{lem:Dykema}
Let $(A,\phi)$, $(B,\psi)$, and $(C,\nu)$ be von Neumann algebras equipped with faithful normal states. Let $\cH$ be a separable Hilbert space, equip $\cB(\cH)$ with a faithful normal state $\omega$, and let $p\in \cB(\cH)^{\omega}$ be a minimal projection. If
	\begin{align*}
		(M,\vphi)&:= \left[ \left\{(A,\phi)\bar{\otimes} (\cB(\cH),\omega)\right\} \oplus (B,\psi)\right] * (C,\nu)\\
		(N,\vphi)&:= \left[ (\cB(\cH),\omega) \oplus (B,\psi)\right]* (C,\nu),
	\end{align*}
then
	\[
		(pMp,\vphi^p) \cong (pNp,\varphi^p) * (A,\phi).
	\]
Moreover, $z(p\colon M)=z(p\colon N)$.
\end{lem}

Let $A$ and $B$ be von Neumann algebras with normal faithful states $\phi$ and $\psi$ respectively, and let $i: A \rightarrow B$ be a normal, injective von Neumann algebra homomorphism. After \cite[Definition 1.4]{Hou07}, we say that $i$ is a \emph{modular inclusion} if it is state preserving and if $i(A)$ is globally invariant under the modular group $\sigma^{\psi}$.

The next lemma was useful in helping determine the structure of the free graph von Neumann algebras studied in \cite{HNFreeGraph} (see Subsection \ref{subsec:graph}).  It will also be useful in establishing a suitable base case in our computation of free products of finite-dimensional von Neumann algebras.

\begin{lem}[{\cite[Theorems 3.1 and 4.3]{Hou07}}]\label{lem:Hou}

Suppose $\alpha$ and $\beta$ satisfy $\frac12 \leq \beta \leq \alpha < 1$ with $\alpha > \frac12$, and let $\lambda = \frac{1-\alpha}{\alpha}$

\begin{enumerate}

\item[\normalfont{(1)}] $\underset{\alpha, 1 - \alpha}{M_{2}(\C)} * \left[\underset{\beta}{\C} \oplus \underset{1-\beta}{\C}\right] \cong (T_{\lambda}, \vphi_{\lambda})$.

\item[\normalfont{(2)}] Suppose $(A, \phi)$ and $(B, \psi)$ are two von Neumann algebras with faithful normal states, and that there exist modular inclusions $\underset{\alpha, 1 - \alpha}{M_{2}(\C)} \hookrightarrow A$ and $(\underset{\beta}{\C} \oplus \underset{1-\beta}{\C}) \hookrightarrow B$.  Then
$$
(A, \phi) * (B, \psi) \cong (A, \phi) * (B, \psi)  * (T_{\lambda}, \vphi_{\lambda}).
$$

\end{enumerate}

\end{lem}

The following lemma is a crucial ingredient for converting certain free products over the scalars into amalgamated free products and vice-versa. Consequently, it lets us appeal to our graph algebras in Subsection \ref{subsec:graph}.

\begin{prop}[{\cite[Proposition 4.1]{Hou07}}]\label{prop:amalgamate}

Let $(M, \phi)$ be a von Neumann algebra with a faithful normal state, and $B \subset M$ a von Neumann subalgebra with a $\phi$-preserving conditional expectation $E_{1}: M \rightarrow B$.  Let $(A, \psi)$ be another von Neumann algebra with faithful normal state, and $E_{2}: (A, \psi) * (B, \phi) \rightarrow (B, \phi)$ the canonical $\phi$-preserving conditional expectation.  Set 
$$
(\cM, E) = (M, E_{1}) \Asterisk_B ((A, \psi) * (B, \phi), E_{2}).
$$
Then if $\vphi = \phi \circ E$,
$$
(\cM, \vphi) \cong (M, \phi) * (A, \psi).
$$
\end{prop}

The following result appears in \cite{HNFreeGraph}. This lemma, combined with Lemma \ref{lem:Dykema} enables us to compute free products by examining suitable compressions.  It will be used extensively.

\begin{lem}[{\cite[Lemmas 3.1 and 3.2]{HNFreeGraph}}]\label{lem:antman}
{\color{white}spacer}
\begin{enumerate}
\item[\normalfont{(1)}] Let $H$ be a countable multiplicative subgroup of $\R^{+}$ and let $p \in (T_{H})^{\vphi_H}$ be a nonzero projection.  Then 
	\[
		\left(pT_{H}p, \vphi_{H}^{p}\right) \cong (T_{H}, \varphi_{H}).
	\]

\item[\normalfont{(2)}] Let $M$ be a von Neumann algebra with almost-periodic faithful normal state $\vphi$. Let $p\in M^\vphi$ a projection such that $(pMp, \vphi^p)\cong (T_H,\vphi_H)$ for some non-trivial, countable subgroup $H$ of $\R^+$, and such that $z:=z(p\colon M^\vphi)=z(p\colon M)$. Then
	\[
		(z M,\vphi^{z})\cong (T_H,\vphi_H).
	\]
In particular, if $z(p\colon M^\vphi)=1$ then $(M,\vphi)\cong (T_H,\vphi_H)$.

\end{enumerate}
\end{lem}

\subsection{Non-tracial graph algebras}\label{subsec:graph}

In \cite{HNFreeGraph}, a von Neumann algebra was constructed from a weighted graph. We will outline the construction here. 

To begin, we let $\Gamma$ be a finite directed graph with vertex set $V$ and edge set $E$ with source and target maps $s$ and $t$ respectively.  There is an involution on $E$, denoted $\op$, which satisfies $s(e^{\op}) = t(e)$ and $t(e^{\op}) = s(e)$ for all $e \in E$.  If $e$ is a self-loop based at $v \in V$, then it is possible to have $e^{\op} = e$, but we do not require this.  Denote the space of loops in $\Gamma$ by $\Lambda_\Gamma$.

We will also assume that $\Gamma$ comes equipped with an edge weighting $\mu: E \rightarrow \R^{+}$ such that $\mu(e)\mu(e^{\op})=1$ for all $e\in E$. For $\sigma=e_1\cdots e_n\in \Lambda_\Gamma$, denote $\mu(\sigma)=\mu(e_1)\cdots \mu(e_n)$. In order to ensure that the von Neumann algebra we construct is non-tracial, we will assume:
	\[
		\exists \sigma\in \Lambda_\Gamma \colon \mu(\sigma)\neq 1.
	\]

We define $A := \ell^{\infty}(V)$ and we let $p_v$ denote the indicator function on $v\in V$.  In \cite{HNFreeGraph}, the authors showed that there is a Fock space representation of a  C*-algebra $\cS(\Gamma, \mu)$ generated by $A$ and elements $(Y_{e})_{e \in E}$ satisfying $p_{v}Y_{e}p_{w} = \delta_{v, s(e)}\cdot \delta_{w, t(e)}Y_{e}$ and $Y_{e}^{*} = \frac{1}{\sqrt{\mu(e)}}Y_{e^{\op}}$.  In addition, this Fock space representation produced a faithful conditional expectation $\bbE: \cS(\Gamma, \mu) \rightarrow A$ under which the C*-algebras C$^{*}(A, Y_{f},Y_{f^{op}})$ as $f$ ranges through all pairs $(e, e^{\op})$ are free with amalgamation over $A$ under $\bbE$.

The key observation about the elements $Y_{e}$ that we will use in this paper is their distribution.   Specifically, if we let $\phi$ be any faithful positive linear functional on $\cS(\Gamma,\mu)$ that is preserved by $\bbE$, and let $(\cM(\Gamma,\mu), \phi)$ denote the von Neumann algebra generated by $\cS(\Gamma,\mu)$ via the GNS representation associated to $\phi$, then as a subalgebra of $(\cM(\Gamma,\mu), \phi)$,
	\[
		(p_{s(e)}W^*(Y_{e}Y_{e}^{*})p_{s(e)}, \phi^p) \cong \begin{cases}
(L(\Z),\tau) &\text{ if } \mu(e) \geq 1\\
\underset{\phi(p_{s(e)})\mu(e)}{(L(\Z),\tau)} \oplus \underset{\phi(p_{s(e)})(1 - \mu(e))}{\C} &\text{ if } \mu(e) < 1
\end{cases}.
	\]
If we let $Y_{e} = u_{e}|Y_{e}|$ be the polar decomposition, then $u_{e}u_{e}^{*} = p_{s(e)}$ if and only if $\mu(e) \geq 1$ and $u_{e}^{*}u_{e} = p_{t(e)}$ if and only if $\mu(e) \leq 1$.\\

The main result of \cite{HNFreeGraph} is identifying $\cM(\Gamma, \mu)$ with an almost periodic free Araki-Woods factor under an appropriate positive linear functional, $\varphi$.   To construct $\vphi$, we let $\Gamma_{\Tr}$ be a subgraph of $\Gamma$ maximal subject to the condition
	\[
		\mu(\sigma) = 1\qquad \forall \sigma\in \Lambda_{\Gamma_{\Tr}}.
	\]
Note that while $\Gamma_{\Tr}$ need not be unique, the condition $\mu(e)\mu(e^{\op})=1$ implies it will always contain every vertex of $V$.  We define $\vphi$ on $A$ as follows:. Let $*$ be a fixed vertex of $V$, pick $\alpha \in \R^{+}$ and declare $\vphi(p_{*}) = \alpha$.  For any $v \in V$, then define
	\[
		\vphi(p_{v}) = \mu(e_{1})\cdots\mu(e_{n})\alpha
	\]
where $e_{1}\cdots e_{n}$ is any path in $\Gamma_{\Tr}$ with source $*$ and target $v$.  We extend $\vphi$ to be defined on $\cM(\Gamma, \mu)$ by pre-composing with $\bbE$.  Let $\Delta_{\vphi}$ be the modular operator of $\vphi$. By \cite{HNFreeGraph}, we have the following:

\begin{thm}\label{thm:freegraph}
\begin{enumerate}[(1)] \item[]
\item $\vphi$ is an almost periodic positive linear functional on $\cM(\Gamma, \mu)$.

\item $\cM(\Gamma_{\Tr}, \mu) \subset \cM(\Gamma, \mu)^{\vphi}$.

\item Each $Y_{e}$ is an eigenoperator  of $\Delta_{\vphi}$ with eigenvalue $\mu(e)\mu(e_{1})\cdots\mu(e_{n})$ such that $e_{1}\cdots e_{n}$ is a path in $\Gamma_{\Tr}$ with source $t(e)$ and target $s(e)$.

\item Suppose that $H:=\< \mu(\sigma)\colon \sigma \in \Lambda_\Gamma\> <\bbR^+$ is non-trivial. Then
	\[
		(\cM(\Gamma, \mu), \vphi) \cong (T_H,\vphi_H)\oplus \bigoplus_{v\in V} \overset{r_v}{\bbC},
	\]
where $r_v\leq p_v$ is non-zero if and only if $\displaystyle \sum_{\substack{e\in E\\ s(e)=v}} \mu(e)<1$, in which case
	\[
		\vphi(r_v)=\vphi(p_v)\left[ 1- \sum_{\substack{e\in E\\ s(e)=v}} \mu(e)\right].
	\]
In particular, if
	\[
		\sum_{\substack{e\in E\\ s(e)=v}} \mu(e) \geq 1
	\]
for all $v\in V$, then $(\cM(\Gamma, \mu), \vphi)\cong (T_H,\vphi_H)$.

\end{enumerate}
\end{thm}

We will use this graphical picture of $(T_{H}, \vphi_{H})$ as a mechanism for realizing free products of certain finite-dimensional von Neumann algebras.

\section{Free Products of Finite-Dimensional von Neumann Algebras}\label{sec:fd}

In this section we will compute free products of arbitrary finite-dimensional von Neumann algebras (see Theorem \ref{thm:finitedim}). Our first step is to compute for $\alpha,\beta\in \left(\frac12,1\right)$
	\[
		\underset{\alpha,1-\alpha}{M_2(\C)} * \left[\underset{\beta}{\C}\oplus \underset{1-\beta}{\C}\right].
	\]
In the case $\alpha \geq \beta$, this is done in \cite[Theorem 3.1]{Hou07}. To handle the case $\alpha<\beta$, we will appeal to the graphical picture of the free Araki--Woods factors from Subsection \ref{subsec:graph}. This computation will serve as the base case for computing more general free products of the form $M_n(\C)*[\C\oplus \C]$.

\subsection{Computing $M_{2}(\C) * [\C \oplus \C]$}

\begin{thm}\label{thm:M2}

Suppose that $\frac12 < \alpha < \beta < 1$.  Define
	\[
		(\cM,\vphi) := \underset{\alpha, 1-\alpha}{\overset{e_{11}, e_{22}}{M_{2}(\C)}} * \left[\underset{\beta}{\overset{q}{\C}} \oplus \underset{1-\beta}{\C}\right].
	\]
Then for $\lambda := \frac{1-\alpha}{\alpha}$ and $\gamma := \frac{1-\beta}{\alpha} + \frac{1 - \beta}{1 - \alpha}$, one has
	\[
		(\cM,\vphi) \cong \begin{cases}
					(T_{\lambda}, \vphi_{\lambda}) &\text{ if } \gamma \geq 1\\
					\underset{\gamma}{(T_{\lambda}, \vphi_{\lambda})} \oplus \underset{\alpha(1 - \gamma), (1-\alpha)(1 - \gamma)}{\overset{\overline{e_{11}}, \overline{e_{22}}}{M_{2}(\C)}} &\text{ if } \gamma < 1
				\end{cases}.
	\]
In the second case, one has $\overline{e_{ii}} \leq e_{ii}\wedge q$ for $i \in \{1, 2\}$.

\end{thm}

\begin{proof}
We begin by writing $(\cM,\varphi)$ as an amalgamated free product so that we can identify it with (a corner of) a free graph von Neumann algebra. Let $D = \underset{\alpha}{\overset{e_{11}}{\C}} \oplus \underset{1-\alpha}{\overset{e_{22}}{\C}}$ be the diagonal of $M_{2}(\C)$ and let
	\begin{align*}
		E_{1} &\colon  \underset{\alpha, 1-\alpha}{\overset{e_{11}, e_{22}}{M_{2}(\C)}} \rightarrow D\\
		E_{2} &\colon \left[\underset{\beta}{\overset{q}{\C}} \oplus \underset{1-\beta}{\C}\right] * D \rightarrow D 
	\end{align*}
be the canonical state-preserving conditional expectations.  From Proposition \ref{prop:amalgamate}, we have
	\[
		(\cM,\vphi) \cong \left(\left(\underset{\alpha, 1-\alpha}{\overset{e_{11}, e_{22}}{M_{2}(\C)}}, E_{1}\right) \underset{D}{\Asterisk} \left( \left[\underset{\beta}{\overset{q}{\C}} \oplus \underset{1-\beta}{\C}\right] * D, E_{2}\right),\vphi\right).
	\]
By \cite{MR1201693}, this free product is
	\[
		\left(\underset{\alpha, 1-\alpha}{\overset{e_{11}, e_{22}}{M_{2}(\C)}}, E_{1}\right) \underset{D}{\Asterisk} \left(\underset{\beta + \alpha - 1}{\overset{q \wedge e_{11}}{\C}} \oplus \underset{(1 - \beta), (1 - \beta)}{M_{2}(L(\Z))} \oplus \underset{\beta - \alpha}{\overset{q \wedge e_{22}}{\C}}, E_{2}\right).
	\]

Let $\Gamma=(V,E)$ be the following graph with edge weighting $\mu$:
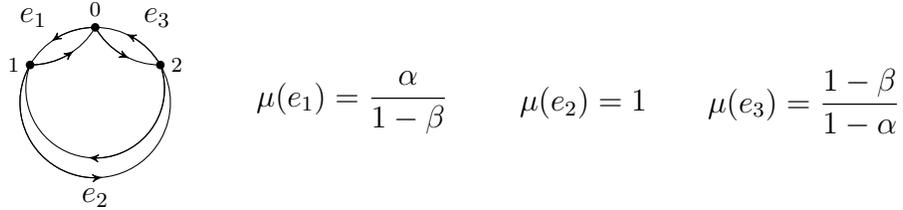
	\[
		\begin{tikzpicture}
		\node[left] at (-0.866,0.5) {\scriptsize$1$}; 
		\draw[fill=black] (-0.866,0.5) circle (0.05);
		
		\node[above left] at (-0.5,0.866) {$e_1$}; 
		\draw (0,1) arc (90:150:1);
		\draw[->] (0,1) arc (90:125:1);
		\draw (-0.866,0.5) arc (-90:-30:1);
		\draw[->] (-0.866,0.5) arc (-90:-55:1);
		
		\node[above] at (0,1) {\scriptsize$0$}; 
		\draw[fill=black] (0,1) circle (0.05);
		
		\node[above right] at (0.5,0.866) {$e_3$}; 
		\draw (0.866,0.5) arc (30:90:1);
		\draw[->] (0.866,0.5) arc (30:65:1);
		\draw (0,1) arc (-150:-90:1);
		\draw[->] (0,1) arc (-150:-115:1);
		
		\node[right] at (0.866,0.5) {\scriptsize$2$}; 
		\draw[fill=black] (0.866,0.5) circle (0.05);
		
		\draw (-0.866,0.5) arc (-210:30:1); 
		\draw[->] (-0.866,0.5) arc (-210:-85:1);
		\draw (0.866,0.5) arc(20:-200: 0.9216);
		\draw[->] (0.866,0.5) arc(20:-95:0.9216);
		\node[below] at (0, -1) {$e_2$};
		
		\node[right] at (2,0) {$\displaystyle \mu(e_1)=\frac{\alpha}{1-\beta}$};
		
		\node[right] at (5.5,0) {$\mu(e_2)=1$};
		
		\node[right] at (8,0) {$\displaystyle \mu(e_3)=\frac{1-\beta}{1-\alpha}$};

		\end{tikzpicture}
	\]
Let $\cM(\Gamma,\mu)$ be the associated graph algebra. By choosing $\Gamma_{\Tr}$ so that it contains $e_1$ and $e_3$ and declaring $\phi(p_0):=1-\beta$, we can find a faithful normal positive linear functional $\phi$ on $\cM(\Gamma,\mu)$ satisfying:
	\[
		\phi(p_1)=\alpha \qquad \phi(p_2) = 1-\alpha.
	\]
Consider the following subalgebra of $\cM(\Gamma, \mu)$:
	\[
		\cN := W^{*}(Y_{e_1}, u_{2}, u_{3}, p_{0}, p_{1}, p_{2})
	\]
where $u_{i}$ is the polar part of $Y_{e_i}$.  Set $P = p_{1} + p_{2}$ and $B = \underset{\alpha}{\overset{p_{1}}{\C}} \oplus \underset{1-\alpha}{\overset{p_{2}}{\C}}$, and let
	\[
		E_B \colon P\cM(\Gamma,\mu)P \to B
	\] 
be a $\phi^P$-preserving conditional expectation.  Recall that $Y_{e_{1}}$, $u_{2}$, and $u_{3}$ are free with amalgamation over $W^{*}(p_{0}, p_{1}, p_{2})$, and therefore
$$
(P\cN P, E_B)\cong \left( PW^{*}(u_{2}, p_{0}, p_{1}, p_{2})P,E_B \right) \Asterisk_B \left( PW^{*}(Y_{e_1}, u_{3}, p_{0}, p_{1}, p_{2})P,E_B \right)
$$
Since $\mu(e_{1})= 1$, it follows that $u_{2}^{*}u_{2} = p_{2}$ and $u_{2}u_{2}^* = p_{1}$ hence 
$$
(PW^{*}(u_{2}, p_{0}, p_{1}, p_{2})P,E_B) \cong \left(\underset{\alpha, 1-\alpha}{\overset{e_{11}, e_{22}}{M_{2}(\C)}}, E_{1}\right).
$$  
where the isomorphism sends $p_{i}$ to $e_{ii}$. Since $\mu(e_{1}) > 1$, it follows that $Y_{e_{1}}Y_{e_{1}}^{*}$ is diffuse in $p_{0}\cN p_{0}$ and $Y_{e_{1}}^{*}Y_{e_{1}}$ has an atom of size $\beta + \alpha - 1$ in $p_{1}\cN p_{1}$.  Furthermore, $u_{3}^{*}u_{3} = p_{0}$ and $u_{3}u_{3}^{*}$ is a projection of mass $1 - \beta$ under $p_{2}$.  Therefore 
$$
\left( PW^{*}(Y_{e_1}, u_{3}, p_{0}, p_{1}, p_{2})P,E_B \right) \cong \left(\underset{\beta + \alpha - 1}{\overset{q \wedge e_{11}}{\C}} \oplus \underset{(1 - \beta), (1 - \beta)}{M_{2}(L(\Z))} \oplus \underset{\beta - \alpha}{\overset{q \wedge e_{22}}{\C}}, E_{2}\right).
$$
Hence
$$
(P\cN P, E_B) \cong \left(\underset{\alpha, 1-\alpha}{\overset{e_{11}, e_{22}}{M_{2}(\C)}}, E_{1}\right) \Asterisk_D \left(\underset{\beta + \alpha - 1}{\overset{q \wedge e_{11}}{\C}} \oplus \underset{(1 - \beta), (1 - \beta)}{M_{2}(L(\Z))} \oplus \underset{\beta - \alpha}{\overset{q \wedge e_{22}}{\C}}, E_{2}\right),
$$
and consequently
	\[
		(P\cN P,\phi^P) \cong (\cM, \vphi).
	\]
So it suffices to compute $(P\cN P,\phi^P)$. Observe that
	\[
		(P\cN P,\phi^P)\cong (p_1 \cN p_1,\phi^{p_1})\otimes \underset{\alpha,1-\alpha}{M_2(\C)}
	\]	
since $p_1$ and $p_2$ are equivalent in $P\cN P$, and so it further suffices to compute $(p_1 \cN p_1,\phi^{p_1})$. This will be accomplished by viewing $p_1\cN p_1$ as living under $p_0+p_1$ rather than $P$.

Set $P' = p_{0} + p_{1}$ and $B' = \underset{1-\beta}{\overset{p_{0}}{\C}} \oplus \underset{\alpha}{\overset{p_{1}}{\C}}$, and let
	\[
		E_{B'}\colon P'\cM(\Gamma,\mu) P' \to B'
	\]
be a $\phi^{P'}$-preserving conditional expectation.  Using the fact that $u_{2}u_{3}$ is a partial isometry and the decomposition of $\cN$ as an amalgamated free product, it follows that
	\begin{align*}
		(P'\cN P', E_{B'}) &\cong (P'W^{*}(u_{2}, u_{3}, p_{0}, p_{1}, p_{2})P', E_{B'}) \Asterisk_{B'} (P'W^{*}(Y_{e_1}, p_{0}, p_{1}, p_{2})P', E_{B'})\\
			&\cong \left(\underset{1-\beta, \frac{\alpha}{1-\alpha}(1-\beta)}{\overset{p_{0}, r_{1}}{M_{2}(\C)}} \oplus \underset{\alpha(1 - \frac{1-\beta}{1-\alpha})}{\overset{p_{1} - r_{1}}{\C}}, E_{B'}\right) \Asterisk_{B'} \left(\underset{1-\beta, 1-\beta}{M_{2}(L(\Z))} \oplus \underset{\alpha - (1-\beta)}{\C}, E_{B'}\right),
	\end{align*}   
where in both factors of the free product, $p_{0}$ and $p_{1}$ are realized by:
	\[
		p_{0} = \begin{pmatrix} 1 & 0 \\ 0 & 0 \end{pmatrix} \oplus 0 \qquad\text{ and }\qquad p_{1} = \begin{pmatrix} 0 & 0 \\ 0 & 1 \end{pmatrix} \oplus 1.
	\]
Consider the second factor in the above free product. By Dykema's picture \cite{MR1201693} of the isomorphism
	\[
		\underset{1-\beta, 1-\beta}{M_{2}(L(\Z))} \oplus \underset{\alpha - (1-\beta)}{\C}\cong \left[\underset{1-\beta}{\overset{p_{0}'}{\C}} \oplus \underset{\alpha}{\overset{p_{1}'}{\C}}\right] * \left[\underset{1-\beta}{\C} \oplus \underset{\alpha}{\C}\right].
	\]
It follows that there is an isomorphism which maps $p_{i}$ to $p_{i}'$ for $i \in \{0, 1\}$.  Therefore, by Proposition \ref{prop:amalgamate}, 
	\[
		(P'\cN P',\phi^{P'}) \cong \left[\underset{1-\beta, \frac{\alpha}{1-\alpha}(1-\beta)}{\overset{p_{0}, r_{1}}{M_{2}(\C)}} \oplus \overset{p_1-r_1}{\underset{\alpha(1 - \frac{1-\beta}{1-\alpha})}{\C}}\right]  * \left[\underset{1-\beta}{\C} \oplus \underset{\alpha}{\C}\right].
	\]
Consider the following von Neumann subalgebra of $P'\cN P'$:
	\[
		(\cK,\phi^{P'}):= \left[\underset{(1 - \beta) +\frac{\alpha}{1-\alpha}(1-\beta)}{\overset{p_{0} + r_{1}}{\C}} \oplus \overset{p_1-r_1}{\underset{\alpha(1 - \frac{1-\beta}{1-\alpha})}{\C}} \right] *  \left[\underset{1-\beta}{\C} \oplus \underset{\alpha}{\C}\right].
	\]
Let $Q:= p_0 + r_1$. Then by Lemma \ref{lem:Dykema}, we have
	\[
		( Q\cN Q,\phi^Q) \cong \underset{1-\beta, \frac{\alpha}{1-\alpha}(1-\beta)}{\overset{p_{0}, r_{1}}{M_{2}(\C)}} * (Q\cK Q, \phi^Q).
	\]\\

\noindent
\underline{\textbf{Case 1:}} Assume $\gamma \geq 1$. Note that this implies $1-\beta+\frac{\alpha}{1-\alpha}(1-\beta) \geq \alpha$. In this case, $Q$ has full central support in $\cK \subset (P'\cN P')^{\phi^{P'}}$.  Furthermore, after computing $Q\cK Q$ we see that
	\[
		(Q\cN Q,\phi^Q) \cong  \underset{1-\beta, \frac{\alpha}{1-\alpha}(1-\beta)}{\overset{p_{0}, r_{1}}{M_{2}(\C)}} * \left[\underset{ \frac{\alpha}{1-\alpha}(1-\beta)}{\C} \oplus \underset{1-\beta + \frac{\alpha}{1-\alpha}(1-\beta) - \alpha}{\C} \oplus \underset{\alpha(1 - \frac{1-\beta}{1-\alpha})}{(L(\Z),\tau)}\right].
	\]
The right hand side accepts a trace-preserving inclusion of $\underset{ \frac{\alpha}{1-\alpha}(1-\beta)}{\C} \oplus \underset{1-\beta}{\C}$, so it follows from Lemma~\ref{lem:Hou} and free absorption that 
	\[
		(Q\cN Q, \phi^Q) \cong (T_{\lambda}, \vphi_{\lambda}).
	\]
Using Lemma~\ref{lem:antman} we see that $(P'\cN P',\vphi^{P'}) \cong (T_{\lambda}, \vphi_{\lambda})$.  Applying Lemma~\ref{lem:antman} again, we obtain 
$(p_{1}\cN p_{1},\phi^{p_1}) \cong (T_{\lambda}, \vphi_{\lambda})$.  Therefore
	\[
		(\cM,\vphi) \cong (P\cN P, \phi^P) \cong (p_{1}\cN p_{1},\phi^{p_1}) \otimes \underset{\alpha, 1-\alpha}{M_{2}(\C)} \cong (T_{\lambda}, \vphi_{\lambda}).
	\]\\

\noindent
\underline{\textbf{Case 2:}} Assume $\gamma < 1$. In this case, define $\overline{p_{1}} :=  P' - z(Q\colon P'\cN P') = P' - z(Q\colon\cK)$.  Then $\overline{p_{1}} \leq p_1$ is nonzero, minimal, and has mass $\alpha\left(1 - \gamma\right)$.  Computing $Q\cK Q$ in this case yields
	\[
		(Q\cN Q, \phi^Q) \cong  \underset{1-\beta, \frac{\alpha}{1-\alpha}(1-\beta)}{\overset{p_{0}, r_{1}}{M_{2}(\C)}} * \left[\underset{\frac{\alpha}{1-\alpha}(1-\beta)}{\C} \oplus \underset{1-\beta}{(L(\Z),\tau)}\right].
	\]
The right hand side of this free product once again accepts a trace-preserving inclusion of $\underset{ \frac{\alpha}{1-\alpha}(1-\beta)}{\C} \oplus \underset{1-\beta}{\C}$.  Arguing as above gives
	\[
		(p_{1}\cN p_{1},\phi^{p_1}) \cong \underset{\alpha\gamma}{(T_{\lambda}, \vphi_{\lambda})} \oplus \underset{\alpha\left(1 - \gamma\right)}{\overset{\overline{p_1}}{\C}}, 
	\]
and tensoring gives
	\[
		(\cM,\vphi) \cong \underset{\gamma}{(T_{\lambda}, \vphi_{\lambda})} \oplus \underset{\alpha\left(1 - \gamma\right), (1-\alpha)\left(1 - \gamma\right)}{\overset{\overline{e_{11}} ,\overline{e_{22}}}{M_{2}(\C)}}. 
	\]
Note that by construction $\overline{e_{ii}} \leq e_{ii}$ for each $i$.  Also note that $1-q$ is in the diffuse summand of $\cM$, so by minimality, $\overline{e_{ii}} \leq q$.
\end{proof}

\begin{rem}
One might hope to use this graphical picture in the case $\alpha \geq \beta$ and thereby recover \cite[Theorem 3.1]{Hou07}.  However, this result was used directly in the above proof in the form of Lemma \ref{lem:Hou}.  Thus the above theorem should be seen as an extension of \cite[Theorem 3.1]{Hou07} rather than a generalization.
\end{rem}

By following the arguments of Theorem 4.3 in \cite{Hou07}, we obtain the following corollary.

\begin{cor}\label{cor:Houdayer}
Assume that  $\frac{1-\beta}{\alpha} + \frac{1 - \beta}{1 - \alpha} \geq 1$. Suppose $(A, \phi)$ and $(B, \psi)$ are two von Neumann algebras with faithful normal states, and that there exist modular inclusions $\underset{\alpha, 1 - \alpha}{M_{2}(\C)} \hookrightarrow A$ and $(\underset{\beta}{\C} \oplus \underset{1-\beta}{\C}) \hookrightarrow B$.  Then
	\[
		(A, \phi) * (B, \psi) \cong (A, \phi) * (B, \psi)  * (T_{\lambda}, \vphi_{\lambda}).
	\]
\end{cor}


\subsection{Computing $M_{n}(\C) * [\C \oplus \C]$  }

We will compute $M_{n}(\C) * [\C \oplus \C]$ in terms of almost periodic free Araki-Woods factors.  Before we do so, we need the following proposition.
\begin{prop}\label{prop:createfree}

Let $D = \underset{\alpha_{1}}{\overset{e_{11}}{\C}} \oplus \cdots \oplus \underset{\alpha_{n}}{\overset{e_{nn}}{\C}}$ be embedded down the diagonal of $\underset{\alpha_{1}, \dots, \alpha_{n}}{M_{n}(\C)}$ with standard matrix units $\{e_{ij}\}$.  Let $E_{1}$ be the canonical state-preserving conditional expectation $E_{1}: \underset{\alpha_{1}, \dots, \alpha_{n}}{M_{n}(\C)} \rightarrow D$.  Assume that $A$ is a von Neumann algebra containing $D$ with a conditional expectation $E_{2}: A \rightarrow D$.  Define
	\begin{align*}
		(\cM, E) &:= (\underset{\alpha_{1}, \dots, \alpha_{n}}{M_{n}(\C)}, E_{1}) \underset{D}{\Asterisk} (A, E_{2}),\\
		(\cN, E) &:= \left( \underset{\alpha_{1}, \dots, \alpha_{n}}{\overset{e_{11}, \cdots, e_{n-1, n-1}}{M_{n-1}(\C)}} \oplus \underset{\alpha_{n}}{\overset{e_{nn}}{\C}} , E_{1}\right) \underset{D}{\Asterisk} (A, E_{2}).
	\end{align*} 
Let $P = e_{n-1, n-1} + e_{n,n}$ and $D' = \underset{\alpha_{n-1}}{\overset{e_{n-1, n-1}}{\C}} \oplus \underset{\alpha_{n}}{\overset{e_{n, n}}{\C}}$.  Then
	\[
		(P\cM P, E) = \left(\overset{e_{n-1, n-1}, e_{nn}}{M_{2}(\C)}, E_{1}\right) \underset{D'}{\Asterisk} (P\cN P, E).
	\]

\end{prop}

\begin{proof}
It is straightforward to see that $P\cM P$ is generated by $\overset{e_{n-1, n-1}, e_{nn}}{M_{2}(\C)}$ and $P\cN P$.  To establish freeness, let $w$ be an alternating word consisting of expectationless elements in $\overset{e_{n-1, n-1}, e_{nn}}{M_{2}(\C)}$ and $P\cN P$.  Since all expectationless elements of $\overset{e_{n-1, n-1}, e_{nn}}{M_{2}(\C)}$ are off-diagonal matrices, we may assume (after taking linear combinations) that $w$ is an alternating word in $\{e_{n-1, n}, e_{n, n-1}\}$ and $(P\cN P)^{\circ \circ}$. Here $C^{\circ \circ}$ denotes the elements of $C$ with zero expectation.   

By Kaplansky density, we may assume that every element in $(P\cN P)^{\circ \circ}$ is of the form $PxP$ where $x$ is an alternating word in  $\left( \underset{\alpha_{1}, \dots, \alpha_{n}}{\overset{e_{11}, \cdots, e_{n-1, n-1}}{M_{n-1}(\C)}} \oplus \underset{\alpha_{n}}{\overset{e_{nn}}{\C}} , E_{1}\right) ^{\circ \circ}$ and $A^{\circ \circ}$.  Again, by taking linear combinations, we can assume that every element of $\left( \underset{\alpha_{1}, \dots, \alpha_{n}}{\overset{e_{11}, \cdots, e_{n-1, n-1}}{M_{n-1}(\C)}} \oplus \underset{\alpha_{n}}{\overset{e_{nn}}{\C}} , E_{1}\right) ^{\circ \circ}$ appearing in $x$ is in $\{e_{i, j} : i \neq j \text{ and } 1 \leq i, j \leq n-1\}$.  Regrouping shows that $w$ is a linear combination of alternating words in $\{e_{i, j}: i \neq j\}$ and $A^{\circ \circ}$ which has expectation 0 by freeness.
\end{proof}

We now compute $M_{n}(\C) * [\C \oplus \C]$ via an induction argument with the base case covered by the previous subsection. As in the previous subsection, we will understand this free product as a particular subalgebra of a free graph von Neumann algebra. By concatenating a pair of edges in the graph, we can witness this subalgebra as a \emph{corner} of a different free graph von Neumann algebra, and hence can determine its isomorphism class.

\begin{thm}\label{thm:induction}
Let  $n \geq 2$ and $1 > \alpha_{1} \geq \alpha_{2} \geq \cdots \geq \alpha_{n} > 0$ with $\sum_{i=1}^{n} \alpha_{i} = 1$ and at least one strict inequality amongst the $\alpha$'s.  Let $\beta \in [\frac12, 1)$, and let $\gamma = \sum_{i=1}^{n} \frac{1-\beta}{\alpha_{i}}$.  Let $H$ be the multiplicative subgroup of $\R^{+}$ generated by $\{\alpha_{i}/\alpha_{j}\, : \, 1 \leq i, j \leq n\}.$ Then:
\begin{enumerate}[(1)]

\item 
	\[
		\underset{\alpha_{1}, \cdots, \alpha_{n}}{\overset{e_{1,1}, \cdots, e_{n,n}}{M_{n}(\C)}} * \left[ \underset{\beta}{\overset{p}{\C}} \oplus \underset{1 - \beta}{\C}\right] \cong \begin{cases}
(T_{H}, \vphi_{H}) &\text{ if } \gamma \geq 1\\
			\underset{\gamma}{(T_{H}, \vphi_{H})} \oplus \underset{\gamma_{1}, \cdots, \gamma_{n}}{\overset{\overline{e_{1,1}}, \cdots, \overline{e_{n,n}}}{M_{n}(\C)}} &\text{ if } \gamma < 1
		\end{cases}.
	\]
In the case when $\gamma < 1$, $\gamma_{i} = \alpha_{i}(1 - \gamma)$, and $\overline{e_{i,i}} \leq e_{i,i} \wedge p$.

\item  Assume that  $\gamma \geq 1$. Suppose $(A, \phi)$ and $(B, \psi)$ are two von Neumann algebras with faithful normal states, and that there exist modular inclusions $\underset{\alpha_{1}, \cdots, \alpha_{n}}{M_{n}(\C)} \hookrightarrow A$ and $(\underset{\beta}{\C} \oplus \underset{1-\beta}{\C}) \hookrightarrow B
$. Then
$$
(A, \phi) * (B, \psi) \cong (A, \phi) * (B, \psi)  * (T_{H}, \vphi_{H}).
$$

\end{enumerate}

\end{thm}

\begin{proof}
We prove this by induction on $n$, with the base case ($n=2$) handled by Lemma~\ref{lem:Hou}, or Theorem~\ref{thm:M2} and Corollary~\ref{cor:Houdayer}. Set $(\cM,\vphi)$ to be the free product in (1). Consider the following von Neumann subalgebra of $\cM$:
	\[
		(\cN, \vphi) := \begin{cases}
\left[\underset{\alpha_{1}, \dots, \alpha_{n-1}}{\overset{e_{1,1}, \cdots, e_{n-1, n-1}}{M_{n-1}(\C)}} \oplus \underset{\alpha_{n}}{\overset{e_{n,n}}{\C}} \right] *  \left[ \underset{\beta}{\overset{p}{\C}} \oplus \underset{1 - \beta}{\C}\right] &\text{ if } \alpha_{1} > \alpha_{2}\\
						&\\
					\left[\underset{\alpha_{1}}{\overset{e_{1,1}}{\C}} \oplus \underset{\alpha_{2}, \dots, \alpha_{n}}{\overset{e_{2,2}, \cdots, e_{n, n}}{M_{n-1}(\C)}}   \right] *  \left[ \underset{\beta}{\overset{p}{\C}} \oplus \underset{1 - \beta}{\C}\right] &\text{ if } \alpha_{1} = \alpha_{2}
				\end{cases}
	\]
We will compute $(\cM,\vphi)$ by way of computing $(\cN,\vphi)$.  

We will prove the inductive step when $\alpha_{1} > \alpha_{2}$, then sketch the (minor) modifications one must take into consideration when $\alpha_{1} = \alpha_{2}$.  Set $Q = e_{1,1} + \cdots + e_{n-1, n-1}$, and 
	\[
		(\cN_{0},\vphi) := \left[\underset{\alpha_{1} + \cdots + \alpha_{n-1}}{\overset{Q}{\C}} \oplus \underset{\alpha_{n}}{\overset{e_{n,n}}{\C}}\right] *  \left[ \underset{\beta}{\overset{p}{\C}} \oplus \underset{1 - \beta}{\C}\right].
	\]
From Lemma \ref{lem:Dykema}, we have
	\[
		(Q\cN Q, \vphi^Q) \cong (Q\cN_{0} Q, \vphi^Q) * \underset{\alpha_{1}, \dots, \alpha_{n-1}}{\overset{e_{1,1}, \cdots, e_{n-1, n-1}}{M_{n-1}(\C)}} .
	\]
To finish computing $(N,\varphi)$, we consider two cases:\\

\noindent
\underline{\textbf{Case 1:}} Assume $\alpha_{1} + \cdots + \alpha_{n-1} \geq \beta$. In this case, $Q$ has full central support in $\cN_{0}$ (hence in $\cN^{\vphi}$), and computing $Q\cN_0 Q$ yields 
	\[
		(Q\cN Q,\varphi^Q) =     \left[\underset{\alpha_{1} +\cdots + \alpha_{n-1} - (1-\beta)}{\overset{Q \wedge p}{\C}} \oplus \underset{\alpha_{n}}{(L(\Z),\tau)} \oplus \underset{\alpha_{1} +\cdots + \alpha_{n-1} - \beta}{\overset{Q \wedge (1-p)}{\C}}\right] * \underset{\alpha_{1}, \dots, \alpha_{n-1}}{\overset{e_{1,1}, \cdots, e_{n-1, n-1}}{M_{n-1}(\C)}}. 
	\]
Note that the left factor always accepts a trace-preserving inclusion of $\underset{\alpha_{1} +\cdots + \alpha_{n-1} - (1-\beta)}{\C}\oplus\underset{1-\beta}{\C}$.  If $1 - \beta > \alpha_{1} +\cdots + \alpha_{n-1} - (1-\beta)$, then it follows by appropriately partitioning the identity of $L(\Z)$ that the left factor will accept a modular inclusion of $\underset{\tilde{\beta}}{\C} \oplus \underset{\tilde{\beta}}{\C}$ where $\tilde{\beta} = \frac{1}{2}(\alpha_{1} + \cdots + \alpha_{n-1})$, and it will follow from the inductive hypothesis of (2) above, as well as free absorption that
	\[
		(Q\cN Q,\vphi^Q) \cong (T_{H'}, \vphi_{H'}),
	\]
where $H'$ the multiplicative subgroup of $\R^{+}$ generated by $\{\alpha_{i}/\alpha_{j}\, : \, 1 \leq i, j \leq n-1\}$. We will therefore assume that $1 - \beta \leq \alpha_{1} +\cdots + \alpha_{n-1} - (1-\beta)$.  Let $\gamma' = \sum_{i=1}^{n-1} \frac{1-\beta}{\alpha_{i}}$ and let $H'$ be as before. If $\gamma' \geq 1$, then it follows from the inductive hypothesis of (2) above, as well as free absorption that
	\[
		(Q\cN Q,\vphi^Q) \cong (T_{H'}, \vphi_{H'}).
	\]
If $\gamma' < 1$, then from the inductive hypothesis of (1) we have
	\[
		\left(\underset{\alpha_{1} +\cdots + \alpha_{n-1} - (1-\beta)}{\overset{Q \wedge p}{\C}} \oplus \underset{1-\beta}{\C}\right) * \underset{\alpha_{1}, \dots, \alpha_{n-1}}{\overset{e_{1,1}, \cdots, e_{n-1, n-1}}{M_{n-1}(\C)}} \cong (T_{H'}, \vphi_{H'}) \oplus \underset{\gamma'_{1}, \cdots, \gamma'_{n-1}}{\overset{q_{1}, \cdots, q_{n-1}}{M_{n-1}(\C)}}
	\]
where $\gamma'_{i} = \alpha_{i}(1 - \gamma')$ and $q_{i} \leq e_{i,i}\wedge p \leq Q\wedge p $.  It follows from this, Lemmas \ref{lem:Dykema} and \ref{lem:antman}, and free absorption that
	\begin{align*}
		\left((Q - Q\wedge p)\cN(Q - Q\wedge p), \vphi^{Q- Q\wedge p}\right) &\cong \left((L(\Z),\tau) \oplus \underset{\alpha_{1} +\cdots + \alpha_{n-1} - \beta}{\overset{Q \wedge (1-p)}{\C}}\right) * (T_{H'}, \vphi_{H'})\\
&\cong (T_{H'}, \vphi_{H'})
	\end{align*}
From the central support statement in Lemma \ref{lem:Dykema} and Lemma \ref{lem:antman} we have
	\[
		(Q\cN Q, \vphi^Q) \cong (T_{H'}, \vphi_{H'}) \oplus \underset{\gamma'_{1}, \cdots, \gamma'_{n-1}}{\overset{q_{1}, \cdots, q_{n-1}}{M_{n-1}(\C)}}.
	\]
Since $Q$ has full central support in $\cN^\vphi$ in either case, applying Lemma \ref{lem:antman} again gives
	\[
		(\cN,\vphi) \cong \begin{cases}
(T_{H'}, \vphi_{H'}) &\text{ if } \gamma' \geq 1\\
 (T_{H'}, \vphi_{H'}) \oplus \underset{\gamma'_{1}, \cdots, \gamma'_{n-1}}{\overset{q_{1}, \cdots, q_{n-1}}{M_{n-1}(\C)}} &\text{ if } \gamma' < 1
			\end{cases}.
	\]\\

\noindent
\underline{\textbf{Case 2:}} Assume $\alpha_{1} + \cdots + \alpha_{n-1} < \beta$. In this case, $1 - z(Q:\cN_{0})$ is nonzero and minimal in $\cN_{0}$, so it follows that  $1 - z(Q:\cN)$ is nonzero and minimal in $\cN$.  Furthermore, it is easy to see that $1 - z(Q, \cN_{0}) = p \wedge e_{n,n}$.  The free product for $Q\cN Q$ becomes
	\[
		(Q\cN Q,\varphi^Q) =   \left[\underset{\alpha_{1} +\cdots + \alpha_{n-1} - (1-\beta)}{\overset{Q \wedge p}{\C}} \oplus \underset{1-\beta}{(L(\Z),\tau)}\right] * \underset{\alpha_{1}, \dots, \alpha_{n-1}}{\overset{e_{1,1}, \cdots, e_{n-1, n-1}}{M_{n-1}(\C)}}. 
	\]
Letting $\gamma' = \sum_{i=1}^{n-1} \frac{1-\beta}{\alpha_{i}}$, $\gamma_{i}' = \alpha_{i}(1 - \gamma')$ and arguing just as in Case 1, we see that 
	\[
		(Q\cN Q,\vphi^Q) \cong \begin{cases}
						(T_{H'}, \vphi_{H'}) &\text{ if } \gamma' \geq 1\\
						(T_{H'}, \vphi_{H'}) \oplus \underset{\gamma'_{1}, \cdots, \gamma'_{n-1}}{\overset{q_{1}, \cdots, q_{n-1}}{M_{n-1}(\C)}} &\text{ if } \gamma' < 1
					\end{cases}. 
	\]
Using Lemma \ref{lem:antman}, it follows that
\begin{align}\label{eqn:Biden}
		(\cN,\vphi) \cong \begin{cases}
					(T_{H'}, \vphi_{H'}) \oplus \underset{\alpha_{n}(1 - \frac{1-\beta}{\alpha_{n}})}{\C} &\text{ if } \gamma' \geq 1\\
					(T_{H'}, \vphi_{H'}) \oplus \underset{\gamma'_{1}, \cdots, \gamma'_{n-1}}{\overset{q_{1}, \cdots, q_{n-1}}{M_{n-1}(\C)}} \oplus \underset{\alpha_{n}(1 - \frac{1-\beta}{\alpha_{n}})}{\C}  &\text{ if } \gamma' < 1
				\end{cases}.
\end{align}

We now proceed using our computations of $(\cN, \vphi)$ from Cases 1 and 2 above.  Set $P = e_{n-1, n-1} + e_{n, n}$.  Let $D = \underset{\alpha_{1}}{\overset{e_{1,1}}{\C}} \oplus \cdots \oplus \underset{\alpha_{n}}{\overset{e_{n,n}}{\C}}$,  and let
	\begin{align*}
		E_{1}&\colon \underset{\alpha_{1}, \dots, \alpha_{n-1}}{\overset{e_{1,1}, \cdots, e_{n-1, n-1}}{M_{n-1}(\C)}} \oplus \underset{\alpha_{n}}{\overset{e_{n,n}}{\C}} \rightarrow D\\
		E_{2}&\colon \left[\underset{\beta}{\C} \oplus \underset{1-\beta}{\C}\right] * D\to D
	\end{align*} 
be the canonical state-preserving conditional expectations. By Proposition \ref{prop:amalgamate}, 
	\[
		(\cN,E) \cong \left(\underset{\alpha_{1}, \dots, \alpha_{n-1}}{\overset{e_{1,1}, \cdots, e_{n-1, n-1}}{M_{n-1}(\C)}} \oplus \underset{\alpha_{n}}{\overset{e_{n,n}}{\C}}, E_{1} \right) \underset{D}{\Asterisk}  \left( \left[\underset{\beta}{\overset{p}{\C}} \oplus \underset{1 - \beta}{\C}\right] * D, E_{2}\right),
	\]
and
	\[
		(\cM,E) \cong \left(\underset{\alpha_{1}, \cdots, \alpha_{n}}{\overset{e_{1,1}, \cdots, e_{n,n}}{M_{n}(\C)}}, E_{1}\right) \underset{D}{\Asterisk}  \left(\left[\underset{\beta}{\overset{p}{\C}} \oplus \underset{1 - \beta}{\C}\right] * D, E_{2}\right).
	\]
It follows from Proposition \ref{prop:createfree} that
	\[
		(P\cM P,E) = \left(\underset{\alpha_{n-1}, \alpha_{n}}{\overset{e_{n-1, n-1}, e_{n,n}}{M_{2}(\C)}}, E_{1}\right) \underset{PDP}{\Asterisk} (P\cN P, E).
	\]
Note that from  Equation (\ref{eqn:Biden}) and Lemma~\ref{lem:antman},
	\[
		(P\cN P,\varphi^P) \cong \begin{cases}
(T_{H'}, \vphi_{H'}) & \text{ if } \alpha_{1} + \cdots + \alpha_{n-1} \geq \beta, \text{ and } \gamma' \geq 1\\
(T_{H'}, \vphi_{H'}) \oplus \underset{\gamma_{n-1}'}{\overset{q_{n-1}}{\C}} & \text{ if } \alpha_{1} + \cdots + \alpha_{n-1} \geq \beta, \text{ and } \gamma' < 1\\
(T_{H'}, \vphi_{H'}) \oplus \underset{\alpha_{n}(1 - \frac{1-\beta}{\alpha_{n}})}{\overset{q_{n}}{\C}} & \text{ if } \alpha_{1} + \cdots + \alpha_{n-1} < \beta, \text{ and } \gamma' \geq 1\\
(T_{H'}, \vphi_{H'}) \oplus \underset{\gamma_{n-1}'}{\overset{q_{n-1}}{\C}} \oplus \underset{\alpha_{n}(1 - \frac{1-\beta}{\alpha_{n}})}{\overset{q_{n}}{\C}} & \text{ if } \alpha_{1} + \cdots + \alpha_{n-1} < \beta, \text{ and } \gamma' < 1\\
\end{cases}
	\]
with $q_{i} \leq e_{i,i}$. Let $\Gamma=(V,E)$ be the following graph with edge weighting $\mu$:
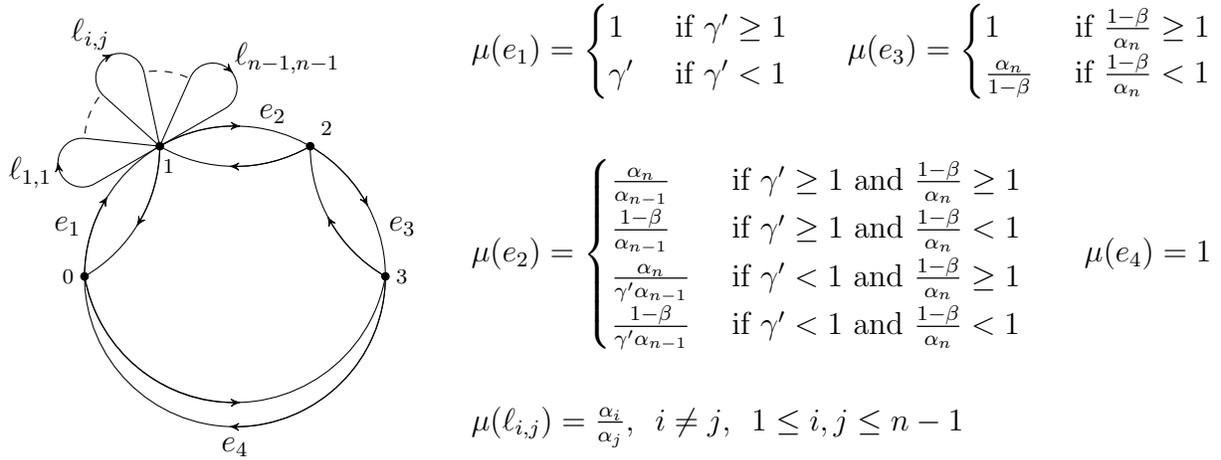
	\[
		\begin{tikzpicture}
	
		\draw (0,0) circle (2);
		
		\draw[fill=black] (-2,0) circle (0.05);
		\node[left] at (-2,0) {\scriptsize$0$};
		
		\draw[fill=black] (-1,1.732) circle (0.05);
		\node[below right] at (-1.1,1.7) {\scriptsize$1$};
		
		\draw[fill=black] (1,1.732) circle (0.05);
		\node[above right] at (1,1.732) {\scriptsize$2$};
		
		\draw[fill=black] (2,0) circle (0.05);
		\node[right] at (2,0) {\scriptsize$3$};
		
		\draw[->] (-2,0) arc (180:148:2);	
		\draw (-2,0) arc (-60:0:2);
		\draw[->] (-1,1.732) arc (0:-32:2);
		\node[above left] at (-1.9,0.4) {$e_1$};
		
		\draw[->] (-1,1.732) arc (120:88:2);	
		\draw (1,1.732) arc (-60:-120:2);
		\draw[->] (1,1.732) arc (-60:-92:2);
		\node[above] at (0.5,1.9) {$e_2$};

		\draw[->] (1,1.732) arc (60:28:2);	
		\draw (2,0) arc (240:180:2); 
		\draw[->] (2,0) arc (240:208:2);
		\node[above right] at (1.9,0.4) {$e_3$};
		
		\draw[->] (2,0) arc (0:-92:2);	
		\draw (2,0) arc (-10:-170:2.03);
		\draw[->] (-2,0) arc (-170:-88:2.03);
		\node[below] at (0,-2) {$e_4$};

		\draw (-1,1.732) --++ (-0.866,-0.5);	
		\draw (-1.866,1.232) arc (-60:-276:0.325);
		\draw (-1, 1.732) --++ (-0.995,0.105);
		\draw[->] (-2.346,1.446) -- (-2.355,1.515);
		\node at (-2.72, 1.38) {$\ell_{1,1}$};

		\draw[dashed] (-1.995,1.9) arc (174:138:1);
		
		\draw (-1,1.732) --++ (-0.743,0.669);	
		\draw (-1.743,2.401) arc (228:12:0.325);
		\draw (-1,1.732) --++ (-0.208,0.978);
		\draw[->] (-1.688,2.924) --++ (0.0577,0.03);
		\node at (-1.92,3.22) {$\ell_{i,j}$};	
		
		\draw[dashed] (-0.63,2.646) arc (66:102:1);
		
		\draw (-1,1.732) --++ (0.866, 0.5);	
		\draw (-0.134,2.232) arc (-60:156:0.325);
		\draw (-1,1.732) --++ (0.407,0.914);
		\draw[->] (-0.079,2.755) --++ (0.0495,-0.0446);
		\node[right] at (-0.15, 2.9) {$\ell_{n-1,n-1}$};
	
		\node[right] at (3,3) {$\mu(e_{1}) = \begin{cases} 1 &\text{ if } \gamma' \geq 1\\
										\gamma' &\text{ if } \gamma' < 1 
									\end{cases}$};
		
		\node[right] at (8,3) {$\mu(e_{3}) = \begin{cases} 1 &\text{ if } \frac{1-\beta}{\alpha_{n}} \geq 1\\
										\frac{\alpha_{n}}{1-\beta} &\text{ if } \frac{1-\beta}{\alpha_{n}} < 1 
									\end{cases}$};
									
		\node[right] at (3,0.3) {$\mu(e_{2}) = \begin{cases} \frac{\alpha_{n}}{\alpha_{n-1}} &\text{ if }  \gamma' \geq 1 \text{ and } \frac{1-\beta}{\alpha_{n}} \geq 1\\
										\frac{1-\beta}{\alpha_{n-1}} &\text{ if }  \gamma' \geq 1 \text{ and } \frac{1-\beta}{\alpha_{n}} < 1\\
										\frac{\alpha_{n}}{\gamma' \alpha_{n-1}} &\text{ if }  \gamma' < 1 \text{ and } \frac{1-\beta}{\alpha_{n}} \geq 1\\
										\frac{1-\beta}{\gamma'\alpha_{n-1}} &\text{ if }  \gamma' < 1 \text{ and } \frac{1-\beta}{\alpha_{n}} < 1\\
									\end{cases}$};

		\node[right] at (11.15,0.3) {$\mu(e_4)=1$};
		
		\node[right] at (3,-2) {$\mu(\ell_{i,j})=\frac{\alpha_i}{\alpha_j}$, $\ i\neq j$, $\ 1\leq i,j\leq n-1$};
	
		\end{tikzpicture}
	\]
Notice that $\mu(e_1)\mu(e_2)\mu(e_3)\mu(e_4) = \alpha_{n}/\alpha_{n-1}$. We assign $\phi$ to $\cM(\Gamma, \mu)$ by choosing $\Gamma_{\Tr}$ so that it contains $e_1,e_2,e_3$ and declaring that $\phi(p_{0}) = \alpha_{n-1}$.  This forces $\phi(p_{3}) = \alpha_{n}$.

Let $\Gamma_0$ be the subgraph of $\Gamma$ obtained by deleting the edges $e_{4}$ and $e_{4}^{op}$.  It follows from Subsection \ref{subsec:graph} that
	\begin{align*}
		( (p_{0} + p_{3})&\cM(\Gamma_0, \mu)(p_{0} + p_{3}), \phi^{p_0+p_3} )\\
			& \cong \begin{cases}
						(T_{H'}, \vphi_{H'}) & \text{ if } \alpha_{1} + \cdots + \alpha_{n-1} \geq \beta, \text{ and } \gamma' \geq 1\\
						(T_{H'}, \vphi_{H'}) \oplus \underset{\gamma_{n-1}'}{\overset{q'}{\C}} & \text{ if } \alpha_{1} + \cdots + \alpha_{n-1} \geq \beta, \text{ and } \gamma' < 1\\
						(T_{H'}, \vphi_{H'}) \oplus \underset{\alpha_{n}(1 - \frac{1-\beta}{\alpha_{n}})}{\overset{q''}{\C}} & \text{ if } \alpha_{1} + \cdots + \alpha_{n-1} < \beta, \text{ and } \gamma' \geq 1\\
						(T_{H'}, \vphi_{H'}) \oplus \underset{\gamma_{n-1}'}{\overset{q'}{\C}}  \oplus \underset{\alpha_{n}(1 - \frac{1-\beta}{\alpha_{n}})}{\overset{q''}{\C}}& \text{ if } \alpha_{1} + \cdots + \alpha_{n-1} < \beta, \text{ and } \gamma' < 1\\
					\end{cases},
	\end{align*}
where $q' \leq p_{0}$ and $q'' \leq p_{3}$.  We immediately see that there is a state-preserving isomorphism
	\[
		(P\cN P,\vphi^P) \cong \left((p_{0} + p_{3})\cM(\Gamma_0, \mu)(p_{0} + p_{3}), \phi^{p_0+p_3}\right)
	\]
sending $e_{n-1, n-1}$ to $p_{0}$ and $e_{n,n}$ to $p_{3}$. Let $D' = \overset{p_0}{\bbC}\oplus\overset{p_3}{\bbC}$, and let $u_{4}$ be the polar part of $Y_{e_4}$.  Note that $u_{4}u_{4}^{*} = p_{3}$ and $u_{4}^{*}u_{4} = p_{0}$.  It follows that 
	\[
		(P\cM P,\vphi^P) \cong \left((p_{0} + p_{3})\cM(\Gamma_0, \mu)(p_{0} + p_{3}), \phi^{p_0+p_3}\right) \Asterisk_{D'} \underset{\alpha_{n-1}, \alpha_{n}}{\overset{p_{0}, p_{3}}{M_{2}(\C)}}
	\]
where $u_{4}$ is in the copy of $M_{2}(\C)$.  From the geometry of the graph and the fact that  $u_{4}u_{4}^{*} = p_{3}$ and $u_{4}^{*}u_{4} = p_{0}$,
	\[
		(e_{n,n}\cM e_{n,n},\vphi^{e_{n,n}}) \cong  \left( p_{3}W^{*}(u_{4}Y_{e_1}, Y_{e_2}, Y_{e_3}, (Y_{\ell_{i, j}})_{i, j}, p_{0}, p_{1}, p_{2}, p_{3})p_{3}, \phi^{p_3}\right).
	\]
The element $u_{4}Y_{e_1}$ has right support $p_{2}$.  As for the other support, note that from the distribution of $Y_{e_1}$, we see that
	\[
		\left(p_{3}W^{*}(u_{4}Y_{e_1}Y_{e_1}^{*}u_{4}^{*})p_{3},\phi^{p_3}\right) = \begin{cases}
													\underset{\alpha_{n}}{(L(\Z),\tau)} &\text{ if } \gamma' \geq 1\\
													\underset{\alpha_{n}\gamma'}{(L(\Z),\tau)} \oplus \underset{\alpha_{n}(1 - \gamma')}{\C} &\text{ if } \gamma' < 1
												\end{cases}.
	\]
Hence the left support $u_4 Y_{e_1}$ is $p_3$ if $\gamma'\geq 1$, and otherwise is a projection of mass $\alpha_n\gamma'$ under $p_3$.
	
Let $\Gamma'=(E',V')$ be the following graph with edge weighting $\mu'$:
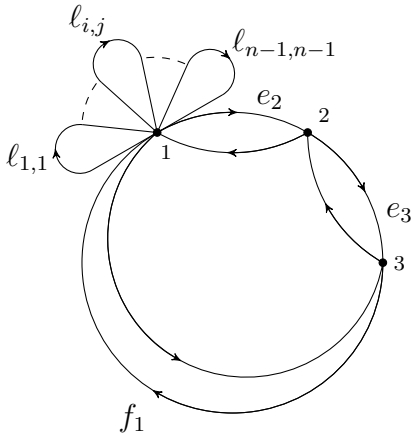
	\[
		\begin{tikzpicture}
	
		\draw (0,0) circle (2);
		
		\draw[fill=black] (-1,1.732) circle (0.05);
		\node[below right] at (-1.1,1.7) {\scriptsize$1$};
		
		\draw[fill=black] (1,1.732) circle (0.05);
		\node[above right] at (1,1.732) {\scriptsize$2$};
		
		\draw[fill=black] (2,0) circle (0.05);
		\node[right] at (2,0) {\scriptsize$3$};
		
		\draw[->] (2,0) arc (0:-122:2);	
		\draw (-1,1.732) arc (130:350:1.843);
		\draw[->] (-1,1.732) arc (130:242:1.843);
		\node[below left] at (-1,-1.732) {$f_1$};
		
		\draw[->] (-1,1.732) arc (120:88:2);	
		\draw (1,1.732) arc (-60:-120:2);
		\draw[->] (1,1.732) arc (-60:-92:2);
		\node[above] at (0.5,1.9) {$e_2$};

		\draw[->] (1,1.732) arc (60:28:2);	
		\draw (2,0) arc (240:180:2); 
		\draw[->] (2,0) arc (240:208:2);
		\node[above right] at (1.9,0.4) {$e_3$};		
		
		\draw (-1,1.732) --++ (-0.866,-0.5);	
		\draw (-1.866,1.232) arc (-60:-276:0.325);
		\draw (-1, 1.732) --++ (-0.995,0.105);
		\draw[->] (-2.346,1.446) -- (-2.355,1.515);
		\node at (-2.72, 1.38) {$\ell_{1,1}$};

		\draw[dashed] (-1.995,1.9) arc (174:138:1);
		
		\draw (-1,1.732) --++ (-0.743,0.669);	
		\draw (-1.743,2.401) arc (228:12:0.325);
		\draw (-1,1.732) --++ (-0.208,0.978);
		\draw[->] (-1.688,2.924) --++ (0.0577,0.03);
		\node at (-1.92,3.22) {$\ell_{i,j}$};	
		
		\draw[dashed] (-0.63,2.646) arc (66:102:1);
		
		\draw (-1,1.732) --++ (0.866, 0.5);	
		\draw (-0.134,2.232) arc (-60:156:0.325);
		\draw (-1,1.732) --++ (0.407,0.914);
		\draw[->] (-0.079,2.755) --++ (0.0495,-0.0446);
		\node[right] at (-0.15, 2.9) {$\ell_{n-1,n-1}$};
	
		\node[right] at (3,1.5) {$\mu'(f_{1}) = \begin{cases} 1 &\text{ if } \gamma' \geq 1\\
										\gamma' &\text{ if } \gamma' < 1 
									\end{cases}$};

		\node[right] at (3,0) {$\mu'(e_{2}) =\mu(e_2) $};
		
		\node[right] at (7,0) {$\mu'(e_3)=\mu(e_3)$};
		
		\node[right] at (3,-1) {$\mu'(\ell_{i,j})=\mu(\ell_{i,j})=\frac{\alpha_i}{\alpha_j}$, $\ i\neq j$, $\ 1\leq i,j\leq n-1$};
	
		\end{tikzpicture}
	\]
We assign $\phi'$ to $\cM(\Gamma', \mu')$ by choosing $\Gamma_{\Tr}$ so that it contains $e_2,e_3$ and declaring that $\phi'(p_{3}) = \alpha_{n}$. It follows that
	\begin{align*}
		\left( (p_1+p_3) W^*(Y_{f_1}, p_1,p_3) (p_1+p_3), (\phi')^{p_1+p_3}\right) &\cong \left( (p_1+p_3) W^*(u_4 Y_{e_1}, p_1, p_3) (p_1+p_3), \phi^{p_1+p_3}\right)\\
			(Y_{f_1}, p_1,p_3)&\mapsto (u_4 Y_{e_1}, p_1,p_3),
	\end{align*}
and in particular, this mapping preserves the canonical conditional expectations onto $\overset{p_1}{\bbC}\oplus \overset{p_3}{\bbC}$. Consequently,
	\begin{align*}
		(p_3 \cM(\Gamma', \mu') p_3, (\vphi')^{p_3}) &\cong (p_{3}W^{*}(u_{4}Y_{e_1}, Y_{e_2}, Y_{e_3}, (Y_{\ell_{i, j}})_{i, j}, p_{0}, p_{1}, p_{2}, p_{3})p_{3}, \vphi^{p_3}\\
				&\cong (e_{n,n}\cM e_{n,n},\vphi^{e_{n,n}}).
	\end{align*}
Recall that $H$ is the multiplicative subgroup of $\R^{+}$ generated by $\{\alpha_{i}/\alpha_{j}: 1 \leq i, j \leq n \}$ and note that from our definition of $\gamma$ above, $\gamma = \gamma' + \frac{1-\beta}{\alpha_{n}}$. It follows from Subsection \ref{subsec:graph} that 
	\[
		(e_{n,n}\cM e_{n,n},\vphi^{e_{n,n}})  \cong (p_3\cM(\Gamma', \mu')p_3, (\phi')^{p_3}) \cong \begin{cases}
																				(T_{H}, \vphi_{H}) &\text{ if } \gamma \geq 1\\
																				\underset{\alpha_{n}\gamma}{(T_{H}, \vphi_{H})} \oplus \underset{\alpha_{n}(1 - \gamma)}{\C} &\text{ if } \gamma > 1
																			\end{cases}.
	\]
Tensoring with $\overset{e_{1,1},\ldots, e_{n,n}}{\underset{\alpha_{1}, \cdots, \alpha_{n}}{M_{n}(\C)}}$ and using Corollary \ref{cor:tensor} gives the desired result, and proves (1) in the case that $\alpha_{1} > \alpha_{2}$.

If $\alpha_{1} = \alpha_{2}$, let $\cN$ be defined as at the beginning of the proof. Then $\alpha_{1} < 1/2$, so it follows that $e_{1,1} \wedge (1-q) = 0$.  This means that $e_{1,1}\cN e_{1,1}$ will contain at most one minimal projection.  This observation means that the computation for $P\cN P$ continues exactly as above, and the graphical models for $P\cN P$ and $P\cN P$ are still valid after we redefine $\gamma' = \sum_{i=2}^{n} \frac{1-\beta}{\alpha_i}$ and exchange any mention of $\alpha_{n}$ and $\alpha_{n-1}$ with $\alpha_{1}$ and $\alpha_{2}$ respectively.

The proof of (2) follows directly from the result of (1) as well as the proof of \cite[Lemma~\ref{lem:Hou}(2)]{Hou07}.
\end{proof}

\subsection{Free products of finite-dimensional von Neumann algebras}

We will use Theorem \ref{thm:induction} to establish the requisite base cases/inductive steps for computing the free product of any two finite-dimensional von Neumann algebras.    
  The following three propositions are the ``building blocks" between Theorem~\ref{thm:induction} above, and Theorem~\ref{thm:finitedim} below.

\begin{prop}\label{prop:matrix}
Let $1 > \alpha_{1} \geq \alpha_{2} \geq \cdots \geq \alpha_{n} > 0$ with $\sum_{i=1}^n \alpha_i^n =1$, and let $1 > \beta_{1} \geq \beta_{2} \geq \cdots \geq \beta_{m} > 0$ with $\sum_{j = 1}^{m}\beta_{j} = 1$.  Let
	\[
		H = \langle \{ \alpha_{i}/\alpha_{j}\, : \, 1 \leq i, j \leq n\} \cup \{ \beta_{i}/\beta_{j}\, : \, 1 \leq i, j \leq m\} \rangle < \bbR^+,
	\]
and assume $H$ is not trivial.  Then
	\[
		\underset{\alpha_{1}, \cdots, \alpha_{n}}{M_{n}(\C)} * \underset{\beta_{1}, \cdots \beta_{m}}{M_{m}(\C)} \cong (T_{H}, \vphi_{H}).
	\]
\end{prop} 

\begin{proof}
We will show that Theorem~\ref{thm:induction}.(2) can always be applied.\\

\noindent
\underline{\textbf{Case 1:}} Assume only one of the sets $\{ \alpha_{i}/\alpha_{j}\, : \, 1 \leq i, j \leq n\}$ or $\{ \beta_{i}/\beta_{j}\, : \, 1 \leq i, j \leq m\}$ is non-trivial. Without loss of generality, we can assume it is the former. Consequently,
	\[
		H=\<\{ \alpha_{i}/\alpha_{j}\, : \, 1 \leq i, j \leq n\}\>,
	\]
and $\beta_1=\cdots =\beta_m = \frac1m$. Set $\beta:= \lceil\frac{m}{2}\rceil \frac{1}{m} \in [\frac12, 1)$, so that $\underset{\beta_1,\ldots,\beta_m}{M_m(\bbC)}$ accepts a modular inclusion of $\underset{\beta}{\bbC}\oplus \underset{1-\beta}{\bbC}$. Note that $\alpha_n \leq 1 - \alpha_1$. It follows that
	\begin{align*}
		\gamma := \sum_{i=1}^n \frac{1-\beta}{\alpha_i} \geq (1-\beta)\left[\frac{n-1}{\alpha_1} + \frac{1}{1-\alpha_1}\right] \geq (1-\beta)[ n + 2\sqrt{n(n-1)}],
	\end{align*}
where the last inequality follows from an easy calculus exercise. Now, note that $1-\beta \geq \frac13$ (where this minimum is attained for $m=3$ and the inequalilty is strict otherwise), and since $n\geq 2$ we have $\gamma\geq 1$. Thus we may apply Theorem~\ref{thm:induction}.(2), free absorption, and Proposition \ref{prop:L(Z)B(H)} to obtain
	\begin{align*}
		\underset{\alpha_{1}, \cdots, \alpha_{n}}{M_{n}(\C)} * \underset{\beta_{1}, \cdots \beta_{m}}{M_{m}(\C)} &\cong (T_H,\vphi_H)*\underset{\beta_{1}, \cdots \beta_{m}}{M_{m}(\C)}*\underset{\alpha_{1}, \cdots, \alpha_{n}}{M_{n}(\C)}\\
			& \cong (T_H,\vphi_H) * \underset{\alpha_{1}, \cdots \alpha_{n}}{M_{n}(\C)}\\
			&\cong (T_H,\vphi_H).   
	\end{align*}\\

\noindent
\underline{\textbf{Case 2:}} Assume both sets $\{ \alpha_{i}/\alpha_{j}\, : \, 1 \leq i, j \leq n\}$ and $\{ \beta_{i}/\beta_{j}\, : \, 1 \leq i, j \leq m\}$ are non-trivial. Now, either $\frac{\beta_m}{\alpha_n} \geq 1$ or $\frac{\alpha_n}{\beta_m} \geq 1$. Without loss of generality, assume the former. Let $k\in \{1,\ldots,m\}$ be the smallest index such that $\beta_1+\cdots +\beta_k \geq \frac12$. Then $k \leq m-1$ since $\beta_m \leq \frac1m\leq \frac 12$, which implies $\beta_1+\cdots +\beta_{m-1} \geq \frac12$. Set $\beta:= \beta_1+\cdots + \beta_k\in [\frac12,1)$, so that $1-\beta\geq \beta_m$. Then
	\[
		\gamma:= \sum_{i=1}^n \frac{1-\beta}{\alpha_i} \geq \frac{1-\beta}{\alpha_n} \geq \frac{\beta_m}{\alpha_n} \geq 1.
	\]
Letting $H'=\<\{ \alpha_{i}/\alpha_{j}\, : \, 1 \leq i, j \leq n\}\>$, we may apply Theorem~\ref{thm:induction}.(2), free absorption, and Proposition~\ref{prop:L(Z)B(H)} to obtain
	\begin{align*}
		\underset{\alpha_{1}, \cdots, \alpha_{n}}{M_{n}(\C)} * \underset{\beta_{1}, \cdots \beta_{m}}{M_{m}(\C)} &\cong (T_{H',}\vphi_{H'}) *\underset{\alpha_{1}, \cdots, \alpha_{n}}{M_{n}(\C)} * \underset{\beta_{1}, \cdots \beta_{m}}{M_{m}(\C)} \\
		&\cong (T_{H',}\vphi_{H'}) * \underset{\beta_{1}, \cdots \beta_{m}}{M_{m}(\C)}\\
		&\cong (T_H, \vphi_H).\qedhere
	\end{align*}

\end{proof}

\begin{prop}\label{prop:matrixabelian}
Let $1 > \alpha_{1} \geq \alpha_{2} \geq \cdots \geq \alpha_{n} > 0$ with $\sum_{i=1}^{n} \alpha_{i} = 1$, and let $1> \beta_{1} \geq \beta_{2} \geq \cdots \geq \beta_{m}>0$ with $\sum_{j=1}^{m} \beta_{j} = 1$.  Let $\gamma = \sum_{i=1}^{n} \frac{1-\beta_{1}}{\alpha_{i}}$.  Suppose $H = \<\{ \alpha_{i}/\alpha_{j}\, : \, 1 \leq i, j \leq n\}\>$. Then
	\[
		\underset{\alpha_{1}, \cdots, \alpha_{n}}{\overset{e_{1,1}, \cdots, e_{n,n}}{M_{n}(\C)}} * \left[ \underset{\beta_{1}}{\overset{p}{\C}} \oplus \underset{\beta_{2}}{\C} \oplus \cdots \oplus\underset{\beta_{m}}{\C} \right] \cong \begin{cases}
								(A, \vphi) &\text{ if } \gamma \geq 1\\
								\underset{\gamma}{(A, \vphi)} \oplus \underset{\gamma_{1}, \cdots, \gamma_{n}}{\overset{\overline{e_{1,1}}, \cdots, \overline{e_{n,n}}}{M_{n}(\C)}} &\text{ if } \gamma < 1
							\end{cases},
	\]
where $(A, \vphi)$ is an interpolated free group factor if $\alpha_{1} = \cdots = \alpha_{n}$ and is $(T_{H}, \vphi_{H})$ if at least one inequality in the $\alpha$'s is strict. In the case when $\gamma < 1$, $\gamma_{i} = \alpha_{i}(1 - \gamma)$, and $\overline{e_{i,i}} \leq e_{i,i} \wedge p$.
\end{prop}

\begin{proof}

The case when $\alpha_{1} = \cdots = \alpha_{n}$ is handled by \cite{MR1201693}.  We therefore assume that at least one of the inequalities are strict. Moreover, note that the case $m = 2$ follows from  Theorem~\ref{thm:induction}.(1). 

If $\gamma \geq 1$, then we claim there is a modular inclusion
	\[
		\underset{\beta}{\C} \oplus \underset{1-\beta}{\C} \hookrightarrow  \underset{\beta_{1}}{\overset{p}{\C}} \oplus \underset{\beta_{2}}{\C} \oplus \cdots \oplus\underset{\beta_{m}}{\C} ,
	\]
where $\beta \in [\frac12, 1)$ and $\sum_{i=1}^{n} \frac{1-\beta}{\alpha_{i}} \geq 1$.  Indeed, if $\beta_1\geq \frac12$ then simply set $\beta:=\beta_1$. Otherwise, let $k\in \{1,\ldots,m\}$ be the largest index such that $\beta_k+\cdots +\beta_m\geq \frac12$, and set $\beta:=\beta_k+\cdots + \beta_m$. Note that $k\geq 2$, since $\beta_2+\cdots+\beta_m = 1-\beta_1 > \frac12$. Since $\sum_{i=1}^{n} \frac{1}{\alpha_i} \geq n^2 \geq 4$, it suffices to show $\beta\leq \frac34$. If $\beta> \frac34$, then $\beta_1 \leq 1-\beta < \frac14$ so that $\beta_k < \frac14$. But then $\beta_{k+1}+\cdots +\beta_m = \beta- \beta_k > \frac34 - \frac14 =\frac12$, contradicting our choice of $k$. Thus the claim holds. 

From Theorem~\ref{thm:induction}.(2), free absorption, and Proposition~\ref{prop:L(Z)B(H)} we have
\begin{align*}
	\underset{\alpha_{1}, \cdots, \alpha_{n}}{\overset{e_{1,1}, \cdots, e_{n,n}}{M_{n}(\C)}} * \left[ \underset{\beta_{1}}{\overset{p}{\C}} \oplus \underset{\beta_{2}}{\C} \oplus \cdots \oplus\underset{\beta_{m}}{\C} \right] &\cong \underset{\alpha_{1}, \cdots, \alpha_{n}}{\overset{e_{1,1}, \cdots, e_{n,n}}{M_{n}(\C)}} * \left[ \underset{\beta_{1}}{\overset{p}{\C}} \oplus \underset{\beta_{2}}{\C} \oplus \cdots \oplus\underset{\beta_{m}}{\C} \right] * (T_{H}, \vphi_{H})\\
		&\cong  \underset{\alpha_{1}, \cdots, \alpha_{n}}{\overset{e_{1,1}, \cdots, e_{n,n}}{M_{n}(\C)}} * (T_{H}, \vphi_{H})\\
		&\cong (T_{H}, \vphi_{H}).
	\end{align*}

Now, if $\gamma  < 1$, let $(M,\vphi)$ be the free product in the statement of the proposition. By Theorem~\ref{thm:induction}.(1), 
	\[
		(N,\vphi):=\underset{\alpha_{1}, \cdots, \alpha_{n}}{\overset{e_{1,1}, \cdots, e_{n,n}}{M_{n}(\C)}} * \left[ \underset{\beta_{1}}{\overset{p}{\C}} \oplus \underset{1-\beta_{1}}{\overset{1-p}{\C}} \right] \cong \underset{\gamma}{(T_{H}, \vphi_{H})} \oplus \underset{\gamma_{1}, \cdots, \gamma_{n}}{\overset{\overline{e_{1,1}}, \cdots, \overline{e_{n,n}}}{M_{n}(\C)}} 
	\]
with $\gamma_{i} = \alpha_{i}(1 - \gamma)$, and $\overline{e_{i,i}} \leq e_{i,i} \wedge p$.  From Lemmas~\ref{lem:Dykema} and \ref{lem:antman},
	\[
		\left((1-p)\cM(1-p), \vphi^{1-p}\right) = ((1-p)\cN(1-p), \vphi^{1-p}) * \left[  \underset{\beta_{2}}{\C} \oplus \cdots \oplus\underset{\beta_{m}}{\C} \right] \cong (T_{H}, \vphi_{H})
	\]
We have $z(1-p\colon \cN^\vphi) = z( 1-p\colon \cN) = z(1-p\colon M)$. As $\cN^{\vphi} \subset \cM^{\vphi}$, it follows that $z(1-p\colon \cM^\vphi)=z(1-p\colon \cN^\vphi)$ and so from Lemma~\ref{lem:antman} we have that 
	\[
		(\cM,\vphi) \cong \underset{\gamma}{(T_{H}, \vphi_{H})} \oplus \underset{\gamma_{1}, \cdots, \gamma_{n}}{\overset{\overline{e_{1,1}}, \cdots, \overline{e_{n,n}}}{M_{n}(\C)}}
	\]
as claimed.
\end{proof}

\begin{prop}\label{prop:matrix1}
Let $1>\alpha_{1} \geq \cdots \geq \alpha_{n}>0$ and $H$ be as in the statement of Proposition~\ref{prop:matrixabelian}.  Let  $1 > \beta \geq 0$.  Let $(B, \phi)$ be either $(L(\F_{t}),\tau)$ for $t \in [1, \infty]$ or $(T_{H'}, \vphi_{H'})$ for some countable, non-trivial $H'$.  Let $\gamma = \sum_{i=1}^{n} \frac{1-\beta}{\alpha_{i}}$.  Then
	\[
		\underset{\alpha_{1}, \cdots, \alpha_{n}}{\overset{e_{1,1}, \cdots, e_{n,n}}{M_{n}(\C)}} * \left[  \underset{\beta}{\C}\oplus \underset{1-\beta}{\overset{q}{(B, \phi)}} \right] \cong
					\begin{cases}
						(A, \vphi) &\text{ if } \gamma \geq 1\\
						\underset{\gamma}{(A, \vphi)} \oplus \underset{\gamma_{1}, \cdots, \gamma_{n}}{\overset{\overline{e_{1,1}}, \dots, \overline{e_{n,n}}}{M_{n}(\C)}} &\text{ if } \gamma < 1
					\end{cases},
	\]
where $\gamma_{i} = \alpha_{i}(1 - \gamma)$, $\overline{e_{i,i}} \leq e_{i,i} \wedge (1-q)$. Here $(A, \varphi) \cong (L(\F_{s}), \tau)$ for some $s>1$ if $H$ is trivial and $B$ is tracial, and  otherwise $(A, \varphi) \cong (T_{G}, \vphi_{G})$ where $G$ is the group generated by $H$ and $H'$.  If $(B,\phi)$ is tracial, we interpret $H'$ to be $\{1\}$.
\end{prop}

\begin{proof}
The case where $H$ and $H'$ are both trivial is handled in \cite{MR1201693}.  We will therefore assume that at least one of $H$ or $H'$ is non-trivial. If $\beta=0$, then this follows from Proposition~\ref{prop:L(Z)B(H)} or Corollary~\ref{cor:ArakiB(H)}. Thus we further assume $\beta>0$.

Let
	\begin{align*}
		(\cM, \vphi) &:= \underset{\alpha_{1}, \cdots, \alpha_{n}}{\overset{e_{1,1}, \cdots, e_{n,n}}{M_{n}(\C)}} * \left[ \underset{\beta}{\C}\oplus \underset{1-\beta}{\overset{q}{(B, \phi)}}\right]\\
		(\cN, \vphi) &:= \underset{\alpha_{1}, \cdots, \alpha_{n}}{\overset{e_{1,1}, \cdots, e_{n,n}}{M_{n}(\C)}} * \left[ \underset{\beta}{\C} \oplus \overset{q}{\underset{1 -\beta}{\C}}\right].
	\end{align*}
Let $\delta:= \sum_{i=1}^n \frac{\beta}{\alpha_i}$. Since $\gamma+\delta  = \sum_{i=1}^n \frac{1}{\alpha_i} \geq 4$, we cannot have $\gamma,\delta<1$. We consider three possible cases:\\

\noindent
\underline{\textbf{Case 1:}} Assume $\gamma ,\delta \geq 1$. In this case, $(\cN,\vphi)$ is either  $(L(\F_{s}),\tau)$ for some $s > 1$ if $H$ is trivial \cite{MR1201693}, or is $(T_{H}, \vphi_{H})$ from Theorem~\ref{thm:induction}.  Applying Lemmas~\ref{lem:Dykema} and \ref{lem:antman} , gives $(q\cM q, \vphi^q) \cong (T_{G}, \vphi_{G})$. Noting that $z(q: \cN^{\vphi}) = 1$ and $\cN^{\vphi}\subset \cM^{\vphi}$, we obtain $z(q\colon \cM^\vphi)=1$. We then apply Lemma~\ref{lem:antman} again to obtain $(\cM,\varphi) \cong (T_{G}, \vphi_{G})$.\\

\noindent
\underline{\textbf{Case 2:}} Assume $\gamma < 1 \leq \delta$. In this case,
	\[
		(\cN,\vphi) \cong \begin{cases}
\underset{\gamma}{\overset{p}{(L(\F_{s}),\tau)}} \oplus \underset{\gamma_{1}, \cdots, \gamma_{n}}{\overset{\overline{e_{1,1}}, \dots, \overline{e_{n,n}}}{M_{n}(\C)}} &\text{ if } H = \{1\}\\
\underset{\gamma}{\overset{p}{(T_{H}, \vphi_{H})}} \oplus \underset{\gamma_{1}, \cdots, \gamma_{n}}{\overset{\overline{e_{1,1}}, \dots, \overline{e_{n,n}}}{M_{n}(\C)}} &\text{ if } H \neq \{1\}
\end{cases}
	\]
for some $s > 1$ and $\overline{e_{i,i}} \leq e_{ii} \wedge (1-q)$.  Applying Lemmas~\ref{lem:Dykema} and \ref{lem:antman} , gives $(q\cM q, \vphi) \cong (T_{G}, \vphi_{G})$. Noting that $z(q: \cN^{\vphi}) = z(q: \cN) = z(q: \cM) = p$, and $\cN^{\vphi} \subset\cM^{\vphi}$, we apply Lemma~\ref{lem:antman} again to obtain $\cM \cong \underset{\gamma}{(T_{G}, \vphi_{G})} \oplus \underset{\gamma_{1}, \cdots, \gamma_{n}}{\overset{\overline{e_{11}}, \dots, \overline{e_{nn}}}{M_{n}(\C)}}$.\\

\noindent
\underline{\textbf{Case 3:}}  Assume $\delta<1\leq \gamma$. In this case,
	\[
		(\cN,\vphi) \cong \begin{cases}
						\overset{p}{\underset{\delta}{(L(F_s),\tau)}} \oplus \underset{\delta_{1}, \cdots, \delta_{n}}{\overset{\overline{e_{1,1}}, \dots, \overline{e_{n,n}}}{M_{n}(\C)}} &\text{ if } H = \{1\}\\
						\overset{p}{\underset{\delta}{(T_{H}, \vphi_{H})}} \oplus \underset{\delta_{1}, \cdots, \delta_{n}}{\overset{ \overline{e_{1,1}}, \dots, \overline{e_{n,n}}}{M_{n}(\C)}} &\text{ if } H \neq \{1\}
					\end{cases}
	\]
for some $s> 1$, $\delta_{i}= \alpha_{i}(1 - \delta)$, and $\overline{e_{i,i}} \leq e_{i,i} \wedge q$. In either case, we have $z(q\colon \cN^\vphi)=1$, and hence $z(q\colon \cM^\vphi)=1$ since $\cN^\vphi\subset \cM^\vphi$.  By Lemma~\ref{lem:Dykema}
	\[
		(q\cM q,\vphi^q)  \cong (q\cN q, \vphi^q) * (B, \phi)
	\]
If $H$ is trivial, then $H'$ is not. So free absorption gives $(q\cM q, \vphi^q) \cong (T_{G}, \vphi_{G})$, and applying Lemma~\ref{lem:antman} as above gives $(\cM, \vphi) \cong (T_{G}, \vphi_{G})$. 

If $H$ is non-trivial, then by Lemma~\ref{lem:antman}
	\[
		(q\cM q,\vphi^q) \cong \left[ \overset{h}{(T_{H}, \vphi_{H})}\oplus \underset{\delta_{1}, \cdots, \delta_{n}}{\overset{\overline{e_{1,1}}, \dots, \overline{e_{n,n}}}{M_{n}(\C)}}\right] * (B, \phi)
	\]
To compute this, we study the folowoing subalgebras of $q\cM q$:
	\begin{align*}
		(\cM_{1},\vphi^q) &= \left[ \overset{h}{\C} \oplus \overset{q - h}{\C} \right] * (B, \phi)\\
		(\cM_{2},\vphi^q) &= \left[ \overset{h}{\C} \oplus \underset{\delta_{1}, \cdots, \delta_{n}}{\overset{\overline{e_{1,1}}, \dots, \overline{e_{n,n}}}{M_{n}(\C)}}\right] * (B, \phi)
	\end{align*}
Note that $(\cM_{1},\vphi^q) \cong (L(\F_{r}),\tau)$ for some $r > 1$ if $H'$ is trivial, and otherwise is $(T_{H'}, \vphi_{H'})$. Applying Lemmas~\ref{lem:Dykema} and \ref{lem:antman}, we compute
	\[
		((q-h)\cM_2 (q-h),\vphi^{q-h}) \cong ( (q-h)\cM_1 (q-h),\vphi^{q-h}) * \underset{\delta_{1}, \cdots, \delta_{n}}{\overset{\overline{e_{1,1}}, \dots, \overline{e_{n,n}}}{M_{n}(\C)}} \cong (T_G,\vphi_G).
	\]
The last isomorphism follows from either Proposition~\ref{prop:L(Z)B(H)} or Corollary~\ref{cor:ArakiB(H)}, depending on whether or not $H'$ is trivial (which dictated the form of $(\cM_1,\vphi^q)$). Now, $z(q-h\colon \cM_1^{\vphi^q})=q$ since $\cM_1^{\vphi^q}$ is always a factor, and $\cM_1^{\vphi^q}\subset \cM_2^{\vphi^q}$. Thus Lemma~\ref{lem:antman} implies $(\cM_2,\vphi^q)\cong (T_G,\vphi_G)$. Applying Lemmas~\ref{lem:Dykema} and \ref{lem:antman} again gives
	\[
		(h\cM h, \vphi^h) \cong (h \cM_2 h ,\vphi^h) * (T_H,\vphi_H) \cong (T_G,\vphi_G).
	\]
Since $z(h\colon \cM_2^{\vphi^q})=q$ and $\cM_2^{\vphi^q}\subset (q \cM q)^{\vphi^q}$, Lemma~\ref{lem:antman} implies $(q\cM q,\vphi^q)\cong (T_G,\vphi_G)$. Finally, recall $z(q\colon \cM^\vphi)=1$ so that $(\cM,  \vphi) \cong (T_{G}, \vphi_{G})$ as desired.
\end{proof}

\begin{prop}\label{prop:matrix2}

Let $1>\alpha_{1} \geq \cdots \geq \alpha_{n}>0$ and $H$ be as in the statement of Proposition~\ref{prop:matrixabelian}.  Let $1>\beta_{1} \geq \cdots \geq \beta_{m} > 0$, $m \geq 2$, and $\beta > 0$ satisfy $\beta+\sum_{j=1}^m \beta_j=1$.  Let $(B, \phi)$ be either $(L(\F_{t}),\tau)$ for $t \in [1, \infty]$ or $(T_{H'}, \vphi_{H'})$ for some non-trivial $H'$.  Let $H''$ be the group generated by $\{ \beta_{i}/\beta_{j} : 1 \leq i, j \leq m\}.$  Let $G$ be the group generated by $H$, $H'$, and $H''$ (where we declare $H'$ to be trivial if $B = L(\F_{t})$).  We have
	\[
		\underset{\alpha_{1}, \cdots, \alpha_{n}}{\overset{e_{1,1}, \cdots, e_{n,n}}{M_{n}(\C)}} * \left[ \underset{\beta_{1}, \cdots, \beta_{m}}{M_{m}(\C)}\oplus \underset{\beta}{\overset{q}{(B, \phi)}}\right] \cong \begin{cases}
			(T_G,\vphi_G) & \text{if $G$ is non-trivial}\\
			(L(\F_{s}),\tau)  &\text{otherwise, for some }s>1
		\end{cases}.
	\]
\end{prop}

\begin{proof}

The case where $H$, $H'$ and $H''$ are all trivial is handled by \cite{MR1201693}, so we can assume that at least one of $H$, $H"$, or $H''$ is non-trivial. Let
	\[
		(\cM,\vphi):=\underset{\alpha_{1}, \cdots, \alpha_{n}}{\overset{e_{1,1}, \cdots, e_{n,n}}{M_{n}(\C)}} * \left[ \underset{\beta_{1}, \cdots, \beta_{m}}{M_{m}(\C)}\oplus \underset{\beta}{\overset{q}{(B, \phi)}}\right].
	\]\\  

\noindent
\underline{\textbf{Case 1:}} Assume $H$ is non-trivial. Denote
	\[
		(C, \psi) := \underset{\beta_{1}, \cdots, \beta_{m}}{M_{m}(\C)}\oplus \underset{\beta}{\overset{q}{(B, \phi)}}.
	\]

\underline{\textbf{Case 1.a:}} Assume $\sum_{i=1}^n \frac{1-\beta_1}{\alpha_i} \geq 1$. We claim there exists $\beta' \in [1/2, 1)$ satisfying $\sum_{i=1}^{n} \frac{1-\beta'}{\alpha_{i}} \geq 1$ and a modular inclusion
	\[
		\underset{\beta'}{\C} \oplus \underset{1-\beta'}{\C} \hookrightarrow (C, \psi).
	\]
Indeed, since $B$ is diffuse, for some $m'\geq 1$ we can split $q$ into orthogonal projections $q_1,\ldots, q_{m'}$ with masses $\psi(q_j)\leq \beta_m$ and then we simply proceed as in the proof of Proposition~\ref{prop:matrixabelian}. Consequently, Theorem~\ref{thm:induction}.(2) and Corollary~\ref{cor:ArakiB(H)} yield
	\[
		(\cM,\vphi) =\underset{\alpha_{1}, \cdots, \alpha_{n}}{\overset{e_{1,1}, \cdots, e_{n,n}}{M_{n}(\C)}} * (C, \psi) \cong (T_{H}, \vphi_{H}) * \underset{\alpha_{1}, \cdots, \alpha_{n}}{\overset{e_{1,1}, \cdots, e_{n,n}}{M_{n}(\C)}} * (C, \psi) \cong (T_{H}, \vphi_{H}) * (C, \psi)
	\]
We compute this latter free product by considering the following subalgebras of $\cM$
	\begin{align*}
 		(\cN_{1},\vphi) &:= (T_{H}, \vphi_{H}) * \left[\underset{1-\beta}{\C}\oplus \underset{\beta}{\overset{q}{\C}}\right]\\
		(\cN_{2},\vphi) &:= (T_{H}, \vphi_{H}) * \left[ \underset{1-\beta}{\C} \oplus \underset{\beta}{\overset{q}{(B, \phi)}}  \right].
	\end{align*}
So by Lemmas~\ref{lem:Dykema} and \ref{lem:antman} we have
	\[
		(q\cN_2q,\vphi^q) \cong (q\cN_1 q, \vphi^q) * (B,\phi) \cong (T_H,\vphi_H) * (B,\phi)\cong  (T_{\<H,H'\>}, \vphi_{\<H,H'\>}),
	\]
where we appeal to free absorption for the last isomorphism if $H'$ is trivial. Since $z(q\colon \cN_2^\vphi) = z(q\colon \cN_1^\vphi)=1$ by virtue of $\cN_1^\vphi\subset \cN_2^\vphi$, it follows from Lemma~\ref{lem:antman} that $(\cN_2,\vphi)\cong (T_{\<H,H'\>}, \vphi_{\<H,H'\>})$. Appealing to the same lemma and using Corollary~\ref{cor:ArakiB(H)} we have
	\begin{align*}
		( (1-q) \cM (1-q), \vphi^{1-q})&\cong ( (1-q) \cN_2 (1-q), \vphi^{(1-q)}) * \underset{\beta_{1}, \cdots, \beta_{m}}{M_{m}(\C)}\\
				&\cong (T_{\<H,H'\>}, \vphi_{\<H,H'\>}) * \underset{\beta_{1}, \cdots, \beta_{m}}{M_{m}(\C)}\\
				& \cong (T_G,\vphi_G).
	\end{align*}
Finally, $z( 1-q\colon \cM) = z( 1-q\colon \cN_2^\vphi)=1$ by virtue of $\cN_2^\vphi \subset \cM^\vphi$, and hence Lemma~\ref{lem:antman} yields the claimed isomorphism class for $(\cM,\vphi)$.

\underline{\textbf{Case 1.b:}} Assume $\sum_{i=1}^{n} \frac{1-\beta_{1}}{\alpha_{i}} < 1$. Then $\sum_{i=1}^n \frac{1- (1-\beta)}{\alpha_i}=\sum_{i=1}^{n} \frac{\beta}{\alpha_{i}} < 1$. Since $\sum_{i=}^n \frac{1}{\alpha_i} \geq 4$, it follows that $\beta<\frac{1}{4}$ and in particular $1-\beta \geq \frac{1}{2}$. Consider the following subalgebras of $\cM$:
	\begin{align*}
		 (\cM_{1},\vphi) &:= \underset{\alpha_{1}, \cdots, \alpha_{n}}{\overset{e_{1,1}, \cdots, e_{n,n}}{M_{n}(\C)}}  * \left[\underset{1-\beta}{\C} \oplus \underset{\beta}{\overset{q}{\C}}  \right]\\
		(\cM_{2},\vphi) &:= \underset{\alpha_{1}, \cdots, \alpha_{n}}{\overset{e_{1,1}, \cdots, e_{n,n}}{M_{n}(\C)}}  * \left[\underset{1-\beta}{\C} \oplus \underset{\beta}{\overset{q}{(B, \phi)}} \right].
\end{align*}
From Theorem~\ref{thm:induction}.(1) we have
	\[
		(\cM_1,\vphi)\cong (T_H,\vphi_H) \oplus \underset{\gamma_{1}, \cdots, \gamma_{n}}{M_{n}(\C)},
	\]
where $\gamma_{i} = \alpha_{i}(1 - \sum_{j=1}^n \frac{1 -(1-\beta)}{\alpha_j})$ and the copy of $M_{n}(\C)$ lies under $1-q$. Using free absorption and tracking the central support of $q$, it follows from Lemmas~\ref{lem:Dykema} and \ref{lem:antman} that
	\[
		(\cM_{2},\vphi) \cong (T_{\<H,H'\>}, \vphi_{\<H,H'\>}) \oplus \underset{\gamma_{1}, \cdots, \gamma_{n}}{M_{n}(\C)},
	\]
Applying Lemmas~\ref{lem:Dykema} and \ref{lem:antman} again, we see that
	\[
		((1-q)\cM(1-q),\vphi^{1-q}) \cong  \left[\underset{1 - \beta - (\gamma_{1} + \cdots + \gamma_{n})}{(T_{\<H,H'\>}, \vphi_{\<H,H'\>})} \oplus \underset{\gamma_{1}, \cdots, \gamma_{n}}{M_{n}(\C)}\right] * \underset{\beta_{1}, \cdots, \beta_{m}}{M_{m}(\C)}.
	\]
Now, the assumption $\sum_{i=1}^n \frac{1-\beta_1}{\alpha_i} <1$ implies that $\beta_{1} > \alpha_{1}\geq \gamma_1$. Indeed, the second inequality is immediate from the definition of $\gamma_1$, and to see the first simply observe  
	\[
		1> \sum_{i=1}^n \frac{1-\beta_1}{\alpha_i} \geq \frac{1-\beta_1}{\alpha_2} \geq \frac{1-\beta_1}{1-\alpha_1}.
	\]
Using the same reasoning, it must therefore be the case that $\sum_{j=1}^m \frac{(1-\beta)-\gamma_1}{\beta_j} \geq 1$. Proceeding as in Case 1.a, we can find $\beta'\in [\frac12(1-\beta), 1-\beta)$ satisfying $\sum_{j=1}^m \frac{(1-\beta)-\beta'}{\beta_j} \geq 1$ and a modular inclusion
	\[
		\underset{\beta'}{\C} \oplus \underset{(1 - \beta) - \beta'}{\C} \hookrightarrow \underset{1 - \beta - (\gamma_{1} + \cdots + \gamma_{n})}{(T_{\<H,H'\>}, \vphi_{\<H,H'\>})} \oplus \underset{\gamma_{1}, \cdots, \gamma_{n}}{M_{n}(\C)},
	\]
and we can show that $((1-q)\cM(1-q),\vphi^{1-q}) \cong (T_{G}, \vphi_{G})$.  Noting that $z(1-q\colon \cM_{1}^{\vphi}) = 1$ and $\cM_1^{\vphi}\subset \cM^\vphi$, we may apply Lemma~\ref{lem:antman} to obtain the desired isomorphism result.\\

\noindent
\underline{\textbf{Case 2:}} Assume $H$ is trivial. Then $\alpha_{1} = \cdots = \alpha_{n} = 1/n$ and $\sum_{i=1}^n \frac{1}{\alpha_i} = n^2$. Clearly we never have $1-\beta, \beta <\frac{1}{n^2}$, so we consider three cases:

\underline{\textbf{Case 2.a:}} Assume $1-\beta,\beta \geq \frac{1}{n^2}$. Let $\cM_1$ and $\cM_2$ be as in Case 1.b. By \cite{MR1201693}, $(\cM_1,\phi)$ is an interpolated free group factor equipped with its trace. In particular, $z(q\colon \cM_1^\vphi)=z(1-q\colon \cM_1^\vphi)=1$. If $H'$ is trivial, then by Lemma~\ref{lem:Dykema} $(\cM_2,\vphi)$ is also an interpolated free group factor equipped with its trace. Otherwise, using Lemma~\ref{lem:Dykema}, free absorption, and Lemma~\ref{lem:antman} we see that $(\cM_2,\vphi)\cong (T_{H'},\vphi_{H'})$. The same approach applied to $(1-q)\cM (1-q)$ reveals the desired isomorphism class for $(\cM,\vphi)$.

\underline{\textbf{Case 2.b:}} Assume that $1-\beta < \frac{1}{n^{2}} \leq \beta$. Let $\cM_1$ and $\cM_2$ be as in Case 1.b. By \cite{MR1201693},
	\[
		(\cM_1,\vphi) \cong (L(\bbF_r),\tau)\oplus \underset{\frac{1}{n} - n(1-\beta), \cdots, \frac{1}{n} - n(1-\beta)}{M_{n}(\C)} ,
	\]
for some $r>1$ and the copy of $M_n(\bbC)$ lies under $q$. Note that $z(q\colon \cM_{1}^\vphi) = 1$, and from Lemma~\ref{lem:Dykema} (and possibly free absorption)
	\[
		(q\cM_{2} q,\vphi^q) \cong \begin{cases} (T_{H'}, \vphi_{H'}) &\text{if $H'$ is non-trivial}\\
								(L(\F_{r'}),\tau) &\text{otherwise, for some }r'>1.
						\end{cases}
	\]
It follows that 
	\[
		\left( (1-q)\cM_{2}(1-q),\vphi^{1-q}\right) \cong \begin{cases}(T_{H'}, \vphi_{H'}) &\text{if $H'$ is non-trivial}\\
								(L(\F_{r''}),\tau) &\text{otherwise, for some }r''>1,
						\end{cases}
	\]
where in the former case we have applied Lemma~\ref{lem:antman} twice, and in the latter case we simply note that $q$ has full central support in $\cM_1$ and hence in $\cM_2$.  Applying Lemma~\ref{lem:Dykema} gives $((1-q)\cM(1-q),\vphi^{1-q}) \cong (T_{G}, \vphi_{G})$ (recall $G$ is assumed to be non-trivial), and applying Lemma~\ref{lem:antman} once more gives the desired isomorphism.

\underline{\textbf{Case 2.c:}} Assume that $\beta < \frac{1}{n^{2}}\leq 1-\beta$. Let $\cM_1$ be as in Case 1.b, but now consider
	\[
		(\cM_3,\vphi):= \underset{\alpha_{1}, \cdots, \alpha_{n}}{\overset{e_{1,1}, \cdots, e_{n,n}}{M_{n}(\C)}}  * \left[ \underset{\beta_{1}, \cdots, \beta_{m}}{M_{m}(\C)} \oplus \underset{\beta}{\overset{q}{\C}} \right].
	\]
By \cite{MR1201693}
	\[
		(\cM_1,\vphi)\cong 	(L(\bbF_r),\tau) \oplus \underset{\frac{1}{n} - n\beta, \cdots, \frac{1}{n} - n\beta}{M_{n}(\C)},
	\]
for some $r> 1$ and the copy of $M_{n}(\C)$ is under $1 - q$. Thus $z(1-q\colon \cM_1^\vphi)=1$.  Since all weights on the $M_{n}(\C)$ in $\cM_{1}$ are identical, it follows that there exists $\beta'\in [\frac12(1-\beta), 1-\beta)$ satisfying $\sum_{k=1}^{m} \frac{(1 - \beta) - \beta'}{\beta_{k}} \geq 1$ and a modular inclusion
	\[
		\underset{\beta'}{\C} \oplus \underset{(1 - \beta) - \beta'}{\C} \hookrightarrow ( (1-q)\cM_1(1-q), \vphi^{1-q}).
	\]

By Lemma~\ref{lem:Dykema} we have
	\[
		((1-q)\cM_3(1-q),\vphi^{1-q}) \cong ((1-q)\cM_1(1-q),\phi^{1-q}) * \underset{\beta_{1}, \cdots, \beta_{m}}{M_{m}(\C)}
	\]
It follows that
	\[
		((1-q)\cM_3(1-q),\vphi^{1-q})  \cong \begin{cases} (T_{H''}, \vphi_{H''}) &\text{if $H''$ is non-trivial}\\
										(L(\F_{r'}),\tau) &\text{otherwise, for some $r'>1$}
								\end{cases}
	\]
Indeed, the former case follows by Theorem~\ref{thm:induction}.(2), Corollary~\ref{cor:ArakiB(H)}, and free absorption, while the latter case follows from \cite{MR1201693}. From here we proceed exactly as in Case 2.b, and obtain the desired isomorphism class for $(\cM,\vphi)$.
\end{proof}

In a similar manner, one can prove the following:

\begin{prop}\label{prop:matrix3}
Let $1>\alpha_{1} \geq \cdots \geq \alpha_{n}>0$ and $H$ be as in the statement of Proposition~\ref{prop:matrixabelian}.  Let $1 >  \beta_{1} \geq \cdots \geq \beta_{k} > 0,$ $1>\beta_{i,1} \geq \cdots \geq \beta_{i,m_{i}} > 0$, $m_{i} \geq 2$ for all $i$, and $\beta > 0$ satisfy $\beta+\sum_{i=1}^k \beta_i + \sum_{i=1}^\ell \sum_{j=1}^{m_{i}} \beta_{j, m_{j}}=1$.  Let $(B, \phi)$ be either $(L(\F_{t}),\tau)$ for $t \in [1, \infty]$ or $(T_{H'}, \vphi_{H'})$ for some non-trivial $H'$.  Let $H''$ be the group generated by $\{ \beta_{i, j}/\beta_{i, k} : 1 \leq i \leq \ell, \, 1 \leq j, k, \leq m_{i}\}$. Set $\gamma = \sum_{i=1}^{n} \frac{1 - \beta_{1}}{\alpha_{i}}$ and let $G$ be the group generated by $H$, $H'$, and $H''$ (where we declare $H'$ to be trivial if $B = L(\F_{t})$).  We have
	\[
		\underset{\alpha_{1}, \cdots, \alpha_{n}}{\overset{e_{1,1}, \cdots, e_{n,n}}{M_{n}(\C)}} * \left[\bigoplus_{i=1}^{k} \underset{\beta{i}}{\overset{q_{i}}{\C}} \oplus \bigoplus_{i=1}^{\ell}  \underset{\beta_{i,1}, \cdots, \beta_{i,m_{i}}}{M_{m_{i}}(\C)}\oplus \underset{\beta}{\overset{q}{(B, \phi)}}\right] \cong \begin{cases}
			(A, \varphi) & \text{if $\gamma \geq 1$}\\
			(A, \varphi) \oplus \underset{\overline{\gamma}_{1}, \cdots, \overline{\gamma}_{n}}{\overset{\overline{e}_{1,1}, \cdots, \overline{e}_{n,n}}{M_{n}(\C)}}  &\text{if $\gamma < 1$}
		\end{cases}
	\]
where $(A, \varphi)$ is an interpolated free group factor if $G$ is trivial, and is $(T_{G}, \vphi_{G})$ otherwise.  Here, $\overline{\gamma_{i}} = \alpha_{i}(1 - \gamma)$, and $\overline{e}_{i,i} \leq q_{1} \wedge e_{i,i}$.
\end{prop}

\begin{thm}\label{thm:finitedim}
Let $(A, \phi)$ and $(B, \psi)$ be finite-dimensional von Neumann algebras (both with dimension at least two) equipped with faithful states $\phi$ and $\psi$.  Assume that at least one of $\phi$ or $\psi$ is not a trace, and that up to unitary conjugation,
	\[
		(A, \phi) = \bigoplus_{i = 1}^{n} \underset{\alpha_{i,1},\cdots, \alpha_{i, k_{i}}}{\overset{p_{i,1}, \cdots, p_{i, k_{i}}}{M_{k_{i}}(\C)}}  \qquad\text{and}\qquad (B, \psi) = \bigoplus_{j = 1}^{m} \underset{\beta_{j,1},\cdots, \beta_{j, \ell_{j}}}{\overset{q_{j,1}, \cdots, q_{j, \ell_{j}}}{M_{\ell_{j}}(\C)}}.
	\]
Let $H$ be the group generated by the point spectra of $\Delta_{\phi}$ and $\Delta_{\psi}$.  Then 
$$
(A, \phi) * (B, \psi) = (T_{H}, \vphi_{H}) \oplus C
$$
where $C$ is finite-dimensional (possibly zero).  The central summands of $C$ are determined exactly as in \cite{Dyk97} as follows:  $C$ can only be nonzero if either $k_{i} = 1$ for some $i$ or  $\ell_{j} = 1$ for some $j$.  If $k_{i} = 1$ for some $i$, then a nonzero central summand appears appears if and only if there is an index $j$, satisfying $\gamma := \sum_{l = 1}^{\ell_{j}} \frac{1 - \alpha_{i, 1}}{\beta_{j, l}} < 1$.  This central summand is 
	\[
		\underset{\gamma_1, \cdots, \gamma_{\ell_{j}}}{\overset{\overline{q_{1}}, \cdots \overline{q_{\ell_{j}}}}{M_{\ell_{j}}(\C)}}
	\]
where $\gamma_{k} = \beta_{j, k}(1 - \gamma)$ and $\overline{q_{k}} \leq p_{i, 1} \wedge q_{j, k}$.  An analogous remark holds for $\ell_{j} = 1$.
\end{thm}

\begin{proof}
Let $Z(A)$ and $Z(B)$ denote the centers of $A$ and $B$ respectively.  If $Z(A)$ and $Z(B)$ are both one-dimensional, then this result is simply Proposition~\ref{prop:matrix}.  If only one of the two is one-dimensional, say $Z(A)$, then first compute $(A,\phi) * (Z(B),\psi)$ using Proposition~\ref{prop:matrixabelian}.  Let $B$ have minimal central projections $q_{1}, \cdots, q_{m}$.  For $j=0,1,\ldots, m$, set
	\[
		(\cM_j,\varphi):= (A,\phi) * \left[\bigoplus_{k=1}^j \underset{\beta_{k,1},\cdots, \beta_{k, \ell_{k}}}{\overset{q_{k,1}, \cdots, q_{k, \ell_{k}}}{M_{\ell_{k}}(\C)}}\oplus \overset{q_{j+1}}{\bbC}\oplus \cdots \oplus \overset{q_m}{\bbC}\right],
	\]
Thus we have a chain of inclusions:
	\[
		A*Z(B)=\cM_0\subset \cM_1\subset \cdots \subset \cM_m = A*B .
	\]
At each step, we use Lemmas~\ref{lem:Dykema} and \ref{lem:antman} and either Proposition~\ref{prop:matrix1} or \ref{prop:matrix2} to compute $(q_{j}\cM_{j}q_{j},\vphi^{q_j})$.  Lemma~\ref{lem:Dykema} keeps track of the central support of $q_{j}$, which by induction will be in $(A * B)^{\vphi}$,  and Lemma~\ref{lem:antman} (or the amplification formula) determines each $(\cM_{j},\vphi)$, including $(A,\phi)*(B,\psi)$.

If $Z(A)$ and $Z(B)$ both have dimension at least two, let their minimal projections be $p_{1}, \cdots, p_{n}$, and $q_{1}, \cdots, q_{m}$, respectively. We first compute $(Z(A),\phi)*(Z(B),\psi)$ using  \cite{MR1201693}. Then for $i=0,1,\ldots, n$ and $j=0,1,\ldots, m$, set
	\[
		(\cM_{i,j},\vphi) := \left[\bigoplus_{\ell = 1}^{i} \underset{\alpha_{\ell,1},\cdots, \alpha_{\ell, k_{\ell}}}{\overset{p_{\ell,1}, \cdots, p_{\ell, k_{\ell}}}{M_{k_{i}}(\C)}} \oplus \overset{p_{i+1}}{\bbC}\oplus\cdots\oplus \overset{p_n}{\bbC}\right] * \left[\bigoplus_{k=1}^j \underset{\beta_{k,1},\cdots, \beta_{k, \ell_{k}}}{\overset{q_{k,1}, \cdots, q_{k, \ell_{k}}}{M_{\ell_{k}}(\C)}}\oplus \overset{q_{j+1}}{\bbC}\oplus \cdots \oplus \overset{q_m}{\bbC}\right],
	\]
Here we have many chains of inclusions to potentially examine, \emph{e.g.}
	\[
		Z(A)*Z(B) = \cM_{0,0} \subset \cM_{1,0}\subset \cdots \subset \cM_{n,0}\subset \cM_{n,1}\subset \cdots \subset \cM_{n,m}=A*B.
	\]
We may consider any chain which increments exactly one index by one at each step. We again use Lemmas~\ref{lem:Dykema} and \ref{lem:antman} and either Proposition~\ref{prop:matrix1}, \ref{prop:matrix2}, or \ref{prop:matrix3} to compute $(q_{j}\cM_{i, j}q_{j},\vphi^{q_j})$ and/or $(p_{i}\cM_{i, j}p_{i},\vphi^{p_i})$ inductively.  Lemma~\ref{lem:Dykema} keeps track of the central support of $q_{i}$, which by induction will be in $(A * B)^{\vphi}$, Lemma~\ref{lem:antman} (or the amplification formula), determines each $(\cM_{i, j},\vphi)$, including hence $(A,\phi)*(B,\psi)$.
\end{proof}


\subsection{Computing $\cB(\cH) *[\bbC\oplus \bbC]$}

We can fairly easily extend the above result to allow finite direct sums of all separable type $\mathrm{I}$ factors, but to do so we must first compute $\cB(\cH) * [\bbC\oplus \bbC]$ when $\cB(\cH)$ is equipped with an arbitrary faithful normal state $\phi$. We require slightly more general modular inclusions than we have considered so far.

\begin{prop}
Let $i_{1}, i_{2}: \underset{\alpha_{1}, \cdots, \alpha_{n}}{\overset{e_{11}, \cdots, e_{nn}}{M_{n}(\C)}} \oplus \underset{\gamma}{\overset{p}{\C}} \rightarrow (A, \phi)$ be two modular inclusions, and assume that $A^{\phi}$ is a factor.  Then there exists a unitary $u \in A^{\phi}$ conjugating $i_{1}$ to $i_{2}$
\end{prop}

\begin{proof}
Since $A^{\phi}$ is a factor and the inclusions are modular (in particular state-preserving), there exists partial isometries $v,w\in A^\phi$ satisfying
\begin{align*}
v^{*}v &= i_{1}(e_{11}) \qquad vv^{*} = i_{2}(e_{11})\\
w^{*}w &= i_{1}(p) \qquad\ ww^{*} = i_{2}(p)
\end{align*}
It is easy to check that $u = i_{2}(p)wi_{1}(p) + \sum_{i=1}^{n} i_{2}(e_{i1})vi_{1}(e_{1i})$ does the job.
\end{proof}

The next proposition follows directly from the proof of \cite[Theorem 4.3]{Hou07}.

\begin{prop}\label{prop:absorb}
Assume that $1 \geq \alpha_{1} \geq \alpha_{2} \geq \cdots \geq \alpha_{n} > 0$ with at least one inequality amongst the $\alpha$'s strict and let $H = \langle \frac{\alpha_{i}}{\alpha_{j}} : 1 \leq i, j \leq n\rangle$. Let $\gamma>0$. Let  $\beta \in (0,1)$ be such that 
	\[
		\left[\underset{\alpha_{1}, \cdots, \alpha_{n}}{M_{n}(\C)} \oplus \underset{\gamma}{\C} \right] * \left[ \underset{\beta}{\C} \oplus \underset{1 - \beta}{\C} \right]
	\]
is a factor (necessarily $(T_{H}, \vphi_{H})$).  If $(A, \phi)$ and $(B, \psi)$ are von Neumann algebras with faithful states accepting the modular inclusions
$$
\underset{\alpha_{1}, \cdots, \alpha_{n}}{M_{n}(\C)} \oplus \underset{\gamma}{\C}  \hookrightarrow (A, \phi) \qquad\text{ and }\qquad \underset{\beta}{\C} \oplus \underset{1 - \beta}{\C} \rightarrow (B, \psi),
$$
then $(A, \phi) * (B, \psi) \cong (A, \phi) * (B, \psi) * (T_{H}, \vphi_{H})$.
\end{prop}

Assume that $\cH$ is a separable infinite-dimensional Hilbert space, and $\phi$ a faithful normal state on $\cB(\cH)$. Recall that we can assume that---after conjugating by a unitary---there is a set of matrix units $\{e_{i,j}\}$ satisfying $\phi(e_{i,j}) = \delta_{i, j}\alpha_{i}$, where $\alpha_{i}>0$  and $\sum_{i=1}^\infty \alpha_i=1$.

\begin{thm}\label{thm:B(H)}
Let $\cH$ be a separable  infinite-dimensional Hilbert space. Let $\phi$ be a faithful normal state on $\cB(\cH)$ such that $\phi(e_{i,j})=\delta_{i,j}\alpha_i$ for matrix units $\{e_{i,j}\}_{i,j\in \bbN}$ and $\alpha_i>0$. If $\beta \in (0,1)$,  then 
	\[
		(\cB(\cH), \phi) * \left[\underset{\beta}{\C} \oplus \underset{1-\beta}{\C}\right] \cong (T_{H}, \vphi_{H})
	\]
where $H = \langle \frac{\alpha_{i}}{\alpha_{j}} : i, j \in \N \rangle$.
\end{thm}

\begin{proof}
For each $n\in \N$, $(\cB(\cH), \phi)$ accepts a modular inclusion of
	\[
		(A,\psi) := \underset{\alpha_{1}, \cdots, \alpha_{n}}{\overset{e_{11}, \cdots, e_{nn}}{M_{n}(\C)}} \oplus \underset{\gamma}{\C},
	\]
where $\gamma = 1 - (\alpha_1+\cdots + \alpha_n)$. Choose $n$ large enough so that 
	\[
		(A, \psi) *  \left[ \underset{\beta}{\C} \oplus \underset{1 - \beta}{\C} \right]
	\]
is a factor and so that for some $i, j \leq n$, $\frac{\alpha_{i}}{\alpha_{j}} \neq 1$.  Let $H' = \langle \frac{\alpha_{i}}{\alpha_{j}} : 1 \leq i, j \leq n \rangle$. It follows from Propositions \ref{prop:absorb} and \ref{prop:L(Z)B(H)} that 
\begin{align*}
(\cB(\cH), \phi) * \left[\underset{\beta}{\C} \oplus \underset{1-\beta}{\C}\right] &\cong (\cB(\cH), \phi) * \left[\underset{\beta}{\C} \oplus \underset{1-\beta}{\C}\right] * (T_{H'}, \vphi_{H'})\\
&\cong (\cB(\cH), \phi) * L(\Z) * (T_{H'}, \vphi_{H'})\\
&\cong (T_{H}, \vphi_{H}) * (T_{H'}, \vphi_{H'})\\
&\cong (T_{H}, \vphi_{H})
\end{align*}
as claimed.
\end{proof}


\subsection{Free products of finite direct sums of type I factors}

Using the results of Proposition \ref{prop:absorb} and Theorem \ref{thm:B(H)},  one can prove the following proposition in exactly the same manner as Propositions \ref{prop:matrix}, \ref{prop:matrixabelian}, \ref{prop:matrix1}, and \ref{prop:matrix2}.

\begin{prop}\label{prop:B(H)steps}

\begin{enumerate}

\item[]

\item[\normalfont{(i)}] Let $A$ and $B$ be separable type I factors, at least one of which is infinite-dimensional, equipped with faithful normal states, $\phi$ and $\psi$ respectively.  Then 
$$
(A, \phi) * (B, \psi) \cong (T_{H}, \vphi_{H})
$$
where $H$ is the group generated by the point spectra of $\Delta_{\phi}$ and $\Delta_{\psi}$.

\item[\normalfont{(ii)}] Let $\cH$ be a separable and infinite-dimensional Hilbert space, and let $\phi$ be a faithful normal state on $\cB(\cH)$.  Then
$$
(\cB(\cH), \phi) * \left[\underset{\alpha_{1}}{\C} \oplus \cdots \oplus \underset{\alpha_{n}}{\C}\right] \cong (T_{H}, \vphi_{H})
$$
where $H$ is the group generated by the point spectrum of $\Delta_{\phi}$.

\item[\normalfont{(iii)}] Let $\cH$ be a separable and infinite-dimensional Hilbert space, and let $\phi$ be a faithful normal state on $\cB(\cH)$. Let $(B,\psi)$ be either $(L(\F_{t}),\tau)$ for $t \geq 1$ or $(T_{H},\vphi_H)$ for some non-trivial $H$.  Let $1>\alpha_1\geq \alpha_2\geq \cdots \geq \alpha_n >0$, $n\geq 1$ , and $\alpha>0$.  Then
	\[
		(\cB(\cH), \phi) * \left[ \underset{\alpha_{1}, \cdots, \alpha_{n}}{M_{n}(\C)} \oplus \underset{\alpha}{(B,\psi)}\right] \cong (T_{G}, \vphi_{G})
	\]
where $G$ is the group generated by $\Delta_{\phi}$, $\Delta_\psi$, and $\<\frac{\alpha_i}{\alpha_j}\colon 1\leq i,j\leq n\>$.

\end{enumerate}
\end{prop}

The next theorem follows from Proposition \ref{prop:B(H)steps} and the proof of Theorem \ref{thm:finitedim}.

\begin{thm}\label{thm:fdB(H)}
Let $(A, \phi)$ and $(B, \psi)$ be finite direct sums of separable type I factors with faithful normal states $\phi$ and $\psi$ respectively, both of which are at least two-dimensional.  Assume that at least one of $\phi$ or $\psi$ is not a trace, and that up to unitary conjugation,
	\[
		(A, \phi) = \bigoplus_{i = 1}^{n_{0}} \underset{\alpha_{i,1},\cdots, \alpha_{i, k_{i}}}{\overset{p_{i,1}, \cdots, p_{i, k_{i}}}{M_{k_{i}}(\C)}} \oplus \bigoplus_{i = 1}^{n_{1}} (\cB(\cH_{i}), \phi_{i})\qquad \text{and}\qquad (B, \psi) = \bigoplus_{j = 1}^{m_{0}} \underset{\beta_{j,1},\cdots, \beta_{j, \ell_{j}}}{\overset{q_{j,1}, \cdots, q_{j, \ell_{j}}}{M_{\ell_{j}}(\C)}} \oplus \bigoplus_{j = 1}^{m_{1}} (\cB(\cK_{j}), \psi_{i}).
	\]
where the $\cH_{i}$ and $\cK_{j}$ are infinite-dimensional.  Let $G$ be the group generated by the point spectra of $\Delta_{\phi}$ and $\Delta_{\psi}$.  Then 
	\[
		(A, \phi) * (B, \psi) = (T_{G}, \vphi_{G}) \oplus C,
	\]
where $C$ is finite-dimensional (possibly zero), and is determined exactly as in Theorem \ref{thm:finitedim}.

\end{thm}

\section{Standard embeddings and inductive limits}\label{sec:std_embeddings}

In \cite{Dyk97}, the notion of ``free subcomplimentation" of $(A, \vphi) \hookrightarrow (B, \vphi)$ was defined and used to prove Dykema's theorem on the structure of free products of inductive limits of finite-dimensional von Neumann algebras.  In order to identify the type III summand of such a free product with an almost periodic free Araki--Woods factor, the notion of a standard embedding along the lines of \cite{MR1201693} is needed for almost periodic free Araki--Woods factors.


\subsection{Hyperfinite matricial models}

\begin{lem}\label{lem:hyperfinite_disintegration}
Let $R$ be the hyperfinite $\mathrm{II}_{1}$ factor, and let $\{y_i\}_{i\in \bbN}$ be a free family of generalized circular elements free from $R$. Let $y_i$ have parameter $\lambda_i \leq 1$, $H=\< \lambda_i\colon i\in \bbN\>$, and define $(M, \vphi) = (W^{*}(R \cup \{y_{i}\}_{i\in\bbN}), \tau * \vphi_H)$.  Suppose $p_{1}, \cdots, p_{n}$ are equivalent projections in $R$, and for $j= 1,\ldots, n$ let $u_{j} \in R$ satisfy $u_{1} = p_{1}$, $u_{j}^{*}u_{j} = p_{1}$, and $u_{j}u_{j}^{*} = p_{j}$.  Then $\{u_{j}^{*}y_iu_{k}\colon i\in\bbN,\ j,k=1,\ldots, n\}$ is a free family of generalized circular elements (with respective parameters $\lambda_i$) in $(p_{1}\cM p_{1}, \vphi^{p_1})$, which is also free from $p_1Rp_1$. 
\end{lem}

\begin{proof}
We first recall Shlyakhtenko's matricial model (see \cite[Section 5]{Shl97}). Let $\cH$ be an infinite-dimensional Hilbert space, let $\cF(\cH)$ denote the full Fock space over $\cH$, and let $N\in \N$. Equip
	\[
		\cB(\cF(\cH))\otimes M_N(\bbC)
	\]
with the state $\omega\otimes \tr$. Let $\{\xi^i_{j,k}, \eta^i_{j,k}\colon i\in \bbN,\ j,k=1,\ldots, N\}$ be an orthonormal family in $\cH$. For each $i\in \bbN$ define
	\[
		Y_i:= \frac{1}{\sqrt{N}}\sum_{j,k=1}^N \left[ \ell(\xi^i_{j,k}) + \sqrt{\lambda_i} \ell(\eta^i_{k,j})^*\right]\otimes e_{j,k}
	\]
Then $\{Y_i\}_{i\in \bbN}$ is a free family and free from $M_N(\bbC)$ with respect to $\omega\otimes\tr$, and $(Y_i, \omega\otimes \tr)\sim_d (y_i, \vphi_H)$ where $\sim_d$ means equality in moments.

We first assume $\tau(p_1) = \frac{a}{b} \in \Q\setminus\{1\}$. Let
	\[
		P_1 := e_{1,1} + \cdots + e_{a,a} \in M_b(\bbC)
	\]
Denote $U_1:=P_1$ and for $j=2,\ldots, n$ set
	\[
		U_j:= e_{(j-1)a+1,1}+e_{(j-1)a+2, 2} + \cdots + e_{ja,a}.
	\]
Note that $U_j^*U_j=P_1$, and $\{U_jU_j^*\}_{j=1}^n$ are orthogonal. 

Now, for $d\in \bbN$, let $\{Y_i^{(d)}\}_{i\in \bbN}$ be as above for $N=b^d$. Let $P_1^{(d)}= P_1\otimes I_{b^{d-1}} \in M_{b^d}(\bbC)$, and similarly define $U_j^{(d)}$ for $j=1,\ldots, n$. By mutual orthogonality of $\{\xi_{j,k}^i, \eta_{j,k}^i\}$, we have that $\{(U_j^{(d)})^* Y_i^{(d)} U_k^{(d)}\colon i\in \bbN,\ j,k=1,\ldots, n\}$ is a free family of generalized circular elements (with respective parameters $\lambda_i$), which is also free from $P_1^{(d)} M_{b^d}(\bbC)P_1^{(d)}$.
Let $A \in M_{b^d}(\bbC)$.  It follows that $(P_{1}AP_{1}, \{(U_j^{(d)})^* Y_i^{(d)} U_k^{(d)}\colon i\in \bbN,\ j,k=1,\ldots, n\})$ converges in moments to $(p_{1}Ap_{1}, \{u_j^* y_i u_k\colon i\in \bbN,\ j,k=1,\ldots, n\})$ in $(p_{1}Mp_{1}, \varphi^{p_{1}})$ where we picture $p_{1}Ap_{1} \in \bigotimes_{i=1}^{\infty} M_{b}(\C),$ an ultra-weakly dense $*$-subalgebra of $R$.  By freeness and ultra-weak continuity,   it follows that $p_1Rp_1$ is  free from $\{u_j^* y_i u_k\colon i\in \bbN,\ j,k=1,\ldots, n\}$ in $(p_1\cM p_1,\vphi^{p_1})$, and moreover, the latter set is a free family of generalized circular elements with respective parameters $\lambda_i$.

Next, we assume $\tau(p_1)$ is irrational. Let $\{q_m\}_{m\in \bbN}\subset \bbQ$ be a sequence converging to $\tau(p_1)$ from below. For $j=1,\ldots, n$, let $(u_j^{(m)})_{m\in\bbN}\subset R$ be a sequence converging $*$-strongly to $u_j$ and satisfying: 
	\begin{itemize}
	\item $\tau((u_j^{(m)})^*u_j^{(m)})=q_m$;
	\item $p_1^{(m)}:=(u_1^{(m)})^* u_1^{(m)} = \cdots = (u_n^{(m)})^*u_n^{(m)}$;
	\item $u_1^{(m)}(u_1^{(m)})^*, \ldots, u_n^{(m)}(u_n^{(m)})^*$ are mutually orthogonal.
	\end{itemize}
Then from our treatment of the rational case, we have that for each $m\in \bbN$, $\{ (u_j^{(m)})^* y_i u_k^{(m)}\colon i\in \bbN,\ j,k=1,\ldots,n\}$ is a free family of generalized circular elements (with respective parameters $\lambda_i$) which are free from $p_1^{(m)} R p_1^{(m)}$ in $( p_1^{(m)} \cM p_1^{(m)}, \vphi^{p_1^{(m)}})$. Since $*$-strong convergence implies convergence in moments, we obtain the desired result.
\end{proof}

\begin{lem}\label{lem:cut_and_paste}
Let $R$ be the hyperfinite II$_{1}$ factor, $(y_1, y_2)$ a free pair of generalized circular elements of parameter $\lambda \leq 1$ that free from $R$, and $p$ a nonzero projection in $R$.  Then:

\begin{itemize}

\item $py_1+ (1-p)y_2$ is a generalized circular element of parameter $\lambda$ and free from $R$;

\item $y_1p + y_2(1-p)$ is a generalized circular element of parameter $\lambda$ and free from $R$;

\item $py_1p + (y_2 - py_2p)$ is a generalized circular element of parameter $\lambda$ and free from $R$.

\end{itemize}
\end{lem}

\begin{proof}
We will treat the first claim, as the others follow by the same techniques. As in the previous lemma, we first consider the case when $\tau(p)=\frac{a}{b}\in \bbQ\setminus\{1\}$. Let
	\[
		P:= e_{1,1}+\cdots + e_{a,a}.
	\]
For $d\in \bbN$, let $\{Y_1^{(d)}, Y_2^{(d)}\}$ be as in Shlyakhtenko's matricial model for $N=b^d$:
	\[
		Y_i^{(d)} = \frac{1}{\sqrt{b^d}} \sum_{j,k=1}^{b^d} \left[ \ell(\xi^i_{j,k}) + \sqrt{\lambda} \ell(\eta^i_{j,k})^*\right]\otimes e_{j,k}
	\]
Let $P_1^{(d)} = P_1\otimes I_{b^{d-1}}\in M_{b^d}(\bbC)$. Then by mutual orthogonality of the $\{\xi_{j,k}^i, \eta_{j,k}^i\}$ we have that $P_1^{(d)} Y_1^{(d)} + (1- P_1^{(d)}) Y_2^{(d)}$ is a generalized circular element of parameter $\lambda$ which is free from $M_{b^d}(\bbC)$. Let $A \in M_{b^{d}}(\C)$.  It follows that the moments of $(A, P_1^{(d)} Y_1^{(d)} + (1- P_1^{(d)}) Y_2^{(d)})$ converge to the moments of $(A, p y_1 + (1-p)y_2)$ where as above, we picture $p$ and $A$ in $\bigotimes_{n=1}^{\infty} M_{b}(\C)$. By ultra-weak continuity, it follows that $py_1+(1-p)y_2$ is a generalized circular element of parameter $\lambda$ which is free from $R$. For the case $\tau(p)$ is irrational, we proceed exactly as in the previous lemma.
\end{proof}

\begin{prop}\label{prop:hyperfinitematrixmodel}
Let $R$ be the hyperfinite $\mathrm{II}_1$ factor, and let $p$ be a non-trivial projection in $R$. Let $y$ be a generalized circular element with parameter $\lambda\in (0,1)$, free from $R$. Then
	\[
		W^*(R\cup \{pyp\}) \cong W^*(R \cup \{y\}).
	\]
Moreover, this isomorphism is state-preserving and restricts to the identity on $R$.
\end{prop}
\begin{proof}
We first note that if $\{y_i\}_{i\in \bbN}$ is a countably infinite family of freely independent generalized circular elements all of parameter $\lambda$, then
	\[
		W^*(R \cup \{y_i\}_{i\in \bbN} )\cong W^*( R \cup \{y\}),
	\]
on account of $\displaystyle(T_\lambda,\vphi_\lambda)\cong \bigast_{i\in \bbN} (T_\lambda,\phi_\lambda)$, and of course this isomorphism is the identity on $R$.

Let $a\in \bbN$ such that $a < \frac{1}{\tau(p)} \leq a+1$. Let $p_1,\ldots, p_{a+1}$ be orthogonal projections in $R$ summing to 1 such that $p_1:=p$, $p_2,\ldots, p_a$ are equivalent to $p_1$, and $p_{a+1}\preceq p_1$. Let $u_1,\ldots, u_{a+1}\in R$ be partial isometries implementing these (sub)equivalences such that $u_j^*u_j \leq p_1$. Observe that
	\[
		\cM:=W^*(R \cup \{y_i\}_{i\in \bbN} ) = W^*(R \cup \{u_j^* y_i u_k\colon i\in \bbN,\ j,k=1,\ldots, a+1\}).
	\]
By Lemma~\ref{lem:hyperfinite_disintegration}, $\{u_j^* y_i u_k\colon i\in \bbN,\ j,k=1,\ldots, a\}$ is a free family of generalized circular elements with parameter $\lambda$, which is free from $p_1 Rp_1$.

Let $q_1:= u_{a+1}^*u_{a+1} \leq p_1$. Let
	\[
		t:=\frac{(a+1)\tau(p_1)}{a\tau(p_1) + \tau(q_1)}.
	\]
Consider the amplification $R^t$ of $R$. Let $P\in R^t$ be a projection such that we can identify $R$ with $PR^tP$. By Lemma~\ref{lem:hyperfinite_disintegration}, we can find a free family $\{ y_i^{(t)}\}_{i\in \bbN}$ of generalized circular elements of parameter $\lambda$ that are free from $R^t$, and which satisfy $Py_i^{(t)} P= y_i$ for each $i\in \bbN$. Let $p_{a+1}^{(t)} \in R^{t}$  satisfy $p_{a+1}^{(t)} = 1 -( p_1+\cdots +p_a)$. By our choice of $t$, $p_{a+1}^{(t)}$ is equivalent to $p_1$, so let $u_{a+1}^{(t)}\in R^t$ be a partial isometry implementing this equivalence (with $(u_{a+1}^{(t)})^* u_{a+1}^{(t)} =p_1$). Moreover, we can choose this partial isometry so that $u_{a+1}^{(t)} q_1= u_{a+1}$. Consequently, for $i\in \bbN$ and $j,k=1,\ldots, n$ we have
	\begin{align}\label{eqn:relating_amp_down}
		u_{a+1}^* y_i u_k &= q_1 (u_{a+1}^{(t)})^* y_i^{(t)} u_k\nonumber\\
		u_j^* y_i u_{a+1} & = u_j^* y_i^{(t)} u_{a+1}^{(t)} q_1\\
		u_{a+1}^* y_i u_{a+1} & = q_1(u_{a+1}^{(t)})^* y_i^{(t)} u_{a+1}^{(t)} q_1\nonumber
	\end{align}
Note that by Lemma~\ref{lem:hyperfinite_disintegration},
	\[
		\{u_j^* y_i^{(t)} u_k, (u_{a+1}^{(t)})^*y_i^{(t)}u_k, u_j^* y_i^{(t)} u_{a+1}^{(t)}, (u_{a+1}^{(t)})^* y_i^{(t)} u_{a+1}^{(t)}\colon i\in \bbN,\ j,k=1,\ldots, a\}
	\]
is a free family of generalized circular elements of parameters $\lambda$ that is free from $p_1 R^t p_1 = p_1Rp_1$ in $(p_1\cM p_1, \vphi^{p_1})$.

Now, let
	\begin{align*}
		\alpha \colon \bbN\times \bbN\times \{\ell, r,c\} &\to \bbN\times \{1,\ldots, a\} \times \{1,\ldots, a\}\\
		\beta \colon \bbN &\to \bbN\times \{1,\ldots, a\}
	\end{align*}
be bijections. If $\alpha(s,t,\epsilon) = (i,j,k)$ for $s,t\in \bbN$ and $\epsilon\in \{\ell,r,c\}$, then set
	\[
		y^\epsilon_{s,t}:= u_j^* y_i u_k.
	\]
If $\beta(s) = (i,j)$ for $s\in \bbN$ set
	\[
		y^\ell_{s,0}:= u_j^* y_i u_{a+1} \qquad\text{and}\qquad y^r_{s,0}:= u_{a+1}^* y_i u_j.
	\]
Finally, set $y^c_{s,0}:= u_{a+1}^* y_s u_{a+1}$ for each $s\in \bbN$. With this notation, we therefore have
	\[
		\cM = W^*(R, \{ y^{\epsilon}_{s,t} \colon \epsilon\in \{\ell,r,c\},\ s\in\bbN,\ t\in\bbN_0\}).
	\]
For $(s,t,\epsilon)\in \bbN\times\bbN_0\times\{\ell,r,c\}$, define
	\[
		z_{s,t}^\epsilon :=\begin{cases} y^\ell_{s,t}q_1+ y^\ell_{s,t+1}(p_1 - q_1) &\text{if }\epsilon=\ell\\
								q_1y_{s,r}^r + (p_1-q_1)y_{s,t+1}^r &\text{if }\epsilon=r\\
								q_1y^c_{s,t} q_1 + \left[ y_{s,t+1}^c - q_1 y_{s,t+1}^c q_1\right] &\text{if }\epsilon=c\end{cases}.
	\]
Using Equations~(\ref{eqn:relating_amp_down}) and Lemma~\ref{lem:cut_and_paste}, it follows that $\{z_{s,t}^\epsilon \colon \epsilon\in \{\ell,r,c\},\ s\in\bbN,\ t\in\bbN_0\}$ is a free family of generalized circular elements with parameter $\lambda$ which are free from $p_1 Rp_1$ in $(p_1 \cM p_1, \vphi^{p_1})$. We also note that
	\[
		\cM = W^*( R, \{z_{s,t}^\epsilon \colon \epsilon\in \{\ell,r,c\},\ s\in\bbN,\ t\in\bbN_0\}).
	\]
Now, observe that by freeness and $\displaystyle(T_\lambda,\vphi_\lambda)\cong \bigast_{i\in \bbN} (T_\lambda,\phi_\lambda)$ 
	\[
		p_1 \cM p_1 = W^*( p_1 Rp_1, \{z_{s,t}^\epsilon \colon \epsilon\in \{\ell,r,c\},\ s\in\bbN,\ t\in\bbN_0\}) \cong W^*( p_1 Rp_1, z),
	\]
where $z=p_1zp_1$ is a generalized circular element of parameter $\lambda$ that is free from $p_1 Rp_1$, and this state-preserving isomorphism restricts to the identity on $p_1 Rp_1$.  Call this isomorphism $\gamma$.  We then define $\tilde{\gamma} \colon \cM \to W^*(R, z)$ by
	\[
		\tilde{\gamma}(x) = \sum_{j,k=1}^n u_j \gamma(u_j^* x u_k) u_k^*.
	\]
It is easy to check that this is state-preserving and restricts to the identity on $R$. Finally, to complete the proof, we appeal to Lemma~\ref{lem:hyperfinite_disintegration} once more in order to realize $z= p_1 Z p_1$ for some generalized circular element with parameter $\lambda$ which is free from $R$. Hence $\cM \cong W^*(R, p_1 Zp_1)$ via a state-preserving isomorphism that restricts to the identity on $R$.
\end{proof}

Suppose $H < \R_{+}$ is non-trivial and has two generating sets: $(\lambda_{i})_{i \in I}$ and $(\lambda_{i'})_{i \in I'}$.  Let $(y_i)_{i\in I}$ be a free generalized circular family free from $R$  with each $y_{i}$ having parameter $\lambda_{i}$.  Similarly,  let $(y_i')_{i'\in I'}$ be a free generalized circular family free from $R$  with each $y_{i'}$ having parameter $\lambda_{i'}$. Since there is an isomorphism $W^{*}(R \cup (y_{i})_{i \in I}) \cong W^{*}(R \cup (y_{i'})_{i' \in I'})$ which is the identity on $R$, we obtain the following corollary:

\begin{cor}\label{cor:changegen}
Suppose $H < \R_{+}$ is non-trivial and has two generating sets: $(\lambda_{i})_{i \in I}$ and $(\lambda_{i'})_{i \in I'}$. 
Let $J$ be a set disjoint from $I$ and $I'$,  let $\lambda_{j} \in (0, 1)$ for each $j \in J$, and let $H'$ be the group generated by $H$ and $(\lambda_{j})_{j \in J}$.
Let $(y_i)_{i\in I \cup J}$ be a free circular family free from $R$  with each $y_{i}$ having parameter $\lambda_{i}$. 
 Similarly,  let $(y_i')_{i'\in I'\cup J}$ be a free circular family free from $R$  with each $y_{i'}$ having parameter $\lambda_{i'}$. Let $(p_k)_{k\in I\cup I'\cup J}$ be a family of nonzero projections in $R$.  Then $W^{*}(R \cup (p_{i}y_{i}p_{i})_{i \in I \cup J}) \cong W^{*}(R \cup (p_{i'}y_{i'}p_{i'})_{i' \in I' \cup J})$ via a mapping which is the identity on $W^{*}(R, (p_{j}y_{j}p_{j})_{j \in J})$.

\end{cor}

Recall from \cite{MR1256179} that Dykema's construction of the interpolated free group factors $L(\F_{t})$ consisted of the following ingredients
\begin{itemize}

\item The hyperfinite $\mathrm{II}_{1}$ factor $R$;

\item A family of projections $(p_{s})_{s \in S}$ in $R$;

\item A free semicircular family $(x_{s})_{s \in S}$.

\end{itemize}
The interpolated free group factor with parameter $t$ is given by
	\[
		L(\F_{t}) := W^{*}(R\cup (p_{s}x_{s}p_{s})_{s \in S})
	\]
where $t = 1 + \sum_{s \in S} \tau(p_{s})^{2}$. It should be noted that if $(x_{s'})_{s' \in S'}$ is another free semicircular family free from $R$ and $(p_{s'})_{s' \in S}$ are any projections in $R$ satisfying $t = 1 + \sum_{s' \in S'} \tau(p_{s'})^{2}$, then $W^{*}(R, (p_{s}x_{s}p_{s})_{s \in S})$ and $W^{*}(R, (p_{s'}x_{s'}p_{s'})_{s' \in S'})$ are isomorphic via an isomorphism which is the identity on $R$.

The presence of generalized circular elements with non-trivial parameters eliminates (via free absorption) the need for a fixed value of $ \sum_{s' \in S'} \tau(p_{s'})^{2}$.  We thus have the following:

\begin{prop}\label{prop:changegen}
Let $I$ and $I'$  be disjoint sets which are finite or countable,  $\lambda_{i} \in (0, 1)$ for each $i \in I \cup I' $, and $H:=\langle \lambda_{i} : i \in I \rangle = \langle \lambda_{i'} : i' \in I' \rangle$.  Let $S,$ $S'$, and $T$ be finite or countable (possibly empty) disjoint sets.  Assume that we are given the following:
	\begin{itemize}
	\item Families of nonzero projections $(p_{s})_{s \in S}$, $(p_{s'})_{s' \in S'}$, $(p_{t})_{t\in T}$, $(q_{i})_{i \in I}$, and $(q_{i'})_{i' \in I'}$ in $R$ ;
	
	\item A free semicircular family $(x_{s})_{s \in S}$ free from $R$;
	
	\item A free semicircular family $(x_{s'})_{s \in S'}$ free from $R$;
	
	\item A free semicircular family $(x_t)_{t\in T}$ free from $R\cup (x_s)_{s\in S}\cup (x_{s'})_{s'\in S'}$;

	\item A free family $(y_{i})_{i \in I}$ of generalized circular elements with respective parameters $\lambda_i$ that is free from $R \cup (x_{s})_{s \in S \cup T}$;

	\item A free family  $(y_{i'})_{i' \in I'}$ of generalized circular elements with respective parameters $\lambda_{i'}$ that is free from $R \cup(x_{s'})_{s' \in S' \cup T}$.

	\end{itemize}
Then	
	\[
		W^*(R\cup(p_s x_s p_s)_{s\in S\cup T}\cup (q_i y_i q_i)_{i\in I}) \cong W^*(R\cup( p_{s'}x_{s'}p_{s'})_{s'\in S'\cup T}\cup (q_{i'} y_{i'} q_{i'})_{i'\in I'})
	\]
via a state-preserving isomorphism which is the identity on $W^{*}(R \cup (p_{t}x_{t}p_{t})_{t \in T})$.

\end{prop}

\begin{proof}

From the proof of Proposition \ref{prop:hyperfinitematrixmodel}, we see that there is a state-preserving isomorphism 
$$
W^{*}(R \cup (p_{t}x_{t}p_{t})_{t \in T} \cup (p_{s}x_{s}p_{s})_{s \in S} \cup (q_{i}y_{i}q_{i})_{i \in I}) \cong W^{*}(R \cup (p_{t}x_{t}p_{t})_{t \in T} \cup (p_{s}x_{s}p_{s})_{s \in S} \cup (y_{i})_{i \in I}).
$$
which is the identity on $W^{*}(R \cup (p_{t}x_{t}p_{t})_{t \in T} \cup (p_{s}x_{s}p_{s})_{s \in S})$.  Let $T'$ be a countably infinite set disjoint from $S$, $S'$, and $T$, and $(x_{t'})_{t' \in T'}$ a free semicircular family free from $R \cup (x_{t})_{t \in T} \cup (x_{s})_{s \in S} \cup (x_{s'})_{s' \in S'} \cup (y_{i})_{i \in I} \cup (y_{i'})_{i' \in I'}$.  By free absorption, there is a state-preserving isomorphism 
$$
W^{*}(R \cup (p_{t}x_{t}p_{t})_{t \in T} \cup (p_{s}x_{s}p_{s})_{s \in S} \cup (y_{i})_{i \in I}) \cong W^{*}(R \cup (p_{t}x_{t}p_{t})_{t \in T} \cup (p_{s}x_{s}p_{s})_{s \in S} \cup (y_{i})_{i \in I} \cup (x_{t'})_{t' \in T'})
$$
which is the identity on $W^{*}(R \cup (p_{t}x_{t}p_{t})_{t \in T}\cup (p_{s}x_{s}p_{s})_{s \in S})$.  Via Dykema's cutting and pasting argument \cite{MR1256179}, there is a state-preserving isomorphism
$$
W^{*}(R \cup (p_{t}x_{t}p_{t})_{t \in T} \cup (p_{s}x_{s}p_{s})_{s \in S} \cup (y_{i})_{i \in I} \cup (x_{t'})_{t' \in T'}) \cong W^{*}(R \cup (p_{t}x_{t}p_{t})_{t \in T} \cup (y_{i})_{i \in I} \cup (x_{t'})_{t' \in T'}) 
$$
which is the identity on $W^{*}(R \cup (p_{t}x_{t}p_{t})_{t \in T}\cup (y_{i})_{i \in I})$.  Finally, by uniqueness of $(T_{H}, \vphi_{H})$, there is a a state-preserving isomorphism 
$$
W^{*}(R \cup (p_{t}x_{t}p_{t})_{t \in T} \cup (y_{i})_{i \in I} \cup (x_{t'})_{t' \in T'}) \cong W^{*}(R \cup (p_{t}x_{t}p_{t})_{t \in T} \cup (y_{i'})_{i' \in I'}, (x_{t'})_{t' \in T'}) 
$$
which is the identity on $W^{*}(R \cup (p_{t}x_{t}p_{t})_{t \in T} \cup (x_{t'})_{t' \in T'})$.  Composing these isomorphisms gives a state-preserving isomorphism
$$
W^{*}(R \cup (p_{t}x_{t}p_{t})_{t \in T} \cup (p_{s}x_{s}p_{s})_{s \in S} \cup (q_{i}y_{i}q_{i})_{i \in I}) \cong W^{*}(R \cup (p_{t}x_{t}p_{t})_{t \in T} \cup (y_{i'})_{i' \in I'} \cup (x_{t'})_{t' \in T'}) 
$$
which is the identity on $W^{*}(R \cup (p_{t}x_{t}p_{t})_{t \in T})$.  Similarly, there is a state-preserving isomorphism
$$
W^{*}(R \cup (p_{t}x_{t}p_{t})_{t\in T} \cup (p_{s'}x_{s'}p_{s'})_{s' \in S'} \cup (q_{i'}y_{i'}q_{i'})_{i' \in I'}) \cong W^{*}(R \cup (p_{t}x_{t}p_{t})_{t \in T} \cup (y_{i'})_{i' \in I'} \cup (x_{t'})_{t' \in T'}) 
$$
which is the identity on $W^{*}(R \cup (p_{t}x_{t}p_{t})_{t \in T})$.  This completes the proof.
\end{proof}

\begin{rem}
Note that $W^{*}(R, (p_{t}x_{t}p_{t})_{t \in T}, (p_{s}x_{s}p_{s})_{s \in S}, (q_{i}y_{i}q_{i})_{i \in I})$ equipped with the free product state $\tau *\vphi_H$ is isomorphic to $(T_{H}, \vphi_{H})$.
\end{rem}

Using freeness of the $x$'s from the $y$'s and Corollary \ref{cor:changegen}, we can upgrade Proposition \ref{prop:changegen} to the following corollary which will be useful for our notions of standard embeddings below.

\begin{cor}\label{cor:changegen2}
Let $I$, $I'$, and $J$  be disjoint sets which are finite or countable, and let  $\lambda_{i} \in (0, 1)$ for each $i \in I\cup I'\cup J$ and $H:=\langle \lambda_{i} : i \in I \cup J \rangle = \langle \lambda_{i'} : i' \in I' \cup J \rangle$.  Let $S,$ $S'$, and $T$ be finite or countable (possibly empty) disjoint sets.  Assume that we are given the following:
	\begin{itemize}
	\item Families of nonzero projections $(p_{s})_{s \in S}$, $(p_{s'})_{s' \in S'}$, $(p_{t})_{t\in T}$, $(q_{i})_{i \in I}$, $(q_{i'})_{i' \in I'}$, and $(q_j)_{j\in J}$ in $R$; 

	\item A free semicircular family $(x_{s})_{s \in S}$ free from $R$;

	\item A free semicircular family $(x_{s'})_{s \in S'}$ free from $R$;
	
	\item A free semicircular family $(x_t)_{t\in T}$ free from $R\cup (x_s)_{s\in S}\cup (x_{s'})_{s'\in S'}$;
		
	\item A free family  $(y_{i})_{i \in I}$ of generalized circular elements with respective parameters $\lambda_i$ that is free from $R \cup (x_{s})_{s \in S}\cup (x_t)_{t\in T}$;
	
	\item A free family  $(y_{i'})_{i' \in I'}$ of generalized circular elements with respective parameters $\lambda_{i'}$ that is free from $R \cup (x_{s'})_{s' \in S'}\cup (x_t)_{t\in T}$;
	
	\item A free family  $(y_{j})_{j \in J}$ of generalized circular elements with respective parameters $\lambda_j$ that is free from $R \cup (x_{s})_{s \in S}\cup (x_{s'})_{s'\in S'}\cup(x_t)_{t\in T}\cup (y_i)_{i\in I}\cup (y_{i'})_{i'\in I'}$;

	\end{itemize}
Then
	\[
		W^{*}(R \cup (p_{s}x_{s}p_{s})_{s \in S\cup T} \cup (q_{i}y_{i}q_{i})_{i \in I \cup I''}) \cong W^{*}(R \cup (p_{s'}x_{s'}p_{s'})_{s' \in S'\cup  T} \cup (q_{i'}y_{i'}q_{i'})_{i' \in I' \cup I''})
	\]
via a state-preserving isomorphism which is the identity on $W^{*}(R \cup (p_{t}x_{t}p_{t})_{t \in T} \cup (q_{j}y_{j}q_{j})_{j \in J})$.
\end{cor}


\subsection{Standard embeddings}

With our hyperfinite matricial model for $(T_{H}, \vphi_{H})$, we will proceed as in Section 4 of \cite{MR1201693} and develop the notion of a ``standard embedding" of almost periodic free Araki--Woods factors.  After the definition, the proofs of Propositions \ref{prop:std_embedding_inductive_limit}, \ref{prop:std_embedding_compression}, and \ref{prop:std_embedding_free_product} follow the proofs in Section 4 of \cite{MR1201693} very closely .

\begin{defn}\label{def:standard}
Let $H$ and $H'$ be non-trivial countable subgroups of $\R^{+}$ with $H\leq H'$. We say that a modular inclusion $\alpha: (T_{H}, \vphi_{H}) \rightarrow (T_{H'}, \vphi_{H'})$ is a \textbf{standard embedding}  if there exist:   
\begin{itemize}

\item  Sets $I \subset I'$,  $\lambda_{i} \in (0, 1)$ for all $i \in I'$ such that $H = \langle \lambda_{i} : i \in I \rangle$, and $H' = \langle \lambda_{i'} : i' \in I' \rangle$;

\item  Sets $S \subset S'$, families of nonzero projections $(p_{s})_{s \in S'}$ and $(q_{i})_{i \in I'}$ in the hyperfinite II$_{1}$ factor $R$;

\item A free family  $\{x_{s},y_{i}\colon s\in S',\ i\in I'\}$ free from $R$ where $x_{s}$ is a semicircular operator, and $y_{i}$ is a generalized circular element with parameter $\lambda_{i}$;

\item  State preserving isomorphisms 
\begin{align*}
\beta: (T_{H}, \vphi_{H}) &\rightarrow W^{*}(R \cup (p_{s}x_{s}p_{s})_{s \in S} \cup (q_{i}y_{i}q_{i})_{i \in I})\\
\gamma: (T_{H'}, \vphi_{H'}) &\rightarrow W^{*}(R \cup (p_{s}x_{s}p_{s})_{s \in S'} \cup (q_{i}y_{i}q_{i})_{i \in I'})
\end{align*}
so that $\gamma \circ \alpha \circ \beta^{-1}$ is the canonical inclusion.
\end{itemize}

If $H = \{1\}$ (i.e. $I$ is empty), then we replace $(T_{H}, \vphi_{H})$ with $L(\F_{t}, \tau)$.  If $I'$ is also empty then this is just a repetition of the definition of standard embedding in  \cite{MR1201693}. 
\end{defn}

Before proving properties about standard embeddings, we need to show that the notion is independent of the generating sets of $H$ and $H'$, the projections $p_{s}$ and $q_i$, the semicircular
operators $x_s$, and the generalized circular elements $y_i$.

\begin{prop}\label{prop:changestandard}
Let $H$ and $H'$ be countable, non-trivial subgroups of $\R^{+}$, and suppose that $\alpha: (T_{H}, \vphi_{H}) \rightarrow (T_{H'}, \vphi_{H'})$ is a standard embedding. Let $I\subset I'$, 
$(\lambda_i)_{i\in I'}\subset (0,1)$, $S\subset S'$, $(p_s)_{s\in S'}$, $(q_i)_{i\in I'}$, $(x_s)_{s\in S'}$, $(y_i)_{i\in I'}$, $\beta$, and $\gamma$ be as in Definition~\ref{def:standard}. Suppose: 
	\begin{itemize}
	\item $J\subset J'$ are sets disjoint from $I'$, and $(\lambda_j)_{j\in J'}\subset (0,1)$ are such that $H=\<\lambda_j\colon j\in J\>$ and $H'=\<\lambda_{j'} \colon j'\in J'\>$;
	
	\item $T\subset T'$ are sets disjoint from $S'$, and $(p_t)_{t\in T'}$ and $(q_j)_{j\in J'}$ are families of nonzero projections in $R$;
	
	\item $\{x_t, y_j\colon t\in T',\ j\in J'\}$ is free from $R$ where $x_t$ is semicircular and $y_j$ is a generalized circular element with parameter $\lambda_j$.
	\end{itemize}
Then there exist state-preserving isomorphisms
	\begin{align*}
		\delta: (T_{H}, \vphi_{H}) &\rightarrow W^{*}(R \cup (p_{t}x_{t}p_{t})_{t \in T} \cup (q_{j}y_{j}q_{j})_{j \in J})\\
		\epsilon: (T_{H'}, \vphi_{H'}) &\rightarrow W^{*}(R \cup (p_{t}x_{t}p_{t})_{t \in T'} \cup (q_{j}y_{j}q_{j})_{j \in J'})
	\end{align*}
so that $\epsilon \circ \alpha \circ \delta^{-1}$ is the canonical inclusion.
\end{prop}

\begin{proof}
It suffices to show there exist state-preserving isomorphisms $\Phi$ and $\Psi$ so that the following diagram commutes:
	\[
		\begin{tikzpicture}
		\node at (-3, 3) {$W^{*}(R \cup (p_{s}x_{s}p_{s})_{s \in S} \cup (q_{i}y_{i}q_{i})_{i \in I})$};
		\draw[right hook->] (0, 3) -- (3, 3);
		\node at (1.5, 3.5) {$i_{1}$};
		\node at (6.1, 3) {$W^{*}(R \cup (p_{s}x_{s}p_{s})_{s \in S'} \cup (q_{i}y_{i}q_{i})_{i \in I'})$};
		\draw[->] (-3, 2.5)--(-3, .5);
		\node at (-3.5, 1.5) {$\Phi$};
		\node at (-3, 0) {$W^{*}(R \cup (p_{t}x_{t}p_{t})_{t \in T} \cup (q_{j}y_{j}q_{j})_{j \in J})$};
		\draw[right hook->] (0, 0) -- (3, 0);
		\node at (1.5, .5) {$i_{2}$};
		\node at (6.1, 0) {$W^{*}(R \cup (p_{t}x_{t}p_{t})_{t \in T'}\cup (q_{j}y_{j}q_{j})_{j \in J'})$};
		\draw[->] (6, 2.5)--(6, .5);
		\node at (6.5, 1.5) {$\Psi$};
		\end{tikzpicture}
	\]
where $i_{i}$ and $i_{2}$ are the canonical inclusions. Indeed, given such isomorphisms we simply take $\delta:= \Phi\circ \beta$ and $\epsilon:= \Psi\circ \gamma$.  

Using Corollary \ref{cor:changegen2}, there is a state-preserving isomorphism
	\begin{align*}
		\theta_{1} \colon  W^{*} &(R \cup (p_{s}x_{s}p_{s})_{s \in S'} \cup (q_{i}y_{i}q_{i})_{i \in I'})\\
				&\downarrow\\
			W^{*} &(R \cup (p_{s}x_{s}p_{s})_{s \in S}\cup (p_{t}x_{t}p_{t})_{t \in T'\setminus T} \cup (q_{i}y_{i}q_{i})_{i \in I} \cup (q_{j}y_{j}q_{j})_{j \in J'\setminus J}),
	\end{align*}
which is the identity on $W^{*}(R \cup (p_{s}x_{s}p_{s})_{s \in S} \cup (q_{i}y_{i}q_{i})_{i \in I})$.  There is also a state-preserving isomorphism
	\begin{align*}
		\theta_{2}: W^{*} &(R \cup (p_{s}x_{s}p_{s})_{s \in S} \cup (p_{t}x_{t}p_{t})_{t \in T'\setminus T} \cup (q_{i}y_{i}q_{i})_{i \in I} \cup (q_{j}y_{j}q_{j})_{j \in J'\setminus J})\\
			\downarrow &\\ 
		W^{*}&(R \cup (p_{t}x_{t}p_{t})_{t \in T'}\cup (q_{j}y_{j}q_{j})_{j \in J'})
	\end{align*}
which is the identity on $W^{*}(R \cup (p_{t'}x_{t'}p_{t'})_{t' \in T'\setminus T} \cup (q_{j'}y_{j'}q_{j'})_{j \in J'\setminus J})$,  Set $\Psi := \theta_{2} \circ \theta_{1}$, and set $\Phi$ to be the restriction of $\Psi$ to  $W^{*}(R \cup (p_{s}x_{s}p_{s})_{s \in S} \cup (q_{i}y_{i}q_{i})_{i \in I})$.  By construction, the above diagram commutes.
\end{proof}

\begin{prop}\label{prop:std_embedding_inductive_limit}
\begin{enumerate}
\item[]

\item[\normalfont{(i)}] Let $H_{1} \leq H_{2} \leq H_{3}$ be countable, non-trivial subgroups of $\R_{+}$.  If $\alpha_{1}: (T_{H_{1}}, \vphi_{H_{1}}) \rightarrow  (T_{H_{2}}, \vphi_{H_{2}})$ and $\alpha_{2}: (T_{H_{2}}, \vphi_{H_{2}}) \rightarrow  (T_{H_{3}}, \vphi_{H_{3}})$ are standard embeddings, then $\alpha_{2} \circ \alpha_{1}$ is a standard embedding.

\item[\normalfont{(ii)}] Let $H_{1} \leq H_{2} \leq H_{3} \leq \cdots \leq H_{n} \leq \cdots$ be countable, non-trivial subgroups of $\R^{+}$, and $\displaystyle H = \bigcup_{n \in \N} H_{n}$.  For each $n \in \N$, suppose that $\alpha_{n}: (T_{H_{n}}, \vphi_{H_{n}}) \rightarrow (T_{H_{n+1}}, \vphi_{H_{n+1}})$ is a standard embedding.  Let $(\cM, \vphi)$ be the inductive limit von Neumann algebra of the $[(T_{H_{n}}, \vphi_{H_{n}}), \alpha_{n}]$.  Then $(\cM, \vphi) \cong (T_{H}, \vphi_{H})$.
\end{enumerate}
\end{prop}

\begin{proof}
(i): Based on Proposition \ref{prop:changestandard} and its proof, we may assume that there are countable sets $I_{1} \subset I_{2} \subset I_{3}$,  $S \subset S'$, $T \subset T'$ with $S'$ and $T'$ disjoint, and $\lambda_{i} \in (0, 1)$ for each $i \in I_{3}$ satisfying $\langle \lambda_{i} : i \in I_{j}\rangle = H_{j}$.  We can further assume that there are semicircular families $(x_{s})_{s \in S'}$ and $(x_{t'})_{t' \in T"}$ both free from $R$; free families $(y_{i})_{i \in I_{2}}$ and $(\tilde{y_{i}})_{i \in I_{3}}$ of generalized circular elements with respective parameters $\lambda_{i}$ that are free from each other, $(x_{s})_{s \in S'}$, $(x_{t})_{t \in T}$, and $R$; families of projections $(p_{s})_{s \in S'}$ and $(p_{t})_{t \in T'}$ in $R$; and state-preserving isomorphisms
\begin{align*}
	\beta: (T_{H_{1}}, \vphi_{H_{1}}) &\rightarrow W^{*}(R \cup (p_{s}x_{s}p_{s})_{s \in S} \cup (y_{i})_{i \in I_{1}})\\
	\gamma: (T_{H_{2}}, \vphi_{H_{2}}) &\rightarrow W^{*}(R \cup (p_{s}x_{s}p_{s})_{s \in S'} \cup (y_{i})_{i \in I_{2}})\\
	\overline{\gamma}: (T_{H_{2}}, \vphi_{H_{2}}) &\rightarrow W^{*}(R \cup (p_{t}x_{t}p_{t})_{t \in T} \cup (\tilde{y_{i}})_{i \in I_{2}})\\
	\delta: (T_{H_{3}}, \vphi_{H_{3}}) &\rightarrow W^{*}(R \cup (p_{t}x_{t}p_{t})_{t \in T'} \cup (\tilde{y_{i}})_{i \in I_{3}})
\end{align*}
so that $\gamma \circ \alpha_{1} \circ \beta^{-1}$ and $\delta \circ \alpha_{2} \circ \overline{\gamma}^{-1}$ are the canonical inclusions. Utilizing $\gamma \circ \overline{\gamma}^{-1}$, we obtain an isomorphism
	\[
		\overline{\delta}: (T_{H_{3}}, \vphi_{H_{3}}) \rightarrow W^{*}(R \cup (p_{s}x_{s}p_{s})_{s \in S'} \cup (\overline{p_{t}}x_{t}\overline{p_{t}})_{t \in T'\setminus T} \cup (y_{i})_{i \in I_{2}} \cup (\tilde{y_{i}})_{i \in I_{3}\setminus I_{2}}),
	\]
where $\overline{p_{t}} \in  W^{*}(R \cup (p_{s}x_{s}p_{s})_{s \in S'} \cup (y_{i})_{i \in I_{2}} )^{\vphi}$ and $\alpha_{2} \circ \alpha_{1} = \overline{\delta}^{-1} \circ i \circ \beta$ (here $i$ is the canonical inclusion).  Since $W^{*}(R\cup (p_{s}x_{s}p_{s})_{s \in S'} \cup (y_{i})_{i \in I_{2}}, )^{\vphi}$ is a factor, choose a unitary
	\[
		u_{t} \in W^{*}(R \cup (p_{s}x_{s}p_{s})_{s \in S'} \cup (y_{i})_{i \in I_{2}}, )^{\vphi}
	\]
so that $u_{t}^{*}\overline{p_{t}}u_{t} \in R$ for each $t \in T'\setminus T$.  Then $R$, $(x_{s})_{s \in S'}$, and $(u^{*}_{t}x_{t}u_{t})_{t \in T'\setminus T}$ is a free family and is free from $(y_{i})_{i \in I_{2}}$ and $(\tilde{y_{i}})_{i \in I_{3}}$.  This means that $\overline{\delta}$ is valued in
	\[
		 W^{*}(R \cup (p_{s}x_{s}p_{s})_{s \in S'}\cup ((u_{t}^{*}\overline{p_{t}}u_{t})(u_{t}^{*}x_{t}u_{t})(u^{*}_{t}\overline{p_{t}}u_{t})_{t \in T'\setminus T} \cup (y_{i})_{i \in I_{2}} \cup (\tilde{y_{i}})_{i \in I_{3}\setminus I_{2}})
	\]
so that $\alpha_{2} \circ \alpha_{1}$ is standard.

(ii): This follows directly from the above proof and (ii) of \cite[Proposition 4.3]{MR1201693}
\end{proof}

\begin{prop}\label{prop:std_embedding_compression}
Let $H \leq H'$ be countable, non-trivial subgroups of $\R^{+}$, $\alpha: (T_{H}, \vphi_{H}) \rightarrow (T_{H'}, \vphi_{H'})$ a modular inclusion, and $p \in T_{H}^{\vphi_{H}}$ a nonzero projection.  Then $\alpha$ is a standard embedding if and only if $\alpha\mid_{pT_{H}p}$ is a standard embedding.  \end{prop}

\begin{proof}
It is straightforward to see that if $\alpha$ is a standard embedding, then
$$
(\alpha \otimes \id): (T_{H}, \vphi_{H}) \otimes \underset{\frac{1}{n}, \dots, \frac{1}{n}}{M_{n}(\C)} \rightarrow (T_{H'}, \vphi_{H'}) \otimes \underset{\frac{1}{n}, \dots, \frac{1}{n}}{M_{n}(\C)}
$$
is a standard embedding.  It therefore suffices to prove that $\alpha$ being a standard embedding implies $\alpha\mid_{pT_{H}p}$ is a standard embedding.  Let $I \subset I'$, $S \subset S'$, $(p_{s})_{s \in S'}$, $(q_{i})_{i \in I'}$, $(x_{s})_{s \in S'}$, $(\lambda_{i})_{i \in I'}$, $(y_{i})_{i \in I'}$, $\beta$, and $\gamma$ be as in Definition \ref{def:standard}.  Without loss of generality, we can assume that (after conjugating by unitaries in the centralizer) $\beta(p) \in R$ and $\gamma(p) \in R$.  We will simply call these images $p$.  We are allowed to assume that $p_{s} \leq p$ and $q_{i} \leq p$ for all $s \in S'$.  $\alpha\mid_{pT_{H}p}$ is therefore conjugate to the canonical inclusion
	\[
		W^{*}(pRp\cup (p_{s}x_{s}p_{s})_{s \in S}\cup (q_{i}y_{i}q_{i})_{i \in I}) \hookrightarrow W^{*}(pRp \cup (p_{s}x_{s}p_{s})_{s \in S'} \cup (q_{i}y_{i}q_{i})_{i \in I'})
	\]
establishing that $\alpha\mid_{pT_{H}p}$ is a standard embedding.
\end{proof}

\begin{prop}\label{prop:std_embedding_free_product}
Let $H < \R_{+}$ be a countable, non-trivial subgroup. The following are standard embeddings:
\begin{enumerate}

\item[\normalfont{(i)}] The canonical inclusion $i: (A, \phi) \hookrightarrow  (A, \phi) * (B, \psi)$ where $A$ is either an interpolated free group factor with trace $\phi$, or an almost periodic free Araki--Woods factor with free quasi-free state $\phi$, and $B$ is either an interpolated free group factor with trace $\psi$, or an almost periodic free Araki--Woods factor with free quasi-free state $\psi$

\item[\normalfont{(ii)}] the canonical inclusion $j:  (T_{H}, \vphi_{H}) \hookrightarrow (T_{H}, \vphi_{H}) * (B, \psi)$ where $B$ is finite-dimensional or $\cB(\cH)$ for $\cH$ separable and infinite-dimensional and $\psi$ a faithful normal state.

\end{enumerate}
\end{prop}
\begin{proof}
(i) Let $R$ and $R'$ be two free copies of the hyperfinite II$_{1}$ factor and $S$, $S'$, $I$ and $I'$ disjoint sets. ($I$ (resp. $I'$) will be empty if $A$ (resp. $B$) is an interpolated free group factor.)  Suppose $\lambda_{i} \in (0, 1)$ for all $i \in I$, $\lambda_{i'} \in (0, 1)$ for all $i' \in I'$, and that $H = \langle \lambda_{i}: i \in I' \rangle$. Let $(x_{i})_{i \in I}$, $(x_{i'})_{i' \in I'}$, be free families of semicircular elements, free from each other and $R \cup R'$.  Let $(p_{s})_{s \in S}$, and $(q_{i})_{i \in I}$ be families of projections in $R$, and $(p_{s'})_{s' \in S'}$, and $(q_{i'})_{i' \in I'}$ be families of projections in $R'$. Finally, let $(y_{i})_{i \in I}$ and $(y_{i'})_{i' \in I'}$ be free families of generalized circular elements, free from each other, $(x_{i})_{i \in I}$, $(x_{i'})_{i' \in I'}$, and $R \cup R'$.  Assume that $y_{i}$ has parameter $\lambda_{i}$ for $i \in I \cup I'$.   We need to show that the inclusion
	\[
		W^{*}(R \cup (p_{s}x_{s}p_{s})_{s \in S} \cup (q_{i}y_{i}y_{i})_{i \in I}) \rightarrow W^{*}(R \cup R' \cup (p_{s}x_{s}p_{s})_{s \in S\cup S'} \cup (q_{i}y_{i}y_{i})_{i \in I\cup I'})
	\]
is standard.  Note by \cite[Corollary 3.6]{MR1256179} that there is a semicircular element $x \in W^{*}(R \cup R')$ which is free from $R$ and satisfies $W^{*}(R \cup \{x\}) = W^{*}(R \cup R')$.  Let $(u_{s'})_{s' \in S'}$ and $(v_{i'})_{i' \in I'}$ be families of unitaries in $W^{*}(R \cup \{x\})$ satisfying $\overline{p_{s'}} := u _{s'}^{*}p_{s'}u_{s'} \in R$ and $\overline{q_{i'}} := v_{i'}^{*}q_{i'}v_{i} \in R$. let $\overline{x_{s'}} = u_{s'}^{*}x_{s'}u_{s'}$ and $\overline{y_{i'}} = v_{i'}^{*}y_{i'}v_{i}$.  Therefore, the inclusion above can be realized as the canonical inclusion
	\begin{align*}
		W^{*}&(R \cup (p_{s}x_{s}p_{s})_{s \in S} \cup (q_{i}y_{i}y_{i})_{i \in I})\\
			\hookdownarrow &\\
		W^{*}&(R \cup R' \cup (p_{s}x_{s}p_{s})_{s \in S} \cup (\overline{p_{s'}}\cdot\overline{x_{s'}}\cdot\overline{p_{s'}})_{s' \in S'} \cup (q_{i}y_{i}y_{i})_{i \in I}, (\overline{q_{i'}}\cdot\overline{y_{i'}}\cdot\overline{y_{i'}})_{i' \in I'})
	\end{align*}
which is standard.

(ii) Choose $n$ large enough so that
	\[
		\underset{\frac{1}{n}, \cdots, \frac{1}{n}}{M_{n}(\C)} * (B, \psi)
	\] 
is a factor (necessarily $(L(\F_{t}),\tau)$ or $(T_{H'}, \vphi_{H'}))$.  Recall that $(T_H,\vphi_H)\cong (T_H,\vphi_H)\otimes \underset{\frac{1}{n}, \cdots, \frac{1}{n}}{M_{n}(\C)}$ and let $p$ be a minimal projection in $M_{n}(\C)$.  Note from Lemma~\ref{lem:Dykema} that
	\begin{align*}
		p(T_{H}, \vphi_{H}) * (B, \psi))p &=  p \left[ (T_H,\vphi_H)\otimes \underset{\frac{1}{n}, \cdots, \frac{1}{n}}{M_{n}(\C)}\right] * (B,\psi) p\\
										&= p\left[\underset{\frac{1}{n}, \cdots, \frac{1}{n}}{M_{n}(\C)} * (B, \psi)\right]p * p(T_H, \vphi_H)p.
	\end{align*}
It follows from (i) that $j\mid_{pT_{H}p}$ is standard, hence $j$ is standard from Proposition \ref{prop:std_embedding_compression}.
\end{proof}


\subsection{Free products of some hyperfinite and other von Neumann algebras}

We will use standard the embedding techniques above to extend Theorem \ref{thm:fdB(H)}.

\begin{thm}\label{thm:hyperfinite}
Let $(A, \phi)$ and $(B, \psi)$ be von Neumann algebras with at least one of $\phi$ or $\psi$ not a trace.  Assume further that $\dim(A), \dim(B) \geq 2$ and $A$ and $B$  are countable direct sums of algebras of the following types:
\begin{itemize}

\item Separable type I factors with faithful normal states;

\item Diffuse von Neumann algebras of the form $\displaystyle \overline{\bigotimes_{n=1}^{\infty} (F_{i}, \phi_{i}) }$ where each $F_{i}$ is finite-dimensional and the state is the tensor product of the $\phi_{i}$;

\item $(M, \gamma) \otimes (L(\F_{t}),\tau)$ with $M$ a separable type I factor, finite or infinite-dimensional;

\item $(N, \gamma') \otimes (T_{G}, \vphi_{G})$ with $N$ a separable type I factor, finite or infinite-dimensional, and $G$ a countable, non-trivial subgroup of $\R_{+}$.
\end{itemize}
\noindent
Let $(\cM, \vphi) = (A, \phi) * (B, \psi)$.  Then $(\cM, \vphi) \cong (T_{H}, \vphi_{H}) \oplus C$ where $H$ is the group generated by the point spectra of $\Delta_{\phi}$ and $\Delta_{\psi}$, and $C$ is finite-dimensional and is determined exactly as in Theorem \ref{thm:finitedim}.
\end{thm}
Note that the class of von Neumann algebras in the second bullet point contains all hyperfinite, diffuse, finite von Neumann algebras as well as the Powers factors $(R_{\lambda}, \phi_{\lambda})$ of type III$_{\lambda},$ and tensor products of Powers factors.  Furthermore observe that the class of von Neumann algebras in the last two bullet points contains the interpolated free group factors with traces and separable free Araki--Woods factors with free quasi-free states, but is larger in general as tensor products are allowed where $\gamma$ is not a trace and the point spectrum of $\Delta_{\gamma'}$ need not be a subset of $G$.

Before we prove the above theorem, we first use standard embeddings to upgrade Theorem~\ref{thm:fdB(H)} to handle infinite direct sums:

\begin{lem}\label{lem:inf_direct_sum_M_n}
Let $(A, \phi)$ and $(B, \psi)$ be von Neumann algebras with faithful states $\phi$ and $\psi$, respectively, of the following form (up to unitary conjugation):
	\[
		(A, \phi) = \bigoplus_{i = 1}^{\infty} \underset{\alpha_{i,1},\cdots, \alpha_{i, k_{i}}}{\overset{p_{i,1}, \cdots, p_{i, k_{i}}}{M_{k_{i}}(\C)}} \oplus \bigoplus_{i=1}^{\infty} \underset{\gamma_{i}}{\overset{\overline{p}_{i}}{(B(\cH_{i}), \phi_{i})}}\qquad\text{and}\qquad (B, \psi) = \bigoplus_{j = 1}^{\infty} \underset{\beta_{j,1},\cdots, \beta_{j, \ell_{j}}}{\overset{q_{j,1}, \cdots, q_{j, \ell_{j}}}{M_{\ell_{j}}(\C)}}  \oplus \bigoplus_{j=1}^{\infty} \underset{\delta_{j}}{\overset{\overline{q}_{j}}{(B(\cK_{j}), \psi_{j})}}.
	\]
Where the $\cH_{i}$ and $\cK_{j}$ are separable infinite-dimensional Hilbert spaces. We assume at least one of $\phi$ or $\psi$ is not a trace, that both $A$ and $B$ are at least two-dimensional, and allow either to be finite-dimensional by having the corresponding weights be zero. Let $G$ be the group generated by the point spectra of $\Delta_{\phi}$ and $\Delta_{\psi}$.  Then 
$$
(A, \phi) * (B, \psi) = (T_{G}, \vphi_{G}) \oplus C
$$
where $C$ is finite-dimensional (possibly zero) and is determined exactly as in Theorem~\ref{thm:finitedim}.
\end{lem}
\begin{proof}
Define $I_0\subset \bbN$ as the subset of $i\in \bbN$ such that either $k_i=1$ and there exists $j\in \bbN$ such that $\sum_{k=1}^n \frac{1-\alpha_{i,1}}{\beta_{j,k}} <1$, or there exists $j\in \bbN$ such that $\ell_j=1$ such that  $\sum_{k=1}^n \frac{1-\beta_{j,1}}{\alpha_{i,k}} <1$. Note that $I_0$ must be a finite set.  Define $J_0$ similarly. Up to relabeling, we may assume $I_0 = \{1,\ldots, N\}$ and $J_0=\{1,\ldots, M\}$ for some $N,M\in \bbN$, and then define
	\[
		A_0:= \bigoplus_{i=1}^N \underset{\alpha_{i,1},\cdots, \alpha_{i, k_{i}}}{\overset{p_{i,1}, \cdots, p_{i, k_{i}}}{M_{k_{i}}(\C)}} \qquad\text{and}\qquad B_0:= \bigoplus_{j = 1}^{M} \underset{\beta_{j,1},\cdots, \beta_{j, \ell_{j}}}{\overset{q_{j,1}, \cdots, q_{j, \ell_{j}}}{M_{\ell_{j}}(\C)}}  ,
	\]
and let $q_A$ and $q_B$ denote their respective identities. Let $K$ be sufficiently large so that if $p_A$ and $p_B$ are the respective identities of
	\[
		\bigoplus_{i=N+K+1}^\infty \underset{\alpha_{i,1},\cdots, \alpha_{i, k_{i}}}{\overset{p_{i,1}, \cdots, p_{i, k_{i}}}{M_{k_{i}}(\C)}} \oplus \bigoplus_{i = N+K+1}^{\infty} \underset{\gamma_{i}}{\overset{\overline{p}_{i}}{(B(\cH_{i}), \phi_{i})}} \qquad\text{and}\qquad \bigoplus_{j = M+K+1}^{\infty} \underset{\beta_{j,1},\cdots, \beta_{j, \ell_{j}}}{\overset{q_{j,1}, \cdots, q_{j, \ell_{j}}}{M_{\ell_{j}}(\C)}} \oplus \bigoplus_{j=M+K+1}^{\infty} \underset{\delta_{j}}{\overset{\overline{q}_{j}}{(B(\cK_{j}), \psi_{j})}},
	\]
then $\phi(p_A) < 1- \psi(q_B)$ and $\psi(p_B) < 1 - \phi(q_A)$. Define
	\begin{align*}
			(A_1,\phi):=&\bigoplus_{i=1}^{N+K} \underset{\alpha_{i,1},\cdots, \alpha_{i, k_{i}}}{\overset{p_{i,1}, \cdots, p_{i, k_{i}}}{M_{k_{i}}(\C)}} \oplus \bigoplus_{i = 1}^{N+K} \underset{\gamma_{i}}{\overset{\overline{p}_{i}}{(B(\cH_{i}), \phi_{i})}} \oplus \overset{p_A}{\bbC}\\
			(B_1,\psi):=&\bigoplus_{j = 1}^{M+K} \underset{\beta_{j,1},\cdots, \beta_{j, \ell_{j}}}{\overset{q_{j,1}, \cdots, q_{j, \ell_{j}}}{M_{\ell_{j}}(\C)}} \oplus \bigoplus_{j=1}^{M+K} \underset{\delta_{j}}{\overset{\overline{q}_{j}}{(B(\cK_{j}), \psi_{j})}} \oplus \overset{p_B}{\bbC}.
	\end{align*}
Increasing $K$ if necessary, we assume that at least one of the summands in either $A_1$ or $B_1$ is non-tracial. Let $H_{s,t}$ denote the subgroup of $\bbR^+$ generated by
	\begin{align*}
			&\left\{ \frac{\alpha_{i,j}}{\alpha_{i,k}} \colon 1\leq i \leq N+K+s-1,\ 1\leq j,k\leq k_i\right\},\\
			&\left\{ \frac{\beta_{j,i}}{\beta_{j,k}} \colon 1\leq j\leq M+K+t-1,\ 1\leq i,k\leq \ell_j \right\},
	\end{align*}
as well as the point spectra of $\Delta_{\phi_{i}}$ and $\Delta_{\psi_{j}}$ for $i \leq i \leq N+K+s-1$  and $j \leq M+K+t-1$. If we let $C$ be as in the statement of the lemma, then by Theorem~\ref{thm:finitedim} we have
	\[
		(A_1,\phi) * (B_1, \psi) \cong \overset{P}{(T_{H_{1,1}} ,\vphi_{H_{1,1}})} \oplus C.
	\]
Moreover, $p_A,p_B\leq P$ by our choice of $K$. Now, for $n,m\geq 2$ define
	\begin{align*}
		(A_n, \phi) :=&\bigoplus_{i=1}^{N+K} \underset{\alpha_{i,1},\cdots, \alpha_{i, k_{i}}}{\overset{p_{i,1}, \cdots, p_{i, k_{i}}}{M_{k_{i}}(\C)}} \oplus \bigoplus_{i = 1}^{N+K} \underset{\gamma_{i}}{\overset{\overline{p}_{i}}{(B(\cH_{i}), \phi_{i})}} \oplus \left(\bigoplus_{i=N+K+1}^{N+K+n-1} \overset{p_i}{\bbC} \oplus\overset{\overline{p}_{i}}{\C}\right) \oplus \overset{p_{A,n}}{\bbC}\\
		(B_n,\psi):=&\bigoplus_{j = 1}^{M+K} \underset{\beta_{j,1},\cdots, \beta_{j, \ell_{j}}}{\overset{q_{j,1}, \cdots, q_{j, \ell_{j}}}{M_{\ell_{j}}(\C)}}  \oplus \bigoplus_{j=1}^{M+K} \underset{\delta_{j}}{\overset{\overline{q}_{j}}{(B(\cK_{j}), \psi_{j})}} \oplus \left(\bigoplus_{j=M+K+1}^{M+K+n-1} \overset{q_j}{\bbC}\oplus \overset{\overline{q}_j}{\bbC}\right) \oplus \overset{p_{B,n}}{\bbC},
	\end{align*}
where
	\begin{align*}
		p_i &= p_{i,1}+\cdots +p_{i,k_i} \qquad\qquad\qquad\qquad \overline{p}_i \text{ is the identity of $\cB(\cH_{i})$}\\
		q_j &= q_{j,1}+\cdots + q_{j,\ell_j} \qquad\qquad\qquad\qquad \overline{q}_j \text{ is the identity of $\cB(\cK_{j})$}\\
		p_{A,n} &= p_A - (p_{N+K+1}+\cdots + p_{N+K+n-1}) - (\overline{p}_{N+K+1}+\cdots + \overline{p}_{N+K+n-1})\\
		p_{B,n} &= p_B - (q_{M+K+1}+\cdots + q_{M+K+n-1}) - (\overline{q}_{M+K+1}+\cdots + \overline{q}_{M+K+n-1}).
	\end{align*}
By Lemma~\ref{lem:Dykema} we have
	\[
		p_{A,1} \left[(A_2,\phi) * (B_1,\psi)\right] p_{A,1} \cong p_{A,1}\left[(A_1, \phi) * (B_1,\psi)\right] p_{A,1} * \left[ \overset{p_{N+K+1}}{\bbC} \oplus \overset{\overline{p}_{N+K+1}}{\bbC} \oplus \overset{p_{A,2}}{\bbC}\right]
	\]
Hence it follows from Propositions \ref{prop:std_embedding_compression} and \ref{prop:std_embedding_free_product} that the canonical inclusion of $(A_1,\phi)* (B_1,\psi)$ into $(A_2,\phi)*(B_1,\psi)$ is a standard embedding. Iterating this argument, we obtain that the canonical inclusions 
	\[
		(A_n,\phi)*(B_n,\psi) \hookrightarrow (A_{n+1},\phi)*(B_n,\psi) \hookrightarrow (A_{n+1},\phi)*(B_{n+1},\psi)
	\]
are standard embeddings. By Proposition~\ref{prop:std_embedding_inductive_limit}, we therefore have
	\[
		(A^{(1, 1)},\phi)* (B^{(1, 1)},\psi)\cong (T_{H_{1,1}},\vphi_{H_{1,1}})\oplus C,
	\]
where
	\begin{align*}
		(A^{(1, 1)},\phi):=&\bigoplus_{i=1}^{N+K} \underset{\alpha_{i,1},\cdots, \alpha_{i, k_{i}}}{\overset{p_{i,1}, \cdots, p_{i, k_{i}}}{M_{k_{i}}(\C)}} \oplus  \bigoplus_{i = 1}^{N+K} \underset{\gamma_{i}}{\overset{\overline{p}_{i}}{(B(\cH_{i}), \phi_{i})}} \oplus \bigoplus_{i=N+K+1}^{\infty} \overset{p_i}{\bbC} \oplus \bigoplus_{i=N+K+1}^{\infty} \overset{\overline{p}_i}{\bbC}\\
		(B^{(1, 1)},\psi):=&\bigoplus_{j = 1}^{M+K} \underset{\beta_{j,1},\cdots, \beta_{j, \ell_{j}}}{\overset{q_{j,1}, \cdots, q_{j, \ell_{j}}}{M_{\ell_{j}}(\C)}} \oplus  \bigoplus_{j=1}^{M+K} \underset{\delta_{j}}{\overset{\overline{q}_{j}}{(B(\cK_{j}), \psi_{j})}}  \oplus \bigoplus_{j=M+K+1}^{\infty} \overset{q_j}{\bbC}  \oplus \bigoplus_{j=M+K+1}^{\infty} \overset{\overline{q}_j}{\bbC}.
	\end{align*}

Next, for $n\geq 2$ define
	\begin{align*}
	(A^{(n, 1)},\phi):=&\bigoplus_{i=1}^{N+K+n-1} \underset{\alpha_{i,1},\cdots, \alpha_{i, k_{i}}}{\overset{p_{i,1}, \cdots, p_{i, k_{i}}}{M_{k_{i}}(\C)}} \oplus  \bigoplus_{i = 1}^{N+K+n - 1} \underset{\gamma_{i}}{\overset{\overline{p}_{i}}{(B(\cH_{i}), \phi_{i})}} \oplus \bigoplus_{i=N+K+n}^{\infty} \overset{p_i}{\bbC} \oplus \bigoplus_{i=N+K+n}^{\infty} \overset{\overline{p}_i}{\bbC}\\
	(A^{(n, 2)},\phi):=&\bigoplus_{i=1}^{N+K+n} \underset{\alpha_{i,1},\cdots, \alpha_{i, k_{i}}}{\overset{p_{i,1}, \cdots, p_{i, k_{i}}}{M_{k_{i}}(\C)}} \oplus  \bigoplus_{i = 1}^{N+K+n - 1} \underset{\gamma_{i}}{\overset{\overline{p}_{i}}{(B(\cH_{i}), \phi_{i})}} \oplus \bigoplus_{i=N+K+n+1}^{\infty} \overset{p_i}{\bbC} \oplus \bigoplus_{i=N+K+n}^{\infty} \overset{\overline{p}_i}{\bbC}\\
	(B^{(n, 1)},\psi):=&\bigoplus_{j = 1}^{M+K+n-1} \underset{\beta_{j,1},\cdots, \beta_{j, \ell_{j}}}{\overset{q_{j,1}, \cdots, q_{j, \ell_{j}}}{M_{\ell_{j}}(\C)}} \oplus  \bigoplus_{j=1}^{M+K+n-1} \underset{\delta_{j}}{\overset{\overline{q}_{j}}{(B(\cK_{j}), \psi_{j})}}  \oplus \bigoplus_{j=M+K+n}^{\infty} \overset{q_j}{\bbC}  \oplus \bigoplus_{j=M+K+n}^{\infty} \overset{\overline{q}_j}{\bbC}\\
	(B^{(n, 2)},\psi):=&\bigoplus_{j = 1}^{M+K+n} \underset{\beta_{j,1},\cdots, \beta_{j, \ell_{j}}}{\overset{q_{j,1}, \cdots, q_{j, \ell_{j}}}{M_{\ell_{j}}(\C)}} \oplus  \bigoplus_{j=1}^{M+K+n-1} \underset{\delta_{j}}{\overset{\overline{q}_{j}}{(B(\cK_{j}), \psi_{j})}}  \oplus \bigoplus_{j=M+K+n+1}^{\infty} \overset{q_j}{\bbC}  \oplus \bigoplus_{j=M+K+n}^{\infty} \overset{\overline{q}_j}{\bbC}.		
	\end{align*}
By Lemma~\ref{lem:Dykema} for $i=N+K+1$ we have
	\[
		p_i\left[ ( A^{(1,2)},\phi)*(B^{(1,1)},\psi)\right] p_i \cong p_i\left[ (A^{(1,1)},\phi)* (B^{(1,1)},\psi)\right] p_i * \overset{p_{i,1},\ldots, p_{i,k_i}}{\underset{\alpha_{i,1},\ldots, \alpha_{i,k_i}}{M_{k_i}(\bbC)}}.
	\]	
As above, we have that the canonical inclusion of $(A^{(1,1)},\phi)* (B^{(1,1)},\psi)$ into $( A^{(1,2)},\phi)*(B^{(1,1)},\psi)$ is a standard embedding. Iterating yields that the canonical inclusions
	\begin{align*}
		( A^{(n,1)},\phi)*(B^{(m, 1)},\psi) &\hookrightarrow ( A^{(n, 2)},\phi)*(B^{(m, 1)},\psi)\\
			 &\hookrightarrow ( A^{(n+1, 1)},\phi)*(B^{(m, 1)},\psi)\\
			 &\hookrightarrow ( A^{(n+1, 1)},\phi)*(B^{(m, 2)},\psi) \hookrightarrow ( A^{(n+1, 1)},\phi)*(B^{(m+1, 1)},\psi)
	\end{align*}
are standard embeddings. Appealing to Proposition~\ref{prop:std_embedding_inductive_limit} concludes the proof.
\end{proof}

\begin{proof}[Proof of Theorem~\ref{thm:hyperfinite}]
Write
	\[
		A = \bigoplus_{i= 1}^\infty A_{0,i} \oplus \bigoplus_{i=1}^\infty A_{1,i} \oplus \bigoplus_{i=1}^\infty A_{2,i} \qquad B=\bigoplus _{j= 1}^\infty B_{0,j} \oplus \bigoplus_{j=1}^\infty B_{1,j} \oplus \bigoplus_{j=1}^\infty B_{2,j}
	\]
where $\{A_{0,i}\}$ consists of all the type I factor direct summands of $A$, $\{A_{1,j}\}$ consists of all diffuse summands of the form $\displaystyle \overline{\bigotimes_{k=1}^{\infty} (F_{k}, \phi_{k}) }$ for $F_{k}$ finite-dimensional, and $\{A_{2, i}\}$ consists of the summands which are of the form
\begin{itemize}

\item $(M, \gamma) \otimes L(\F_{t})$ with $M$ a separable type I factor, finite or infinite-dimensional;

\item $(N, \gamma') \otimes (T_{G}, \vphi_{G})$ with $N$ a separable type I factor, finite or infinite-dimensional, and $G$ a nontrivial countable subgroup of $\R_{+}$.

\end{itemize}

The $\{B_{0, j}\}$, $\{B_{1, j}\}$, and $\{B_{2, j}\}$ are defined similarly. For the collections $\{A_{0,i}\}$ and $\{B_{0,j}\}$ define $I_0$ and $J_0$ as in Lemma~\ref{lem:inf_direct_sum_M_n}. Set
	\[
		A_1:= \bigoplus_{i\in 1}^\infty A_{0,i} \oplus \bigoplus_{i=1}^\infty A_{1,i}' \oplus \bigoplus_{i=1}^{\infty} M_{n_{i}}(\C)  \qquad B_1=\bigoplus _{j\in 1}^\infty B_{0,j} \oplus \bigoplus_{j=1}^\infty B_{1,j}' \oplus \bigoplus_{j=1}^{\infty} M_{m_{j}}(\C) 
	\]
where $A_{1,i}'$ is a finite-dimensional subalgebra of $A_{1,i}$ with dimension large enough so that its minimal projections have mass smaller than
	\[
		 1- \psi\left( \sum_{j=1}^\infty 1_{B_{0,j}}\right).
	\]
The state on each $M_{n_{i}}(\C)$ is tracial, the identity on $M_{n_{i}}(\C)$ is the identity on $A_{2, i}$, the inclusion of $M_{n_{i}}(\C)$ into $A_{2, i}$ is modular, and $n_{i}$ is large enough so that the summand $M_{n_{i}}(\C)$ is in the diffuse summand of $A_{1}*B_{1}$. Similar statements hold for the $B_{1,j}'$ and $M_{m_{j}}(\C)$. These conditions ensure that $A_1*B_1$ has the predicted finite-dimensional $C$. By Lemma~\ref{lem:inf_direct_sum_M_n} we know
	\[
		(A_1,\phi)* (B_1,\psi) \cong (T_{H_{1,1}}, \vphi_{H_{1,1}}) \oplus C,
	\]
where $H_{1,1}$ is the subgroup of $\bbR^+$ generated by the point spectra of $\Delta_{\phi\mid_{A_1}}$ and $\Delta_{\psi\mid_{B_1}}$. We then proceed to build back up to $A$ and $B$ by tensoring the summands $A_{1,i}'$ and $B_{1,j}'$ by appropriate finite-dimensional algebras, one at a time, as well as by tensoring von Neumann algebras of the form
\begin{itemize}

\item $(M, \gamma) \otimes L(\F_{t})$ with $M$ a separable type I factor, finite or infinite-dimensional (first tensor with $M$, then tensor with $L(\F_{t})$);

\item $(N, \gamma') \otimes (T_{G}, \vphi_{G})$ with $N$ a separable type I factor, finite or infinite-dimensional, and $G$ a nontrivial countable subgroup of $\R_{+}$  (first tensor with $N$, then tensor with $(T_{G}, \vphi_{G}$).

\end{itemize}
on each $M_{n_{i}}(\C)$  and each  $M_{m_{j}}(\C)$ and using
	\[
		\underset{\frac{1}{n}, \dots, \frac{1}{n}}{M_{n}(\C)}\otimes (L(\F(1 + n^{-2}(t-1))),\tau) \cong (L(\F_{t}),\tau) \qquad \text{and}\qquad \underset{\frac{1}{n}, \dots, \frac{1}{n}}{M_{n}(\C)} \otimes (T_{G}, \vphi_{G}) \cong (T_{G}, \vphi_{G})
	\]
for $G$ non-trivial. Lemma~\ref{lem:Dykema} together with Propositions \ref{prop:std_embedding_compression} and \ref{prop:std_embedding_free_product} ensure the canonical inclusions are standard embeddings on the orthogonal compliment of $C$. The result then follows from Proposition~\ref{prop:std_embedding_inductive_limit}.
\end{proof}

\begin{rem} It should be noted that, by Theorem 7.2 in \cite{HSV16}, there are  hyperfinite von Neumann algebras equipped with non-almost periodic states whose free product is not a free Araki--Woods factor of any kind.  Given Theorem \ref{thm:hyperfinite} above, it is natural to conjecture that the free product of injective von Neumann algebras is a free Araki--Woods factor plus a finite-dimensional von Neumann algebra if and only if both injective von Neumann algebras are equipped with almost periodic states.

\end{rem}

\bibliographystyle{amsalpha}
\bibliography{bibliography}

\providecommand{\bysame}{\leavevmode\hbox to3em{\hrulefill}\thinspace}
\providecommand{\MR}{\relax\ifhmode\unskip\space\fi MR }
\providecommand{\MRhref}[2]{%
  \href{http://www.ams.org/mathscinet-getitem?mr=#1}{#2}
}
\providecommand{\href}[2]{#2}
\begin{thebibliography}{Ued11b}

\bibitem[DR13]{MR3164718}
Kenneth~J. Dykema and Daniel Redelmeier, \emph{The amalgamated free product of
  hyperfinite von {N}eumann algebras over finite dimensional subalgebras},
  Houston J. Math. \textbf{39} (2013), no.~4, 1313--1331. \MR{3164718}

\bibitem[Dyk93]{MR1201693}
Kenneth~J. Dykema, \emph{Free products of hyperfinite von {N}eumann algebras
  and free dimension}, Duke Math. J. \textbf{69} (1993), no.~1, 97--119,
  \mathscinet{MR1201693}, \doi{10.1215/S0012-7094-93-06905-0}. \MR{1201693
  (93m:46071)}

\bibitem[Dyk94]{MR1256179}
\bysame, \emph{Interpolated free group factors}, Pacific J. Math. \textbf{163}
  (1994), no.~1, 123--135. \MR{1256179}

\bibitem[Dyk95]{MR1363079}
\bysame, \emph{Amalgamated free products of multi-matrix algebras and a
  construction of subfactors of a free group factor}, Amer. J. Math.
  \textbf{117} (1995), no.~6, 1555--1602, \mathscinet{MR1363079},
  \doi{10.2307/2375030}. \MR{1363079 (97b:46075)}

\bibitem[Dyk97]{Dyk97}
\bysame, \emph{Free products of finite-dimensional and other von {N}eumann
  algebras with respect to non-tracial states}, Free probability theory
  ({W}aterloo, {ON}, 1995), Fields Inst. Commun., vol.~12, Amer. Math. Soc.,
  Providence, RI, 1997, pp.~41--88. \MR{1426835}

\bibitem[HN20]{HNFreeGraph}
Michael Hartglass and Brent Nelson, \emph{Non-tracial free graph von {N}eumann
  algebras}, Comm. Math. Phys. (2020), no.~1, 1--40. \MR{4152265}

\bibitem[Hou07]{Hou07}
Cyril Houdayer, \emph{On some free products of von {N}eumann algebras which are
  free {A}raki-{W}oods factors}, Int. Math. Res. Not. IMRN (2007), no.~23, Art.
  ID rnm098, 21. \MR{2377217}

\bibitem[HSV19]{HSV16}
Cyril Houdayer, Dimitri Shlyakhtenko, and Stefaan Vaes, \emph{Classification of
  a family of non-almost-periodic free {A}raki-{W}oods factors}, J. Eur. Math.
  Soc. (JEMS) \textbf{21} (2019), no.~10, 3113--3142. \MR{3994101}

\bibitem[R\u95]{MR1372534}
Florin R\u{a}dulescu, \emph{A type {${\rm III}_\lambda$} factor with core
  isomorphic to the von {N}eumann algebra of a free group, tensor {$B(H)$}},
  Ast\'{e}risque (1995), no.~232, 203--209, Recent advances in operator
  algebras (Orl\'{e}ans, 1992). \MR{1372534}

\bibitem[Shl97]{Shl97}
Dimitri Shlyakhtenko, \emph{Free quasi-free states}, Pacific J. Math.
  \textbf{177} (1997), no.~2, 329--368. \MR{1444786}

\bibitem[Ued99]{MR1738186}
Yoshimichi Ueda, \emph{Amalgamated free product over {C}artan subalgebra},
  Pacific J. Math. \textbf{191} (1999), no.~2, 359--392. \MR{1738186}

\bibitem[Ued01]{MR1813523}
\bysame, \emph{Remarks on free products with respect to non-tracial states},
  Math. Scand. \textbf{88} (2001), no.~1, 111--125. \MR{1813523}

\bibitem[Ued11a]{MR2838053}
\bysame, \emph{Factoriality, type classification and fullness for free product
  von {N}eumann algebras}, Adv. Math. \textbf{228} (2011), no.~5, 2647--2671.
  \MR{2838053}

\bibitem[Ued11b]{MR2875863}
\bysame, \emph{On type {$\rm III_1$} factors arising as free products}, Math.
  Res. Lett. \textbf{18} (2011), no.~5, 909--920. \MR{2875863}

\bibitem[Ued13]{MR3092256}
\bysame, \emph{Some analysis of amalgamated free products of von {N}eumann
  algebras in the non-tracial setting}, J. Lond. Math. Soc. (2) \textbf{88}
  (2013), no.~1, 25--48. \MR{3092256}

\bibitem[Ued16]{MR3483468}
\bysame, \emph{Discrete cores of type {III} free product factors}, Amer. J.
  Math. \textbf{138} (2016), no.~2, 367--394. \MR{3483468}

\bibitem[VDN92]{MR1217253}
D.~V. Voiculescu, K.~J. Dykema, and A.~Nica, \emph{Free random variables}, CRM
  Monograph Series, vol.~1, American Mathematical Society, Providence, RI,
  1992, A noncommutative probability approach to free products with
  applications to random matrices, operator algebras and harmonic analysis on
  free groups. \MR{1217253 (94c:46133)}

\end{thebibliography}

\end{document}